\documentclass[11pt,oneside, reqno, a4paper]{amsart} 

\usepackage[english]{babel}
\usepackage[T1]{fontenc}
\usepackage{amsmath,mathdots}
\usepackage{amssymb}
\usepackage{float}
\usepackage{amsfonts, bm}
\usepackage{mathtools}
\usepackage[utf8]{inputenc}
\usepackage[mathcal]{eucal}
\usepackage{enumerate}
\usepackage{mathrsfs}
\usepackage{tikz}
\usetikzlibrary{arrows, matrix}
\usepackage{amsthm}
\usepackage{esint}
\usepackage{hyperref}
\usepackage[totalwidth=14cm,totalheight=20.5cm]{geometry}
\usepackage{epsfig,fancyhdr,color}
\usepackage{multicol} 
\usepackage{csquotes}
\usepackage{tikz-cd}
\usepackage[title]{appendix}

\newtheorem{cor}{Corollary}[section]

\newtheorem{theorem}[cor]{Theorem}

\newtheorem{prop}[cor]{Proposition}

\newtheorem{lemma}[cor]{Lemma}

\theoremstyle{definition}
\newtheorem{defi}[cor]{Definition}
\newtheorem{manualtheoreminner}{Theorem}
\newenvironment{manualtheorem}[1]{%
  \renewcommand\themanualtheoreminner{#1}%
  \manualtheoreminner
}{\endmanualtheoreminner}

\newtheorem{remark}[cor]{Remark}
\newtheorem{example}[cor]{Example}

\newtheorem{manualcorollaryinner}{Corollary}
\newenvironment{manualcorollary}[1]{%
  \renewcommand\themanualcorollaryinner{#1}%
  \manualcorollaryinner
}{\endmanualcorollaryinner}

\newcommand{\R}{\mathbb{R}}

\newcommand{\h}{\mathbb{H}}

\newcommand{\p}{\mathbf{P}}

\newcommand{\SL}{\mathrm{SL}}

\newcommand{\vl}{|}
\newcommand{\Ker}{\mathrm{Ker}}

\newcommand{\trace}{\mathrm{tr}}

\newcommand{\II}{\mathrm{I}\mathrm{I}}

\newcommand{\GL}{\mathrm{GL}}

\newcommand{\SO}{\mathrm{SO}}

\newcommand{\Hom}{\mathrm{Hom}}
\newcommand{\Span}{\mathrm{Span}}

\newcommand{\Id}{\mathrm{Id}}

\newcommand{\AdS}{\mathrm{AdS}}

\newcommand{\g}{\mathbf{g}}
\newcommand{\ome}{\boldsymbol{\omega}}

\DeclareMathAlphabet{\mathpzc}{OT1}{pzc}{m}{it}

\usepackage{mathtools}

\title[superconformal surfaces and cyclic Higgs bundles]{superconformal surfaces in para-complex hyperbolic space and cyclic Higgs bundles}

\allowdisplaybreaks

\begin{document}

\setcounter{secnumdepth}{3}
\setcounter{tocdepth}{2}

\title[]{Equiaffine immersions and pseudo-Riemannian space forms} 

\author[Nicholas Rungi]{Nicholas Rungi}
\address{NR: Department of Mathematics, University of Turin, Italy.} \email{nicholas.rungi@unito.it}

\date{\today}

\begin{abstract}
We introduce an explicit construction that produces immersions into the pseudosphere 
$\mathbb{S}^{n,n+1}$ and the pseudohyperbolic space $\mathbb{H}^{n+1,n}$ starting from equiaffine 
immersions in $\mathbb{R}^{n+1}$, and conversely. We describe how these immersions interact with a 
para-Sasaki metric defined on $\mathbb{H}^{n+1,n}$ via a principal $\mathbb{R}$-bundle structure over a para-Kähler manifold $\mathbb H_\tau^n$, called the para-complex hyperbolic space. In the case where the immersion in $\mathbb{R}^{n+1}$ is an $n$-dimensional hyperbolic affine 
sphere, we obtain spacelike maximal immersions in $\mathbb{H}^{n+1,n}$ that satisfy a 
transversality condition with respect to the principal $\mathbb{R}$-bundle structure. As a first application, given a strictly convex subset $\Omega \subset \mathbb{RP}^n$, we define a boundary set $\overline{\Lambda}_\Omega$ in the partial flag variety of lines and hyperplanes in $\R^{n+1}$, and prove the existence and uniqueness of a spacelike, Lagrangian, maximal $n$-submanifold in $\mathbb H^n_\tau$ with boundary $\overline{\Lambda}_\Omega$. We also discuss its implications in the case of $\mathbb H^{n+1,n}$. As a second application, we show that the Blaschke lift of the hyperbolic affine sphere, introduced 
by Labourie for $n=2$, into the symmetric space of $\mathrm{SL}(n+1,\mathbb{R})$ is a harmonic map.

\end{abstract}

\maketitle

\tableofcontents

\section{Introduction}
\noindent Given a smooth oriented $n$-dimensional manifold $M$, affine differential geometry (\cite{nomizu1994affine, li2015global}) is concerned with the study of immersions $f:M\to\mathbb R^{n+1}$ and with their properties that remain invariant under unimodular affine transformations. For this reason, it differs from Euclidean geometry in that $\mathbb R^{n+1}$ is endowed only with its affine structure, rather than with its flat metric. On the other hand, pseudo-Euclidean geometry of neutral signature deals with the study of the space $\mathbb R^{2n+2}$ endowed with a symmetric bilinear form of signature $(n+1,n+1)$, which we denote by $\mathbb R^{n+1,n+1}$. The group of transformations preserving both the orientation of the space and the non-degenerate bilinear form is isomorphic to $\mathrm{SO}(n+1,n+1)$. This Lie group acts transitively on two quadrics contained in $\mathbb R^{n+1,n+1}$, namely the \emph{pseudosphere} $\mathbb S^{n,n+1}$ and the \emph{pseudohyperbolic space} $\mathbb H^{n+1,n}$, obtained respectively as the sets of unit and negative-unit vectors. \\ \\
The aim of this work is to establish an explicit connection between affine differential geometry and pseudo-Euclidean geometry of neutral signature and to study several aspects of it. More precisely, starting from an equiaffine immersion $f:M\to\mathbb R^{n+1}$, we define two explicit maps $\sigma^+:M\to\mathbb S^{n,n+1}$ and $\sigma^-:M\to\mathbb H^{n+1,n}$, which are related to each other by an anti-isometry of the ambient space $\R^{n+1,n+1}$. Several aspects of this construction will be discussed: the interaction of the immersions $\sigma^\pm$ with a distinguished metric structure, the \emph{para-Sasaki metric}, defined on the pseudosphere and the pseudohyperbolic space; how to reverse the construction whenever $M$ is also simply connected; and finally the correspondence between properties arising on the affine and pseudo-Euclidean sides. For instance, the map $\sigma^-$ is maximal in $\mathbb H^{n+1,n}$ if and only if the conormal map of $f$ is a proper affine sphere. As a first application, we study the existence (and uniqueness) problem for $n$-dimensional spacelike, maximal and Lagrangian submanifolds in the para-complex hyperbolic space $\mathbb H^n_\tau$ with prescribed boundary at infinity. This space arises as the base of a principal $\R$-bundle $\pi_-: \mathbb H^{n+1,n} \to \mathbb H^n_\tau$, thereby enabling the study of lifts of these submanifolds to the pseudohyperbolic space, showing how their boundary relates to natural subspaces that can be defined in $\partial_\infty \mathbb H^{n+1,n}$. Another interesting feature of this construction is that the dimension of the maximal submanifold is not equal to the dimension of the positive signature part of \(\mathbb{H}^{n+1,n}\). The immersions we obtain differ from those studied recently in the literature (\cite{bonsante2010maximal, andersson2012cosmological, CTT, labourie2024plateau, seppi2023complete}) concerning \(p\)-dimensional spacelike maximal submanifolds in \(\mathbb{H}^{p,q}\), or constant mean curvature hypersurfaces in $\mathbb{H}^{p,1}$ (\cite{trebeschi2024constant}). It would be interesting to investigate the equivariant counterpart of this construction 
and to understand whether it is related to the Anosov property for representations 
of hyperbolic groups into \(\mathrm{SO}(n+1,n+1)\), as studied for instance in \cite{CTT,beyrer2023mathbb,seppi2023complete}.
As a second application, in the context of hyperbolic affine spheres, we generalize the \emph{Blaschke lift} introduced by Labourie for $n=2$ (\cite{Labourie_cubic}), defining a map
$
\mathcal G_f:M\to\mathbb X_{n+1},
$
into the symmetric space of $\SL(n+1,\R)$, which turns out to be harmonic for all $n\ge2$.

\subsection{The correspondence}\label{sec:the_correspondence_introduction}
We now turn to the explanation of the correspondence between the immersions arising in seemingly different contexts. As mentioned earlier, we are interested in affine immersions $f : M \to \mathbb{R}^{n+1}$, called \emph{equiaffine}, where $M$ is a smooth $n$–dimensional manifold, endowed with a vector field $\xi : M \to \mathbb{R}^{n+1}$ that is pointwise transverse to $f_*(T_pM)$ and satisfies the following structure equations:
$$
\begin{cases}
D_X f_*Y = f_*\big(\nabla_X Y\big) + h(X,Y)\xi \\
D_X \xi = -f_*\big(S(X)\big)
\end{cases}
$$ where $D$ denotes the flat connection on $\mathbb{R}^{n+1}$, while $\nabla$, $h$, and $S$ are respectively a torsion–free connection on $TM$, a symmetric bilinear form on $TM$, and an endomorphism of $TM$. In the following we will denote the equiaffine immersion by the pair $(f,\xi)$, and we will say that it is non-degenerate if and only if the bilinear form $h$ is non-degenerate. Whenever such an immersion exists, one can consider its dual $\nu : M \to (\mathbb{R}^{n+1})^*$, defined by requiring that $\nu(\xi) = 1$ and that its kernel coincides pointwise with $f_*(T_pM)$. Similarly, one can derive structure equations for $\nu$ and show that it is always a centroaffine immersion, that is, a special case where the transversal vector field can be chosen to be a multiple of the position vector. The affine connection and the symmetric bilinear form on $TM$ induced by $\nu$ will be denoted by $\bar\nabla$ and $\bar h$.
 This construction allows us to define two maps $\sigma^\pm : M \to \mathbb{R}^{n+1} \oplus (\mathbb{R}^{n+1})^*$ by
$
\sigma^\pm(p) := (\pm \xi_p, \nu_p).
$
Introducing the bilinear form
$$
\hat\g\big((v,\varphi),(w,\psi)\big) := \tfrac{1}{2}\big(\varphi(w) + \psi(v)\big)
$$
on the space $V:=\mathbb{R}^{n+1} \oplus (\mathbb{R}^{n+1})^*$, we obtain an isometric copy of the standard model pseudo–Euclidean space $\mathbb{R}^{n+1,n+1}$ (after identifying $\mathbb{R}^{n+1}$ with its dual via the standard inner product). By the very definition of the conormal map $\nu$, it follows that $\sigma^-$ (respectively $\sigma^+$) is contained in the quadric consisting of vectors of norm $-1$ (respectively $+1$) in $\mathbb{R}^{n+1,n+1}$, namely in the so–called pseudohyperbolic space $\mathbb{H}^{n+1,n}$ of signature $(n+1,n)$ (respectively in the pseudosphere $\mathbb{S}^{n,n+1}$ of signature $(n,n+1)$). In what follows, we shall discuss only the results concerning one of the two maps $\sigma^\pm$, since one is determined from the other via the anti–isometry of $V$ (with respect to $\hat\g$) that changes the sign of the component along the first copy of $\mathbb{R}^{n+1}$.\\ \\ As already mentioned, the Lie group $\mathrm{SO}_0(V,\hat\g) \cong \mathrm{SO}_0(n+1,n+1)$ acts transitively on $\mathbb{H}^{n+1,n}$. In particular, there exists a copy of $\mathbb{R} \cong \mathrm{SO}_0(1,1) < \mathrm{SO}_0(n+1,n+1)$ acting naturally on such space freely and properly, giving rise to a projection map
$
\pi_- : \mathbb{H}^{n+1,n} \to \mathbb{H}^{n+1,n}/\mathbb{R},
$
which inherits the structure of a principal $\mathbb{R}$–bundle. The quotient manifold $\mathbb{H}^{n+1,n}/\mathbb{R}$ can in fact be identified with a para–Kähler manifold of real dimension $2n$, known as the \emph{para–complex hyperbolic space} (Proposition \ref{prop:diffeo_hyperbolic_incidence}) and denoted by $\mathbb{H}_\tau^n$ (\cite[\S 8.4]{trettel2019families},\cite{RT_bicomplex}). The para–Kähler metric on the base can then be suitably pulled back to $\mathbb{H}^{n+1,n}$ (Appenidx \ref{appendix_B}), endowing this space with a \emph{para–Sasaki metric}. In particular, the para-complex structure $\p$ of $\mathbb H_\tau^n$ induces an endomorphism $\phi$ on the tangent bundle of $\mathbb H^{n+1,n}$, via the map $\pi_-$.
\begin{manualtheorem}A[{Theorem \ref{prop:immersion_sigma} and \ref{prop:immersion_in_H_tau}}]\label{thm:A}
Given a non-degenerate equiaffine immersion $(f,\xi)$ with $\det(S)\neq 0$, the map $\sigma^-=(-\xi,\nu):M\to\mathbb H^{n,n+1}$ is a horizontal and $\phi$-anti-invariant immersion whose induced metric coincides with the affine metric of the dual immersion $\nu$. Moreover, it projects to a non-degenerate Lagrangian immersion $\bar\sigma=\pi_-\circ\sigma^-$ into $\h_\tau^n$.
\end{manualtheorem} 
\noindent The \emph{horizontal} and $\phi$\emph{-anti-invariant} properties for $\sigma^-$ are defined using the principal $\mathbb{R}$-bundle structure of the pseudohyperbolic space, and therefore the para-Sasaki metric. Moreover, since $\sigma^-$ is related to $\sigma^+$ by an anti–isometry of the ambient space, the same result holds for $\sigma^+$, except that the induced metric changes by a sign. Using the structure equations of the immersions $(f,\xi)$ and $\nu$, one can explicitly compute the mean curvature $\mathbf H_{\sigma^-}$ of $\sigma^-$ in terms of an affine invariant of the conormal map, namely the Pick tensor $\bar C(\cdot,\cdot,\cdot):= \big(\bar\nabla_{\cdot}\bar h\big)(\cdot,\cdot)$. More precisely, one shows that $\mathbf H_{\sigma^-}$ coincides, up to a non-zero multiplicative constant, with the trace of $\bar C$ taken with respect to $\bar h$.
\begin{manualcorollary}B[{Corollary \ref{cor:sigma_maximal_iff_proper_affine_sphere}}]\label{thm:B}
Given $(f,\xi)$ as above, the following statements are equivalent: \begin{enumerate}
    \item[(i)] $\sigma^-:M\to\mathbb H^{n+1,n}$ is maximal, i.e. $\mathbf H_{\sigma^-}\equiv 0$; \item[(ii)] $\bar\sigma=\pi_-\circ\sigma^-:M\to\h_\tau^n$ is maximal; \item[(iii)] the dual immersion $\nu:M\to(\R^{n+1})^*$ is a proper affine sphere. \end{enumerate}
\end{manualcorollary}
\noindent As for the inverse problem, starting from a non-degenerate and Lagrangian immersion $\bar\sigma:M\to\mathbb H^n_\tau$, with $M$ simply connected, we aim to find an equiaffine immersion $(f,\xi)$ such that $\sigma^- = (-\xi,\nu):M\to\mathbb H^{n+1,n}$ is horizontal, $\phi$–anti-invariant, satisfies $\bar\sigma = \pi_-\circ\sigma^-$, and where $\nu$ is precisely the conormal map to $f$. The results in \cite{hildebrand2011half, hildebrand2011cross} and \cite{dillen1990conjugate} allow one to perform this inversion of the construction with an implicit approach. In our case, we adopt an explicit and geometric method. Indeed, given a non-degenerate and Lagrangian immersion $\bar\sigma:M\to\mathbb H^n_\tau$, with $M$ simply connected, one can always find a lift $\sigma^-=(-\xi,\nu)$ to $\mathbb H^{n+1,n}$. Using the principal $\mathbb R$-bundle structure, we know that any other lift can be written as $\sigma^-_\mu = (-f^\mu,\nu^\mu)$, where $\mu$ is a smooth function, and $f^\mu=e^\mu\xi$ (resp. $\nu^\mu=e^{-\mu}\nu$) is a centroaffine immersion in $\R^{n+1}$ (resp. in $(\R^{n+1})^*$).
\begin{manualtheorem}C[{Theorem \ref{thm:inverse_problem_lift}}]\label{thm:C}
Let $M$ be simply-connected and $\bar\sigma:M\to\mathbb H^n_\tau$ be a non-degenerate and Lagrangian immersion, then \begin{itemize}
    \item[(i)] each lift $\sigma_\mu^- : M \to \mathbb{H}^{n+1,n}$ is $\phi$–anti-invariant, and the induced metric coincides with that induced by $\bar{\sigma}$;
\item[(ii)] there exists a unique function $\hat{\mu}$, up to an additive constant, such that the lift $\sigma_{\hat{\mu}}^-$ is horizontal;
\item[(iii)] the centroaffine immersions $f^{\hat{\mu}}$ and $\nu^{\hat{\mu}}$ induced by $\sigma_{\hat{\mu}}^-$ are dual to each other and unique up to homothety.
\end{itemize}
\end{manualtheorem}
\noindent By performing computations similar to those carried out previously, we obtain the following corollary, which completes the inverse construction.
\begin{manualcorollary}D[{Corollary \ref{cor:inverse_sigma_maximal_iff_proper_affine_sphere}}]\label{cor:D}
In the setting above, the following are equivalent: \begin{itemize}
    \item[(i)] $\bar\sigma$ is maximal;
    \item[(ii)] the horizontal lift $\sigma^-_{\hat\mu}:M\to\mathbb H^{n+1,n}$ is maximal;
    \item[(iii)] the affine immersion $f^{\hat\mu}$ and its dual $\nu^{\hat\mu}$ are proper affine spheres.
\end{itemize}
\end{manualcorollary}
\noindent In Section \ref{sec:examples}, we provide a series of explicit examples showing that our correspondence recovers several constructions already known in the literature, while at the same time generalizing them. \\ \\ In the two papers \cite{hildebrand2011cross, hildebrand2011half}, Hildebrand studied a relation between Lagrangian immersions into an isometric model of $\mathbb H_\tau^n$ and affine immersions in \(\mathbb R^{n+1}\), using an approach different from our explicit construction. See also \cite{dorfmeister2024half}. \\ \\ In the work \cite{LucVrancken2002} Vrancken studies an implicit correspondence between immersions in pseudo-Riemannian space forms and centroaffine immersions in $\mathbb{R}^{n+1}$ which, again, differs from our explicit approach using the conormal map and exhibiting the interaction of the immersion with the para-Sasaki structure. \\ \\
A similar correspondence, though in a different context, was investigated by El Emam–Seppi (\cite{el2022gauss}). Starting from hypersurfaces in hyperbolic space $\mathbb H^{n+1}$, they constructed a Gauss map into the space $\mathcal G(\mathbb H^{n+1})$ of oriented geodesics in $\mathbb H^{n+1}$, which carries a para-Kähler metric. Among the various problems they examined was precisely that of determining suitable conditions under which this Gauss map lifts to an immersion into the unit tangent bundle $T^1\mathbb H^{n+1}$, which is a principal $\mathbb R$-bundle over $\mathcal G(\mathbb H^{n+1})$.

\subsection{The boundary at infinity}
The projective model of the pseudohyperbolic space, defined as
$\hat{\mathbb H}^{n+1,n} := \mathbb H^{n+1,n} /\sim$, where $(x,y) \sim (-x,-y)$,
is contained in $\mathbb R\mathbb P^{2n+1}$ and admits a natural \emph{boundary at infinity}, as in the Riemannian setting. Using the bilinear form $\hat\g$ introduced in Section \ref{sec:the_correspondence_introduction} and defined on $V:=\R^{n+1}\oplus(\R^{n+1})^*$, the boundary at infinity of the pseudohyperbolic space is given by
$
\partial_\infty \hat{\mathbb H}^{n+1,n} = \big\{ [v,\varphi] \in \mathbb P(V) \mid \varphi(v)=0 \big\},
$
that is, it consists of the projective classes of pairs \((v,\varphi)\), where 
\(v \in \mathbb R^{n+1}\) and \(\varphi\) is a linear functional whose associated 
hyperplane contains \(v\). It is sometimes referred to as the \emph{Einstein Universe} of signature $(n,n)$. The same construction can be applied to the para-complex hyperbolic space. We denote by $\hat{\mathbb H}^n_\tau$ its projective model, so that its boundary at infinity is described as
$
\partial_\infty \hat{\mathbb{H}}_\tau^n = \{ ([v],[\varphi]) \in \mathbb{RP}^n \times (\mathbb{RP}^n)^* \mid \varphi(v) = 0 \},
$
that is, it consists of all pairs of lines and hyperplanes in $\R^{n+1}$ with the line contained in the hyperplane. This space is exactly the partial flag variety parametrizing lines and hyperplanes in $\R^{n+1}$, which we denote by $\mathcal{F}_{1,n}(\R^{n+1})$.
 \\ \\ 
The first application of the results from the previous section concerns an existence (and uniqueness) problem for maximal spacelike Lagrangian \(n\)-submanifolds in 
\(\hat{\mathbb H}^n_\tau\) with prescribed boundary data. More precisely, given a strictly convex subset \(\Omega \subset \R\mathbb P^n\), its 
lift to \(\mathbb R^{n+1}\setminus\{0\}\) determines the full cone
\(
\mathcal C(\Omega)=\mathcal C_+(\Omega)\sqcup \mathcal C_-(\Omega),
\)
whose two components \(\mathcal C_\pm(\Omega)\) are sharp convex cones, meaning 
that their interiors contain no lines. Starting from such \(\Omega \subset \R\mathbb P^n\), we define the following subset of 
the Einstein universe:
\[
\Lambda_\Omega := \{[v,\varphi_v] \in \mathbb P(V) \mid 
v \in \partial\mathcal C(\Omega),\ [\varphi_v] \text{ is the unique supporting hyperplane at } [v]\},
\]
where $[v],[\varphi_v]$ represent the projective classes in $\mathbb{RP}^n,(\mathbb{RP}^n)^*$, respectively. We will refer to subsets of the form 
\(\Lambda_\Omega \subset \partial_\infty \hat{\mathbb H}^{\,n+1,n}\) as 
\emph{hyperplane boundary sets}. The principal $\R$-bundle structure $\hat\pi_-:\hat{\mathbb H}^{n+1,n}\to\hat{\mathbb H}^n_\tau$ induces a projection at the level of boundary at infinity $p:\mathrm{Ein}_0^{n,n}\subset\partial_\infty\hat{\mathbb H}^{n+1,n}\to\mathcal{F}_{1,n}(\R^{n+1})$ where $\mathrm{Ein}^{n,n}_0$ is the subset formed by points $[v,\varphi]\in\mathbb{P}(V)$ with $\varphi(v)=0$, $v\neq 0$ and $\varphi\neq 0$. In particular, we can consider the projection of the hyperplane boundary set $\overline{\Lambda}_\Omega:=p(\Lambda_\Omega)$ onto the partial flag variety.
\begin{manualtheorem}E[{Theorem \ref{cor:plateau_problem_paracomplex_hyperbolic}}]\label{thm:E_introduction}
Given a subset $\overline{\Lambda}_\Omega\subset\mathcal{F}_{1,n}(\R^{n+1})$ obtained from a hyperplane boundary set $\Lambda_\Omega$, there exists a $n$-dimensional spacelike, maximal, Lagrangian submanifold in $\hat{\mathbb H}^n_\tau$ whose boundary at infinity is equal to $\overline{\Lambda}_\Omega$. Moreover, the induced metric is complete and the $n$-submanifold is unique.
\end{manualtheorem}
\noindent In brief, given a strictly convex set \(\Omega \subset \mathbb{R}\mathbb{P}^n\), 
the proof relies on the existence of a hyperbolic affine sphere 
\(f : M \to \mathbb{R}^{n+1}\) asymptotic to the boundary of the sharp convex cone 
\(\mathcal{C}_+(\Omega)\) (\cite{cheng1977regularity, cheng1986complete, gigena1978integral, gigena1981conjecture, li1990calabi, li1992calabi}). This immersion carries a positive definite 
affine metric, the \emph{Blaschke metric}, which is known to be complete; completeness is in fact equivalent to the immersion \(f\) being proper (\cite{cheng1977regularity, cheng1986complete}). Using the construction described in Section \ref{sec:the_correspondence_introduction}, we can then define an 
\(n\)-dimensional submanifold in \(\hat{\mathbb{H}}^n_\tau\) with the desired 
properties, and verify that its boundary at infinity coincides with 
\(\overline{\Lambda}_\Omega\), which turns out to be homeomorphic to $\partial\Omega$. 
\\ \\ 
Concerning the pseudohyperbolic case, given a strictly convex subset $\Omega\subset\R\mathbb P^n$, we can define a $1$-parameter family of immersions $\{\sigma_t^-\}_{t \in \R}$ with values in $\hat{\mathbb H}^{n+1,n}$. For each fixed $t \in \R$, we compute the boundary at infinity of $\sigma_t^-$ and show that it forms a subset $\Lambda_\Omega^t \subset \Lambda_\Omega$ (Lemma \ref{lem:buondary_pseudosphere}). These subsets, as $t \in \R$ varies, are all distinct and generate a foliation of the hyperplane boundary set $\Lambda_\Omega$. Finally, an interesting construction also arises for the immersion
$
\varsigma^-(p,t) := (-e^t f(p), e^{-t} \nu_p),
$
obtained by post-composing $\sigma^-$ with the $\R$-action on $\hat{\mathbb H}^{n+1,n}$. This gives rise to an $(n+1)$-dimensional submanifold in $\hat{\mathbb H}^{n+1,n}$ whose boundary at infinity can be reconstructed using $\Lambda_\Omega$ and the convex set itself (Proposition \ref{prop:boundary_immersion_into_paracomplex_hyperbolic}).

\subsection{Blaschke lift of hyperbolic affine spheres}
As in the previous section, we now focus on the case where the convex affine immersion $f:M\to\mathbb R^{n+1}$ is a hyperbolic affine sphere endowed with its Riemannian Blaschke metric $h$. In the case $n=2$, Labourie (\cite{Labourie_cubic}) defined, starting from $f$, a map $\mathcal G_f:M\to\mathbb X_{3}$ into the symmetric space of $\SL(3,\mathbb R)$, called the \emph{Blaschke lift}. He proved that this map is harmonic and conformal — hence minimal — where the Blaschke metric is used on the affine sphere, and the Riemannian metric induced by the Killing form on the Lie algebra is used on the symmetric space. In fact, this construction can be generalized to any dimension. Given a hyperbolic affine sphere $f:M\to\mathbb R^{n+1}$, after identifying
$
\mathbb X_{n+1}=\SL(n+1,\mathbb R)/\SO(n+1)
$
with the space of scalar products on $\mathbb R^{n+1}$ with unit determinant, we define
$
\mathcal G_f(p):=h_p(\cdot,\cdot)\oplus\lambda(p)\cdot f(p),
$
where $p\in M$ and $\lambda(p)=(\det h_p)^{-1}$.
From the definition and from the properties of hyperbolic affine spheres, it is clear that the element $\mathcal G_f(p)$ defines a point in $\mathbb X_{n+1}$, indeed for any $p\in M$ there is a decomposition $\R^{n+1}=T_pM\oplus\R\cdot f(p)$. We can therefore generalize Labourie’s result to any dimension:
\begin{manualtheorem}F[{Theorem \ref{thm:harmonic_transvers_map_lift} and Corollary \ref{cor:gauss_lift_harmonic}}]\label{thm:F}
Given a hyperbolic affine sphere $f:M\to\R^{n+1}$ its Blaschke lift $\mathcal G_f:M\to\mathbb X_{n+1}$ is harmonic.
\end{manualtheorem}
\noindent The main idea behind the proof of this result is to express the Blaschke lift $\mathcal G_f$ as the composition of two other maps that are simpler to study from this point of view.
More precisely, if we consider the induced map $\sigma^+=(f,\nu):M\to\mathbb S^{n,n+1}$, we can construct a lift $\tilde{\mathcal G}_f$ into the pseudo-Riemannian space
$
\mathcal X := \{ (u,P,q) \mid u\in\mathbb R^{n+1}, \ P\subset\mathbb R^{n+1} \text{ hyperplane}, \ u\notin P, \ q \ \text{scalar product on} \ P \}.
$
Using the homogeneous description of these spaces in terms of Lie groups, one can show the existence of a pseudo-Riemannian submersion
$
\pi_{n+1}:\mathcal X \to \mathbb X_{n+1}
$
such that $\mathcal G_f = \pi_{n+1}\circ \tilde{\mathcal G}_f$. At this point, a direct computation shows that $\tilde{\mathcal G}_f$ is harmonic and that the image of its differential is orthogonal to $\ker(\mathrm d\pi_{n+1})$ (Theorem \ref{thm:harmonic_transvers_map_lift}). This allows one to conclude that $\mathcal G_f$ is harmonic in $\mathbb X_{n+1}$. It is worth pointing out that, in order to prove these results, the assumption that $f$ is a hyperbolic affine sphere plays a crucial role, and it seems unlikely that this hypothesis could be weakened. The scheme of the maps involved is summarized in the diagram below.
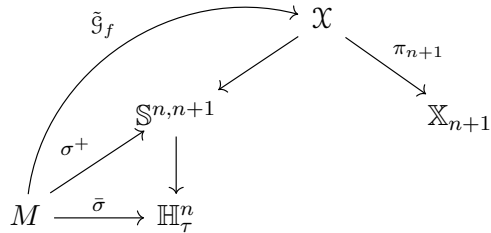
\begin{equation*}\begin{tikzcd}
                                                                                                              &                                      & \mathcal X \arrow[ld, ] \arrow[rd, "\pi_{n+1}"] &                 \\
                                                                                                              & {\mathbb S^{n,n+1}} \arrow[d,] &                                                                   & \mathbb X_{n+1} \\
M \arrow[r, "\bar\sigma"] \arrow[ru, "\sigma^+", shift left] \arrow[rruu, "\tilde{\mathcal G}_f", bend left=49] & \mathbb H_\tau^n                     &                                                                   &                
\end{tikzcd}\end{equation*}
\noindent As a consequence, we obtain a map 
\(
\mathcal{G}_{\sigma^+} := \Phi \circ \mathcal{G}_f : M \to \mathbb{Y}_{n+1},
\)
which is also harmonic (Theorem \ref{thm:harmonic_map_SO(n+1,n+1)}), where $\mathbb{Y}_{n+1}$ is the symmetric space of 
$\SO_0(n+1,n+1)$ and 
\(
\Phi : \mathbb{X}_{n+1} \hookrightarrow \mathbb{Y}_{n+1}
\)
is the explicit inclusion induced by 
\(
\iota : \SL(n+1,\mathbb{R}) \hookrightarrow \SO_0(n+1,n+1).
\) (Section \ref{sec:Lie_groups_inclusion}).
It would certainly be interesting to understand whether, in the setting of complete 
$p$-dimensional maximal spacelike submanifolds in $\mathbb{H}^{p,q}$ (\cite{seppi2023complete}), one can construct a 
map from the maximal submanifold into the symmetric space of $\SO_0(p,q)$ that turns out to 
be harmonic (in the case $p=2$, $q\ge 1$, this is proved in \cite{CTT}).

\subsection*{Structure of the paper}
We now outline the organization of the paper, in order to assist the reader in locating the various arguments. Section \ref{sec:2} is entirely devoted to preliminaries, and includes a review of affine differential geometry, convex real projective geometry, pseudo-Euclidean geometry in neutral signature and its quadrics as pseudo-Riemannian submanifolds, as well as the principal $\R$-bundle structure of $\mathbb H^{n+1,n}$ and $\mathbb S^{n,n+1}$ over a para-K\"ahler base manifold. 
Section \ref{sec:main_construction} is devoted to the proof of the main construction of the paper. The first part concerns the construction of immersions into pseudo-Riemannian quadrics arising from equiaffine immersions, thereby establishing Theorem \ref{thm:A} and Theorem \ref{thm:B} (Section \ref{sec:immersions_pseudo_spheres} and \ref{sec:immersion_para-complex_hyperbolic}). The second part addresses the inverse problem, proving Theorem \ref{thm:C} and Corollary \ref{cor:D} (Section \ref{sec:inverse_problem}). 
The first application, concerning the existence of \(n\)-dimensional maximal spacelike submanifolds in \(\mathbb H^{n+1,n}\) with prescribed boundary at infinity, is discussed in Section \ref{sec:boundary_problem}. As for the construction of harmonic maps into the symmetric spaces of \(\SL(n+1,\mathbb R)\) and \(\SO_0(n+1,n+1)\), we refer the reader to Section \ref{sec:5}. Finally, In Appendix \ref{appendix_A} and Appendix \ref{appendix_B}, we introduce a simple yet general criterion for inducing contact forms on hypersurfaces in symplectic manifolds, and para-Sasaki metrics on principal $G$-bundles over para-Kähler manifolds, respectively.

\subsection*{Acknowledgements}The author is grateful to Andrea Seppi for his constant support during the preparation of this work, for carefully reading an initial draft of the paper, and for providing comments and suggestions that improved the exposition. The author also thanks Jérémy Toulisse for helpful conversations on harmonic maps during his visit to Nice in December 2024. The author thanks Enrico Trebeschi for his interest in this work and for many discussions on the topic. The material presented in this article has benefited from conversations with Colin Davalo and Alex Moriani. N.R. is funded by the European Union (ERC, GENERATE, 101124349). Views and opinions expressed are however those of the author(s) only and do not necessarily reflect those of the European Union or the European Research Council Executive Agency. Neither the European Union nor the granting authority can be held responsible for them.
\section{Preliminaries}\label{sec:2}
\subsection{Affine differential geometry}\label{sec:affine_differential_geometry}
The aim of this section is to introduce the main definitions and objects of affine differential geometry that will be used in the sequel. The exposition will largely follow \cite{nomizu1994affine}, although certain results will be made explicit and further developed for the purpose of the main construction, which will be explained in Section \ref{sec:main_construction}. \\ \\ We are interested in the study of hypersurface immersions $f : M \to \mathbb{R}^{n+1}$ equipped with a transverse vector field $\xi$, that is, such that for every $p \in M$, we have a decomposition
$$
\mathbb{R}^{n+1} = f_*\big(T_pM\big) \oplus\Span\{\xi_p\}.
$$
For any $X,Y\in\Gamma(TM)$, the immersion satisfies the following structure equations:
\begin{equation}\begin{cases}\label{eq:structure_equations_f}
    D_Xf_*Y=f_*\big(\nabla_XY\big)+h(X,Y)\xi \\ D_X\xi=-f_*\big(S(X)\big)+\tau(X)\xi
\end{cases}\end{equation} where $D$ is the flat connection on $\R^{n+1}$, $\nabla$ is the induced \emph{affine connection}, $h$ is a symmetric bi-linear form on $TM$ called the \emph{affine metric}, $S$ is an endomorphism of $TM$ called the \emph{affine shape operator} and $\tau$ is a $1$-form on $M$. \begin{example}\label{ex:centro_affine}
Let $f:M\to\R^{n+1}$ be a transverse immersion and suppose that the transverse vector field $\xi$ is proportional to the position vector, namely for all $p\in M$ we have $\xi_p=f(p)$. In this case $D_X\xi=f_*(X)$, so that $\tau=0$ and $S=-\Id$ (\cite[Chapter II, Example 2.2]{nomizu1994affine}). Such immersions are referred to as \emph{centroaffine immersions} and will be fundamental in the sequel.
\end{example} \noindent In the general case of a transverse immersion $f:M\to\R^{n+1}$, given a volume form $\tilde\omega$ on $\R^{n+1}$ we can define a volume form $\theta$ on $M$ by the formula: $$\theta(X_1,\dots,X_n):=\tilde\omega\big(f_*(X_1),\dots,f_*(X_n),\xi\big),\quad \ X_i\in\Gamma(TM) \ \text{for} \ i=1,\dots,n$$ \
\begin{defi}
The immersion $f:M\to\R^{n+1}$ is called \emph{equiaffine} if $\nabla\theta=0$ or, equivalently (\cite[Chapter II, Proposition 1.4]{nomizu1994affine}), if $\tau(X)=0$ for any $X\in\Gamma(TM)$.
\end{defi}
\noindent Since $ \theta $ is defined from the data of $ f $ and $ \xi $, we will denote an equiaffine immersion by the pair $ (f, \xi) $. In particular, all centroaffine immersions introduced in Example \ref{ex:centro_affine} are also equiaffine.
\begin{lemma}[\cite{nomizu1994affine}]\label{lem:CodazzihandS}
Let $(f,\xi)$ be an equiaffine immersion, then for any $X,Y,Z\in\Gamma(TM)$ we have \begin{enumerate}
    \item[(i)] $(\nabla_Xh)(Y,Z)=(\nabla_Yh)(X,Z)$;
    \item[(ii)] $(\nabla_XS)Y=(\nabla_YS)X$;
    \item[(iii)] $h(S(X),Y)=h(X,S(Y))$.
\end{enumerate}
\end{lemma}
\noindent The first and second equations in the above lemma are often referred to as the \emph{Codazzi equations} for $h$ and $S$ respectively. In particular, the first one implies that the $(0,3)$-tensor defined by
\begin{equation}\label{eq:tensoreC}
C(X, Y, Z) := (\nabla_X h)(Y, Z)
\end{equation}
is totally symmetric. Up to this point, no assumption has been made on the rank of the symmetric bilinear form $h$, which, a priori, may be degenerate. In fact, it can be shown that the rank does not depend on the choice of the transverse vector field $\xi$, but only on the hypersurface itself (see \cite[Chapter II, Proposition 2.5]{nomizu1994affine}). By requiring the bilinear form $h$ to be non-degenerate, i.e., of maximal rank, one may consider affine immersions with additional properties.
\begin{defi}
Let $(f,\xi)$ be an equiaffine immersion with non-degenerate affine metric $h$, then it is called a \emph{Blaschke immersion} if $\theta=\omega_h$, where $\omega_h$ is the volume form of $h$ on $M$. In such a case, for any $p\in M$ the line going through $p$ in the direction of $\xi_p$ is called the \emph{affine normal}.
\end{defi}
\begin{defi}\label{def:affine_sphere}
A Blaschke immersion $f:M\to\R^{n+1}$ is called a \emph{proper affine sphere} if $S=-\lambda\Id$, for $\lambda:M\to\R^*$ a smooth function. In particular, the affine normals meet at one point in $\R^{n+1}$, called \emph{the center}.
\end{defi}
\begin{remark}
By the Codazzi equation for $S$ (Lemma \ref{lem:CodazzihandS} point (ii)), in the case of a proper affine sphere, it can be shown that the function $\lambda$ must be constant (\cite[Chapter II, Proposition 3.4]{nomizu1994affine}).
\end{remark}
\noindent In the case of a centroaffine immersion $f : M \to \mathbb{R}^{n+1}$ with non-degenerate affine metric, one can obtain a criterion for determining whether $f$ is a proper affine sphere by using the trace of the totally symmetric tensor $C$, defined in (\ref{eq:tensoreC}). Indeed, there exists a unique 1-form $A$ with values in $\mathrm{End}(TM)$ such that
\begin{equation}
C(X, Y, Z) = -2h\big(A(X)Y, Z\big)
\end{equation}
for all $X, Y, Z \in \Gamma(TM)$. In particular, it can be shown that $A = \nabla - \nabla^h$, where $\nabla^h$ denotes the Levi-Civita connection of $h$ (\cite[Chapter II, Proposition 4.1]{nomizu1994affine}). We can thus define the trace of $C = \nabla h$ as a $1$-form on $M$ by the following formula:
\begin{equation}\label{eq:traccia_C}\mathrm{Tr}_h(C)(X) := \operatorname{tr}(h^{-1} \nabla_X h) = -2 \operatorname{tr}\big(A(X)\big) \quad \text{for} \ X\in\Gamma(TM),\end{equation}
where we regard $\nabla_X h$ as a smooth section of $S^2 T^*M$, that is, a symmetric non-degenerate bilinear form on $TM$. In a local $h$-orthonormal basis $\{e_1,\dots,e_n\}$ we have \begin{equation}\label{eq:trace_Pick_tensor}
\mathrm{Tr}_h(C)(X)=\sum_{i=1}^n\epsilon_i\cdot C(X,e_i,e_i)=-2\sum_{i=1}^n\epsilon_i\cdot h(A(X)e_i,e_i),
\end{equation}where $\epsilon_i=h(e_i,e_i)$. 

\begin{prop}[{\cite[Lemma 4.4]{benoist2013cubic}}]\label{prop:properaffinesphere_trace_less}
A centroaffine immersion $f:M\to\R^{n+1}$ with non-degenerate affine metric $h$ is a proper affine sphere if and only if $\trace\big(A(X)\big)=0$ for any $X\in\Gamma(TM)$ or, equivalently, if and only if the $1$-form $\mathrm{Tr}_h(C)$ is equal to zero. 
\end{prop}
\noindent  The next two results are technical lemmas concerning the tensor $A$ and specific symmetries of its covariant derivatives.
\begin{lemma}[{\cite[Chapter II, Proposition 9.7]{nomizu1994affine}}]\label{lem:total_symmetry_nabla_C}
Let $f:M\to\R^{n+1}$ be a Blaschke immersion with non-degenerate metric $h$ and $\det(S)\neq 0$, then the tensor \begin{equation}T(X,Y,Z,W):=h\big((\nabla^h_XA)(Y)Z,W\big)\end{equation} is totally symmetric, i.e. is symmetric in each entry, if and only if $f$ is a proper affine sphere.
\end{lemma}
\noindent It is worth noting that the tensor $(\nabla^h_{\cdot} A)(\cdot)$ is a bi-linear form on $T^*M$ with values in the bundle of endomorphisms of $TM$, in the case where the immersion is merely equiaffine, and not necessarily a proper affine sphere. 
In particular, for any $X,Y\in\Gamma(TM)$, we can define
$$\operatorname{tr}_h\big(\nabla^h A\big)(X,Y) := \operatorname{tr}\big((\nabla^h_X A)(Y)\big).$$
\begin{lemma}\label{lem:trace_covariant_pick_tensor}
Let $(f,\xi)$ be an equiaffine immersion with non-degenerate metric $h$, then $$\trace_h(\nabla^hA)(X,Y)=X\cdot\trace\big(A(Y)\big)-\trace\big(A(\nabla^h_XY)\big).$$In particular, if $f$ is a proper affine sphere then $\trace_h(\nabla^hA)(X,Y)=0$ for any $X,Y\in\Gamma(TM)$.
\end{lemma}
\begin{proof}
Let $\{e_1,\dots,e_n\}$ be a local orthonormal basis for $h$, by the previous definition $$\trace\big(\nabla^hA\big)(X,Y)=\sum_{i=1}^n\epsilon_i\cdot h\big((\nabla_X^hA)(Y)e_i,e_i\big),$$ where $\epsilon_i=h(e_i,e_i)$ is equal to either $1$ or $-1$. Taking the derivative along $X$ of $$\trace(A(Y))=\sum_{i=1}^n\epsilon_i\cdot h(A(Y)e_i,e_i)$$ we obtain $$
X\cdot\trace\big(A(Y)\big)=\sum_{i=1}^n\epsilon_i\Big(h\big(\nabla_X^h\big(A(Y)e_i\big),e_i\big)+h\big(A(Y)e_i,\nabla_X^he_i\big)\Big).
$$ Expanding the derivative of $A(Y)e_i$, we conclude \begin{align*}
X\cdot\trace\big(A(Y)\big)=&\sum_{i=1}^n\epsilon_i\Big(h\big((\nabla_X^hA)(Y)e_i,e_i\big)+h\big(A(\nabla_X^hY)e_i,e_i\big) \\ & +h\big(A(Y)\nabla_X^he_i,e_i\big)+h\big(A(Y)e_i,\nabla_X^he_i\big)\Big).
\end{align*}The first term in the right hand side of the above equation is exactly $\trace_h\big(\nabla^hA\big)(X,Y)$; the second one is equal to $\trace\big(A(\nabla_X^hY)\big)$; and the third and fourth one are equal by $h$-symmetry of $A(Y)$. In the end, we get \begin{align*}\label{eq:trace_nabla_A}X\cdot\trace\big(A(Y)\big)=\trace_h\big(\nabla^hA\big)(X,Y)+\trace\big(A(\nabla^h_XY)\big)-2\sum_{i,j}^n\epsilon_i\cdot\omega^j_i(X)\cdot h(A(Y)e_i,e_j)\end{align*} where $(\omega^j_i)_{ij}$ is defined by $\nabla_X^he_i=\sum_{j=1}^n\omega^j_i(X)e_j$, namely it is the local matrix of $1$-forms associated with $\nabla^h$. It is not difficult to show that, as in the Riemannian setting, the relation $\omega^j_i(X)=-\epsilon_i\epsilon_j\omega^i_j(X)$ holds. With this at hand, we can show that the sum in the above equation is equal to zero; indeed by exchanging the indices $i$ and $j$ we get \begin{align*}
    \sum_{i,j}^n\epsilon_i\cdot\omega^j_i(X)\cdot h(A(Y)e_i,e_j)&=\sum_{i,j}^n\epsilon_j\cdot\omega^i_j(X)\cdot h(A(Y)e_j,e_i) \\ &=-\sum_{i,j}^n\epsilon_i\cdot\omega^j_i(X)\cdot h(e_j,A(Y)e_i) \tag{$A$ is $h$-symmetric}
\end{align*}but this is possible if and only if the sum is zero, namely $$\trace_h(\nabla^hA)(X,Y)=X\cdot\trace\big(A(Y)\big)-\trace\big(A(\nabla^h_XY)\big).$$ In the case where $f$ is a proper affine sphere, we know from Proposition \ref{prop:properaffinesphere_trace_less} that this is equivalent to having $\trace(A(X)) = 0$ for every $X \in \Gamma(TM)$, from which we easily deduce the claim.
\end{proof}
\noindent In the case where a Blaschke immersion $f:M\to\R^{n+1}$ has definite affine metric $h$, there is a unique choice for the affine normal $\xi$ so that $h$ is actually positive-definite (\cite[Chapter III, Proposition 7.2]{nomizu1994affine}). In light of this, if $f:M\to\R^{n+1}$ is a convex Blaschke immersion and the affine normal points towards the convex side, the induced affine metric $h$ is positive definite.
\begin{defi}
Let $f:M\to\R^{n+1}$ be a convex proper affine sphere, i.e. $S=-\lambda\Id$ for $\lambda:M\to\R^*$, then it is called \emph{hyperbolic} if $\lambda>0$ and \emph{elliptic} if $\lambda<0$.
\end{defi}
\begin{remark}
In the setting of the definition, the function \(\lambda\) is in fact constant (\cite[Chapter II, Proposition 3.4]{nomizu1994affine}). In the first case of the above definition, the affine normal pointing toward the convex side, and hence such that $h$ is positive definite, is given by $\xi = f$, while in the second case $\xi = -f$. Moreover, the center of a hyperbolic (resp. elliptic) affine sphere lies on the concave (resp. convex) side of the hypersurface.
\end{remark}
\noindent We conclude this section by introducing the conormal map associated with an equiaffine immersion and stating some key properties that will be used in the sequel (see \cite[Chapter II, \S 5]{nomizu1994affine} for more details). 
\begin{defi}\label{def:conormal_map}
The \emph{conormal map} $\nu:M\to(\mathbb{R}^{n+1})^*\setminus\{0\}$ associated with an equiaffine immersion $(f,\xi)$ is defined by: $$\nu_p(\xi_p)=1,\quad \nu_p\big(f_*(X)\big)=0 \ \ \text{for each} \ X\in T_pM$$ where $\nu_p$ is to be understood as a linear functional on $\mathbb{R}^{n+1}$.\end{defi}
\noindent The conormal map $\nu$ always defines a centroaffine immersion, whose structure equations are given by
\begin{equation}\label{eq:structure_equations_nu}\begin{cases}
    D^*_X\nu_*Y=\nu_*\big(\overline\nabla_XY\big)-\overline h(X,Y)\nu \\ D^*_X\nu=\nu_*(X)
\end{cases}\end{equation} where $D^*$ is the flat connection on $(\mathbb{R}^{n+1})^*$, $\overline\nabla$ is the \emph{dual affine connection} and $-\bar h$ is the induced \emph{dual affine metric}. \begin{prop}[\cite{nomizu1994affine}]\label{prop:properties_of_conormal_map}
The conormal map $\nu$ of an equiaffine immersion $(f,\xi)$ with non-degenerate affine metric $h$ satisfies the following properties: \begin{enumerate}
    \vspace{0.2em}\item[(i)] $\nu_*(Y)(\xi)=0$ and $\nu_*(Y)\big(f_*(X)\big)=-h(Y,X)$ for each $X,Y\in\Gamma(TM)$. In particular, $\nu$ is an immersion; 
    \vspace{0.2em}\item[(ii)] $\overline h(X,Y)=h(S(X),Y)=h(X,S(Y))$ for each $X,Y\in\Gamma(TM)$;
    \vspace{0.2em}\item[(iii)] $X\cdot h(Y,Z)=h(\nabla_XY,Z)+h(Y,\overline\nabla_XZ)$ for each $X,Y,Z\in\Gamma(TM)$;
    \vspace{0.2em}\item[(iv)] $\nabla^h=\frac{1}{2}\big(\nabla+\overline\nabla\big)$.
\end{enumerate}
\end{prop}
\subsection{Convex projective geometry}\label{sec:convex_projective_manifolds}
We briefly recall the theory of convex subsets in projective space, their duals, and some fundamental results on convex projective manifolds (see \cite{benoist2001convexes, cooper2015convex, canary2021anosov} for instance). 
\\ \\
A subset $\Omega\subset\R\mathbb P^n$ is called \emph{convex} if there exists a projective line $l$ disjoint from $\Omega$ such that $\Omega\subset\R\mathbb P^n\setminus l\cong\mathbb A^n$ is convex in the usual sense. A convex subset $\Omega\subset\R\mathbb P^n$ is \emph{properly convex} if it is bounded in some affine chart and \emph{strictly convex} if its boundary $\partial\Omega$ does not contain any non trivial projective segment. Any such $\Omega$ determines a \emph{sharp convex cone} in $\R^{n+1}$ defined as $\mathcal C_+(\Omega):=\{(sx,s) \ | \ x\in\Omega, \ s>0\}$. By a sharp convex cone we mean a convex cone whose interior contains no affine line. In fact, the lift of $\Omega\subset\R\mathbb P^n$ to $\R^{n+1}\setminus\{0\}$ is given by $\mathcal{C}(\Omega):=\mathcal{C}_+(\Omega)\sqcup\mathcal C_-(\Omega)$ the full cone associated with $\Omega$, where $\mathcal C_-(\Omega)=\{(sx,s) \ | \ x\in\Omega, \ s<0\}$. It is also possible to define the \emph{dual sharp convex cone} as $$\mathcal{C}_+^*(\Omega) =\{ \varphi \in (\mathbb{R}^{n+1})^* \mid \varphi(x)> 0, \ \text{for all} \ x \in\mathcal C^+(\Omega)\}.$$  Using the identification of $(\mathbb{R}\mathbb{P}^n)^*$ with the Grassmannian of hyperplanes in $\mathbb{R}^{n+1}$, by sending a projective class of linear functionals to its kernel, we introduce the dual convex $\Omega^*\subset(\R\mathbb P^n)^*$ as the set of hyperplanes in $\mathbb{R}^{n+1}$ that are disjoint from $\overline{\Omega}$. In particular, its boundary $\partial \Omega^*$ becomes the set of supporting hyperplanes of $\Omega$. Moreover, if $\Omega$ is strictly convex, then $\partial\Omega^*$ is $C^1$ so that every point $p \in \partial \Omega$ admits a unique supporting hyperplane at $p$, which in turn corresponds to a unique point in $\partial \Omega^*$. 
\begin{defi}\label{def:strictly_convex_manifold}
A \emph{properly convex} (resp. \emph{strictly convex}) projective manifold $M$ of dimension $n$ is the quotient of a properly convex (resp. strictly convex) subset $\Omega\subset\R\mathbb P^n$ by the action of a discrete subgroup $\Gamma<\mathbb P\GL(n+1,\R)$ which preserves $\Omega$ and acts freely and properly discontinuously on it. 
\end{defi}
\noindent There is an induced action of $\mathbb P\GL(n+1,\R)$ on $(\R\mathbb P^n)^*$ given by $[N]\cdot[\varphi]:=[\varphi\circ N^{-1}]$, for $N\in\GL(n+1,\R)$ and $\varphi\in(\R^{n+1})^*$. In particular, given a properly convex projective manifold $M=\Omega/\Gamma$ we obtain the \emph{dual} properly convex projective manifold $M^*=\Omega^*/\Gamma^*$, where $\Gamma^*$ is a subgroup of $\mathbb P\GL(n+1,\R)$ isomorphic to $\Gamma$ acting on the dual projective space as just described. We will be mostly concerned with the case in which the action of $\Gamma$ on $\Omega$ is cocompact, so that the quotient $M$ becomes a closed smooth manifold. In this case we will say that $\Gamma$ \emph{divides} $\Omega$. Such an actions were studied by Benoist, in the case of $\Omega$ being strictly convex (\cite{benoist2001convexes, benoist2003convexes, benoist2005convexes, benoist2006convexes}), where he proved that, in the cocompact case, the group $\Gamma$ is word-hyperbolic and the inclusion $\Gamma \hookrightarrow \mathbb{P}\mathrm{GL}(n+1,\mathbb{R})$ satisfies a certain Anosov property (\cite{guichard2012anosov}). Our interest in this class of manifolds also stems from the following result, obtained through the contributions of several authors:
\begin{theorem}[\cite{cheng1977regularity, cheng1986complete, gigena1981conjecture, sasaki1980hyperbolic, li1990calabi, li1992calabi}]\label{thm:asymptotic_hyperbolic_sphere}
Given any properly convex subset $\Omega \subset \mathbb{R}^{n+1}$, there exists a unique convex, properly embedded hyperbolic affine sphere
$
f : H \to \mathbb{R}^{n+1},
$
with affine mean curvature equal to $-1$, having the vertex of $\mathcal{C}_+(\Omega)$ as its center, and asymptotic to the boundary $\partial \mathcal{C}_+(\Omega)$. Conversely, every properly immersed hyperbolic affine sphere
$
f : H \to \mathbb{R}^{n+1}
$
is in fact properly embedded and asymptotic to the boundary of the cone generated by the convex hull of $H$ and its center.
\end{theorem}
\noindent Properness of the immersion $f$ is equivalent to the completeness of the affine metric $h$ induced on the hyperbolic affine sphere (\cite{cheng1977regularity, cheng1986complete}). Moreover, it has been shown (\cite{shirokov1962affine}, see also \cite{gigena1978integral, gigena1981conjecture}) that the image of the conormal map
$
\nu:H\to(\mathbb R^{n+1})^*
$
(see Definition \ref{def:conormal_map}) of a hyperbolic affine sphere is another such hyperbolic affine sphere
$
H^*\subset(\mathbb R^{n+1})^*.
$ A fundamental property is that $\nu$ is an isometry for the affine metrics induced on $H$ and $H^*$, and it conjugates the affine connection $\nabla$ on $H$ to the affine connection $\bar\nabla$ on $H^*$. According to Theorem \ref{thm:asymptotic_hyperbolic_sphere}, the hyperbolic affine sphere $H^*$ is asymptotic to the dual cone $\mathcal{C}_+^*(\Omega)$, which allows us to identify $\Omega$ with its dual $\Omega^*$ via $\nu$. In the case where there exists a subgroup
$
\Gamma<\mathbb P\GL(n+1,\mathbb R)
$
dividing $\Omega$, with dual action $\Gamma^*$ induced on $\Omega^*$, the uniqueness statement in Theorem \ref{thm:asymptotic_hyperbolic_sphere} and the $\SL(n+1,\mathbb R)$-invariance of the affine normal together imply that $\nu$ is $\Gamma$-equivariant. In other words, we obtain the following result:
\begin{prop}\label{prop:isometry_dual_immersion}
Given a properly convex projective manifold $M=\Omega/\Gamma$ of dimension $n$, the conormal map $\nu$ with respect to the affine sphere structure induces a map to the dual manifold $M^*=\Omega^*/\Gamma^*$. Such a map is an isometry of the affine metrics and interchanges the two affine connections $\nabla$ and $\bar\nabla$.
\end{prop}

\subsection{Pseudo-Riemannian space forms}\label{sec:hyperbolic_generale_p_q}
Here we introduce the preliminaries of pseudo-Riemannian geometry that will be needed later, together with the pseudosphere and the pseudohyperbolic space (see \cite{anciaux2011minimal} for a general discussion). 
\\ \\ 
Consider the space $\mathbb{R}^{p,q}$ endowed with the symmetric bilinear form $\langle v,w\rangle:=v_1w_1+\dots+v_p w_p - v_{p+1}w_{p+1}-\dots -v_{p+q}w_{p+q}
$
of signature $(p,q)$, with $p,q\ge 1$. Vectors in $\R^{p,q}$ split into three types: \begin{equation*} v \ \text{is} \ \begin{cases} \ \text{spacelike} \quad \text{if} \ \langle v,v\rangle >0, \\ \ \text{timelike} \quad \  \text{if} \ \langle v,v\rangle<0, \\ \ \text{isotropic} \quad  \text{if} \ \langle v,v\rangle=0.
    
\end{cases}\end{equation*}Similarly, a subspace $V\subset\R^{p,q}$ will be called spacelike, timelike, or isotropic depending on whether the bilinear form $\langle\cdot,\cdot\rangle$ restricted to $V$ is positive definite, negative definite, or identically zero. A simple linear algebra argument shows that the maximal dimension of a spacelike, timelike, or isotropic subspace is $p$, $q$, or $\min\{p,q\}$, respectively. The \emph{pseudosphere} of signature $(p-1,q)$ is defined as the quadric
$$
\mathbb{S}^{p-1,q}:=\{v\in\mathbb{R}^{p,q}\mid \langle v,v\rangle=1\},
$$ 
which can be endowed with a pseudo-Riemannian metric of signature $(p-1,q)$ induced by the bilinear form $\langle\cdot,\cdot\rangle$, and of constant sectional curvature equal to $1$. Indeed, for every $v\in\mathbb{S}^{p-1,q}$ we have a decomposition
$$
\mathbb{R}^{p,q}=T_v\mathbb{S}^{p-1,q}\oplus \mathbb{R}\cdot v,
$$ where $T_v\mathbb{S}^{p-1,q}$ is identified with the $\langle\cdot,\cdot\rangle$-orthogonal complement of $v$. In a similar way, we define the \emph{pseudohyperbolic space} of signature $(p,q-1)$ as the quadric $$\mathbb H^{p,q-1}:=\{v\in\mathbb{R}^{p,q}\mid \langle v,v\rangle=-1\}$$ which inherits a pseudo-Riemannian metric of signature $(p,q-1)$ and of constant sectional curvature $-1$, by the same process as the pseudosphere. \begin{remark}\label{rem:pseudo_sphere_anti_isometry}
By replacing the indefinite product $\langle\cdot,\cdot\rangle$ with $-\langle\cdot,\cdot\rangle$, one obtains an anti-isometry between $\mathbb{H}^{p,q-1}$ and $\mathbb{S}^{q-1,p}$. This allows us to refer to $\mathbb{S}^{q-1,p}$ as a pseudohyperbolic space, and conversely to $\mathbb{H}^{p,q-1}$ as a pseudosphere.
\end{remark} 
\noindent The \emph{projective model} $\hat{\mathbb{S}}^{p-1,q}$ of the pseudosphere space is defined as the quotient of $\mathbb{S}^{p-1,q}$ by the antipodal map $v\mapsto -v$, hence the space $\mathbb{S}^{p-1,q}$ is a $2:1$ cover of $\hat{\mathbb{S}}^{p,q}$. While $\mathbb{S}^{p-1,q}$ is orientable for every value of $p$ and $q$, the projective model is orientable if and only if $p+q-1$ is odd. We define the boundary at infinity of the pseudosphere as the subset\begin{equation*}
   \partial_{\infty}\hat{\mathbb S}^{p-1,q}:=\mathbb{P}\big(\{v\in\R^{p,q} \ | \ \langle v,v\rangle=0\}\big)\subset\R\mathbb P^{p+q-1},
\end{equation*}namely the projectivized cone of isotropic vectors. Such a space is often referred to as the \emph{Einstein Universe} and denoted by $\mathrm{Ein}^{p-1,q-1}$. \\ \\ The group of linear isometries of $(\mathbb{R}^{p,q},\langle\cdot,\cdot\rangle)$, which also acts by isometries on $\mathbb{S}^{p-1,q}$, is denoted by $\mathrm{O}(p,q)$ and consists of four connected components. The index-two subgroup formed by the orientation-preserving linear isometries is $\mathrm{SO}(p,q)$, which has two connected components. The component containing the identity, $\mathrm{SO}_0(p,q)$, acts on $\mathbb{R}^{p,q}$ preserving the orientation of both maximal positive and maximal negative subspaces. The subgroup of $\mathrm{SO}(p,q)$ stabilizing a positive line in $\mathbb{R}^{p,q}$ is isomorphic to $\mathrm{S}(\mathrm{O}(p-1,q)\times \mathrm{O}(1))$, and hence $\mathbb{S}^{p-1,q}$ can be realized as the pseudo-Riemannian symmetric space $$
\mathbb{S}^{p-1,q}\cong\mathrm{SO}(p,q)/\mathrm{S}(\mathrm{O}(p-1,q)\times \mathrm{O}(1)),
$$and with a similiar argument for the pseudohyperbolic space we obtain 
$$
\mathbb{H}^{p,q-1}\cong\mathrm{SO}(p,q)/\mathrm{S}(\mathrm{O}(p,q-1)\times \mathrm{O}(1)).
$$

\subsubsection{An isometric model}\label{sec:isometric_model}\hfill\vspace{0.3em}\\  
We will study the space $\mathbb{R}^{p,q}$ in the case $p=q$, and from now on we set $p=q=n+1$. Instead of considering the bilinear form $\langle\cdot,\cdot\rangle$ from Section \ref{sec:hyperbolic_generale_p_q}, in this case we may introduce another one, defined as: \begin{equation}\label{eq:bilinearform_b} b\big((x,y),(\tilde x,\tilde y)\big):=(x^T,y^T)B\begin{pmatrix}\tilde x \\ \tilde y\end{pmatrix}, \quad \text{where} \ B:=\frac{1}{2}\begin{pmatrix}
    0 & \Id \\ \Id & 0
\end{pmatrix}
\end{equation} which is again of signature $(n+1,n+1)$ given that $B$ is congruent to $\begin{psmallmatrix}
    \Id & 0 \\ 0 & -\Id
\end{psmallmatrix}$ via the matrix $P=\begin{psmallmatrix}
    \Id & -\Id \\ \Id & \Id
\end{psmallmatrix}$. Moreover, the squared norm of a vector $(x,y)$ in this model is given: $$b\big((x,y),(x,y)\big)=x^T\cdot y=x_1y_1+\dots+x_{n+1}y_{n+1}.$$ \begin{remark}
It is important to emphasize that, with the use of this bilinear form, the first and second factors of $\mathbb{R}^{n+1}$, parametrized respectively by $\{(x,0)\mid x\in\mathbb{R}^{n+1}\}$ and $\{(0,y)\mid y\in\mathbb{R}^{n+1}\}$, are maximal isotropic subspaces, that is, subspaces on which the restriction of $b$ is the zero bilinear form. We are thus decomposing the ambient space $\mathbb{R}^{n+1,n+1}$ as a direct sum of two maximal transverse isotropic subspaces, in contrast with the situation for the bilinear form $\langle\cdot,\cdot\rangle$ introduced in Section \ref{sec:hyperbolic_generale_p_q}, where the first and second factors of the ambient space are, respectively, positive and negative definite subspaces.
\end{remark}\noindent By taking the set of vectors $(x,y)\in\mathbb{R}^{n+1}\oplus\mathbb{R}^{n+1}$ such that $b\big((x,y),(x,y)\big)=1$, we obtain an isometric model $\mathbb{S}_b^{n,n+1}$ of the pseudosphere (see Section \ref{sec:hyperbolic_generale_p_q}), with explicit isometry given by \begin{equation}\begin{aligned}\label{eq:isometry_pseudo-spheres}
\Phi:\mathbb{S}^{n,n+1}&\longrightarrow \mathbb{S}_b^{n,n+1} \\
(x,y)\ &\mapsto \ \left(x+y,x-y\right).
\end{aligned}\end{equation} From now on, we shall identify $\mathrm{SO}(n+1,n+1)$ with the group of linear transformations of $\mathbb{R}^{n+1}\oplus\mathbb{R}^{n+1}$ preserving the bilinear form $b$ and the orientation of the space, and, with a slight abuse of notation, the space $\mathbb{S}_b^{n,n+1}$ will be denoted by $\mathbb{S}^{n,n+1}$. From what has been said in the previous section, we have a realization as a pseudo-Riemannian symmetric space $$\mathbb{S}^{n,n+1}\cong\mathrm{SO}(n+1,n+1)/\mathrm{S}(\mathrm{O}(n,n+1)\times \mathrm{O}(1))$$ where $\mathrm{SO}(n+1,n+1)=\{M\in\SL(2n+2,\R) \ | \ M^TBM=B\}$.
\begin{remark}\label{rem:changing_signature_p=q}
In light of the anti-isometry between $\mathbb{H}^{p,q-1}$ and $\mathbb{S}^{p-1,q}$ described in Remark \ref{rem:pseudo_sphere_anti_isometry}, we obtain an isomorphism between the groups $\mathrm O(p,q)$ and $\mathrm O(q,p)$ which, in the case $p=q=n+1$, merely interchanges the roles of spacelike and timelike subspaces without altering their maximal possible dimension.\end{remark}
\subsection{The principal $\R$-bundle structure}\label{sec:paracomplex_geometry}  An interesting feature in the neutral signature case $\R^{n+1,n+1}$ is that the isometry in (\ref{eq:isometry_pseudo-spheres}) can be reinterpreted using para-complex geometry. The set of para-complex numbers (\cite{cockle1849iii}), denoted with $\R_\tau$, is defined as the commutative algebra over $\R$ generated by $\{1,\tau\}$ where $\tau$ is a non-real number such that $\tau^2=1$. In particular, any number $z\in \R_\tau$ can be written as $z=x+\tau y$ for $x,y\in\R$. There is a notion of conjugate para-complex number $\bar z^{\tau}:=x-\tau y$ which permits to define the absolute value as $$\vl z \vl_\tau^2=z\bar z^{\tau}=x^2-y^2\in\R.$$ The main difference compared to complex numbers is that the absolute value $\vl\cdot\vl_\tau^2$ does not have a sign and can also be zero. In fact, it can be shown that para-complex numbers with zero absolute value are exclusively the zero divisors. Among these, the two numbers $$e_+:=\frac{1+\tau}{2},\qquad e_-:=\frac{1-\tau}{2},$$ known as idempotent units\footnote{In fact, an easy computation shows that $e_{\pm}\cdot e_\pm=e_\pm$}, enable the algebra $\R_\tau$ to be realized as a direct product: \begin{align*}
    \qquad\quad\R_\tau&=\R\oplus\tau\R\longrightarrow\R e_+\oplus\R e_- \\ & z=x+\tau y\longmapsto (x+y,x-y).
\end{align*}
In particular, since $e_+$ and $e_-$ are linearly independent, any arbitrary para-complex number $z=x+\tau y$ can also be expressed as $z=z^+e_++z^-e_-$, where $z^+:=x+y$ and $z^-:=x-y$. Notice that the space $\mathbb{R}^{n+1}_\tau$ consisting of $(n+1)$-tuples of para-complex numbers, may be identified with $\mathbb{R}^{n+1}e_+ \oplus \mathbb{R}^{n+1}e_-$ via the above component-wise decomposition. This process corresponds precisely to constructing the isometry (\ref{eq:isometry_pseudo-spheres}) on the space $\mathbb{R}^{n+1,n+1}$ between the bilinear form $\langle\cdot,\cdot\rangle$ and $b$. The subgroup of invertible numbers in $\R_\tau$ with absolute value equal to one $$\mathcal{U}:=\{z\in\R_\tau \ | \ \vl z\vl_\tau^2=1\}$$ is equal, as a set, to two hyperbolas in the real plane, and therefore to two copies of $\R$ as a group. This stands in stark contrast to the complex case, where the counterpart of the group $\mathcal U$ defined above is a copy of $S^1$, and is therefore compact. Nevertheless, it can be parametrized as $\mathcal{U}=\{\pm\cosh{t}+\tau\sinh{t} \ | \ t\in\R\}$ or, in the product algebra structure induced by the idempotent units $e_\pm$, we have $$\mathcal{U}=\{(z^+,z^-)\in\R e_+\oplus\R e_- \ | \ (z^+,z^-)=(e^t,e^{-t}) \ \text{or} \ (-e^{-t},-e^t), \ t\in\R\}.$$
Its action on an element $(x,y)\in\R^{n+1,n+1}$ is given by
$
\Psi_t(x,y) := (e^t x, e^{-t} y),
$
or, equivalently, in matrix form, by \begin{equation} \label{eq:matrix_invertible_elements} \Psi_t(x,y)=\begin{pmatrix}
   e^t\cdot\Id & 0 \\ 0 & e^{-t}\cdot\Id
\end{pmatrix}\begin{pmatrix}
    x \\ y
\end{pmatrix}=\Lambda_t\begin{pmatrix}
    x \\ y
\end{pmatrix}, \ \text{for} \ t\in\R.\end{equation}
The $1$-parameter subgroup generated by the matrices $\Lambda_t$ and $-\Lambda_{-t}$ in (\ref{eq:matrix_invertible_elements}) gives rise to an isomorphic copy of $\SO(1,1)$ inside $\SO(n+1,n+1)$ since it preserves the bilinear form $b((x,y),(x,y))=x^T\cdot y$, as an easy computation shows. In particular, the action of the copy of $\mathbb R \cong \SO_0(1,1)$ generated by $\{\Lambda_t\}_{t\in\mathbb R}$ acts properly and freely on $\mathbb S^{n,n+1}$ and on $\mathbb H^{\,n+1,n}$. The quotients
\(
\mathbb S^{n,n+1}/\mathbb R 
\qquad \text{and} \qquad 
\mathbb H^{\,n+1,n}/\mathbb R
\)
are smooth $2n$-dimensional manifolds, diffeomorphic to a para-Kähler manifold $\h_\tau^n$ (see Appendix \ref{appendix_B} for the definition), called the \emph{para-complex hyperbolic space}. In the literature it is generally defined in a different way, using a direct approach based on para-complex numbers (see \cite[\S 8.4]{trettel2019families} and \cite[Definition 2.1]{rungi2025complex}), but this definition is equivalent to ours.
In other words, the pseudosphere and the pseudohyperbolic space can be regarded as two principal $\mathbb R$-bundles over the same para-Kähler manifold. A general argument (Proposition \ref{prop:para_Sasaki_principal_bundle}) allows one to induce a new metric structure on the total space of the bundle from the para-Kähler metric on the base. This structure, referred to as \emph{para-Sasaki} (Definition \ref{def:para_Sasaki_metric}), is described explicitly in Corollary \ref{cor:para_sasaki_pseudo_sphere}. One of the most notable features is that the isometry group of the para-Kähler metric on the base is isomorphic to $\mathrm{GL}(n+1,\mathbb R)$, and the space itself is homogeneous with stabilizer $\mathrm{GL}(n,\mathbb R)\times\mathbb R^{*}$ (\cite[\S 2.2]{rungi2025complex}). Using this fact, one proves that the isometry group of the para-Sasaki metric on the total space is again isomorphic to $\mathrm{GL}(n+1,\mathbb R)$. In other words, we obtain the following:
\begin{lemma}\label{lem:para_sasaki_pseudosphere}
The pseudosphere $\mathbb{S}^{n,n+1}$ or, equivalently the pseudohyperbolic space $\mathbb{H}^{n+1,n}$, endowed with the para-Sasaki metric just described is a $\GL(n+1,\R)$-homogeneous space with stabilizer isomorphic to $\GL(n,\R)$.
 \end{lemma}\noindent We shall see in the next section that this representation of $\mathbb S^{n,n+1}$ as a $\GL(n+1,\R)$-homogeneous space also admits a description in terms of the explicit inclusion of Lie groups $\GL(n+1,\R)\hookrightarrow \SO(n+1,n+1)$. 

\subsubsection{Lie groups inclusion}\label{sec:Lie_groups_inclusion}\hfill\vspace{0.3em}\\
Using the bilinear form $b$ introduced in (\ref{eq:bilinearform_b}), one can explicitly write the inclusion $
\iota:\GL(n+1,\mathbb R)\hookrightarrow \SO(n+1,n+1).
$ Indeed, the space $(\mathbb R^{n+1,n+1},b)$ decomposes as the direct sum $V_+\oplus V_-$ of two maximal isotropic subspaces, and the action of $M\in\GL(n+1,\mathbb R)$ on an element $(v_+,v_-)$ is given by $
M\cdot (v_+,v_-) := (M v_+, (M^T)^{-1} v_-)
$. In other words, the above inclusion is given by \begin{equation}\label{eq:inclusion_GLnR}
\iota:M \longmapsto X:=\begin{pmatrix} M & 0 \\ 0 & (M^T)^{-1} \end{pmatrix}.
\end{equation} It is clear that $\det(X)=1$, and moreover $$
X^T B X = \frac{1}{2}\begin{pmatrix} M^T & 0 \\ 0 & M^{-1} \end{pmatrix} 
\begin{pmatrix} 0 & (M^T)^{-1} \\ M & 0 \end{pmatrix} = B.
$$ The subgroup $\GL(n+1,\R)<\SO(n+1,n+1)$ is acting as the stabilizer of pairs $(V_+,V_-)$ made of transverse maximal dimensional isotropic subspaces in $\R^{n+1,n+1}$. The group homomorphism $\iota$ respects also the inclusion of the corresponding maximal compact subgroups. In fact, the maximal compact in $\SO(n+1,n+1)$ is isomorphic to $\mathrm{S}\big(\mathrm O(n+1)\times\mathrm O(n+1)\big)$ and its inclusion, in terms of the model $(\R^{n+1,n+1}, b)$ is given by \begin{equation}(R,S)\longmapsto\frac{1}{2}\begin{pmatrix}
    R+S & R-S \\ R-S & R+S
\end{pmatrix},\end{equation} for $R,S\in\mathrm O(n+1)$ with $\det R=\det S$. On the other hand, by definition of $\iota$, given $M\in\GL(n+1,\R)$, we have $$\iota(M)=\begin{pmatrix} M & 0 \\ 0 & M \end{pmatrix} \iff M\in\mathrm O(n+1),$$where $\mathrm O(n+1)$ is an isomorphic copy of the maximal compact in $\GL(n+1,\R)$. Therefore, we deduce that $\iota$ sends the maximal compact of $\GL(n+1,\R)$ into $\mathrm S\big(\mathrm{O}(n+1)\times\mathrm O(n+1)\big)$ as the diagonal subgroup.
\begin{remark}
In terms of the pseudosphere $\mathbb S^{n,n+1}$ (respectively, the pseudohyperbolic space $\mathbb H^{n+1,n}$), the Lie group $\SO(n+1,n+1)$ acts by isometries only with respect to the pseudo-Riemannian metric $g$ of signature $(n,n+1)$ (respectively, $(n+1,n)$). In contrast, the subgroup $\GL(n+1,\R)$ can be interpreted as the isometry group of the para-Sasaki structure (see Corollary \ref{cor:para_sasaki_pseudo_sphere}).
\end{remark}

\section{The main construction}\label{sec:main_construction}

\subsection{Immersions in the pseudosphere}\label{sec:immersions_pseudo_spheres}
This constitutes one of the central parts of the paper, where we will explain in detail the relation between non-degenerate equiaffine immersions $f:M\to\R^{n+1}$ (not necessarily centroaffine) and non-degenerate immersions $\sigma^+:M\to\mathbb S^{n,n+1}$ compatible with the para-Sasaki metric (see Corollary \ref{cor:para_sasaki_pseudo_sphere}). 
In the case where the immersion in $\mathbb S^{n,n+1}$ is maximal, this becomes equivalent to having 
a proper affine sphere in $(\R^{n+1})^*$. Two fundamental remarks are in order at this stage: first, an implicit approach has benn studied by Vrancken for centroaffine immersions (\cite{LucVrancken2002}). 
Second, when the immersion $f:M\to\R^{n+1}$ is a hyperbolic affine sphere, so that $\sigma^+:M\to\mathbb S^{n,n+1}$ is timelike and maximal, the subspace $\mathrm d\sigma^+\big(TM\big)$ does not saturate the maximal possible dimension of negative definite subspaces of $T\mathbb S^{n,n+1}$, being $M$ of dimension $n$. Their anti-isometric counterpart are the $n$-dimensional spacelike maximal immersions $\sigma^-:M\to\mathbb H^{n+1,n}$, which differ substantially from those studied in \cite{CTT, labourie2023quasicircles,labourie2024plateau,seppi2023complete, moriani2024rigidity} and related to the study of Anosov representations in $\mathrm{SO}(p,q)$, with $p\le q$ (\cite{danciger2018convex,beyrer2023mathbb}).
\\ \\
We will now provide a brief introduction to the general theory of immersions into the pseudosphere or, equivalently, into the pseudohyperbolic space, with the necessary adjustments, before moving on to the specific construction of the maps $\sigma^+$ and $\sigma^-$ mentioned at the beginning of the section. An immersed $n$-dimensional manifold $\sigma: M \to \mathbb{S}^{n,n+1}$ will be called \emph{non-degenerate} whenever the induced metric $g_T:=\sigma^*g|_{TM}$ is non-degenerate. Moreover, it is said \emph{timelike} (resp. \emph{spacelike}) if $g_T$ is negative definite (resp. positive definite) and, in such a case, its normal bundle $NM:=(\sigma^*T\mathbb S^{n,n+1})^{\perp_{g}}$ is endowed with a positive definite (resp negative definite) metric $g_{N}$. Using the splitting $\sigma^*T\mathbb{S}^{n,n+1}=TM\perp_{g} N\Sigma$, the Levi-Civita connection on $\mathbb{S}^{n,n+1}$ decomposes as
\[
        \nabla^{g}=\begin{pmatrix}
            \nabla^{g_T} & B \\
            \II & \nabla^{N} 
        \end{pmatrix} \ ,
\]
where $\nabla^{g_T}$ is the Levi-Civita connection of the induced metric on $\sigma(M)$, $\II$ is a $1$-form with values in $\Hom(TM, NM)$ and it is called the \emph{second fundamental form}, $B\in \Omega^{1}(M, \Hom(NM, TM))$ is the \emph{shape operator}, and $\nabla^{N}$ is compatible with the metric $g_{N}$. Since $\nabla^{g}$ is torsion-free, the second fundamental form is actually symmetric, in the sense that
\[
    \II(X,Y)=\II(Y,X) \ \ \ \ \forall X,Y\in \Gamma(TM)  .
\]
Moreover, the second fundamental form and the shape operator are related by
\[
    g_{N}(\II(X,Y), \hat n) = -g_{T}(Y,B(X,\hat n))
\]
for all $X, Y \in \Gamma(TM)$ and normal vector fields $\hat n\in\Gamma(NM)$. 
\begin{remark}\label{rem:change_signature_induced_metric}
As explained in Remark \ref{rem:changing_signature_p=q}, there is an anti-isometry \begin{align*}\mathrm F:\big(&\R^{n+1,n+1}, b\big)\to\big(\R^{n+1,n+1}, b\big) \\ & \ \ \ \ \ \ \ (x,y)\longmapsto(-x,y)\end{align*}which induces a diffeomorphism between the quadrics $\mathbb{S}^{n,n+1}$ and $\mathbb{H}^{n+1,n}$, as it interchanges timelike and spacelike subspaces. In particular, if $\sigma : M \to \mathbb{S}^{n,n+1}$ is a non-degenerate immersion and the induced metric $g_T$ has signature $(k, n-k)$, then the new immersion obtained as $\mathrm{F} \circ \sigma : M \to \mathbb{H}^{n+1,n}$ is again non-degenerate, but now the induced metric has the opposite signature, being equal to $-g_T$.
\end{remark}
\begin{defi}\label{def:maximal} A non-degenerate immersion $\sigma: M \to \mathbb{S}^{n,n+1}$ is \emph{maximal} if its \emph{mean curvature} $\mathbf{H}_\sigma$ vanishes identically, namely if $\trace_{g_T}(\II)=0$.
\end{defi}

\noindent It follows easily from the definition that the immersion is maximal if and only if in a $g_T$-orthonormal frame $\{e_{1}, e_{2},\dots,e_n\}$ of $TM$, we have 
\begin{equation}
    \sum_{i=1}^n\epsilon_i\cdot\II(e_i,e_i)=0 \ ,
\end{equation}where $\epsilon_i:=g_T(e_i,e_i)$ is equal to plus or minus one, for $i=1,\dots,n$. Before proceeding, we must briefly introduce the special metric structure of \(\mathbb{S}^{\,n+1,n}\), which will then allow us to 
define additional properties for immersions into the pseudosphere (for more details see Appendix \ref{appendix_B}). A \emph{para-Sasaki metric} on $\mathbb{S}^{n,n+1}$ (Corollary \ref{cor:para_sasaki_pseudo_sphere}) consists of a tuple $(\eta,\zeta,\phi, g)$ where: \begin{itemize}
    \item[(i)] $\eta$ is a $1$-form such that $\eta\wedge(\mathrm d\eta)^n \neq 0$;\vspace{0.3em}
    
\item[(ii)] $\zeta$ is the unique vector field such that $\iota_\zeta(\mathrm d\eta)=\mathrm d\eta(\zeta,\cdot)=0$ and $\eta(\zeta)=1$. In particular, there is a decomposition
$T\mathbb{S}^{n,n+1} = \Ker \ \eta \oplus \mathbb R \zeta;$\vspace{0.3em}

\item[(iii)] $\phi$ is an endomorphism of $T\mathbb{S}^{n,n+1}$ such that $\phi(\zeta)=0$, $\eta\circ\phi=0$, $\phi^2(X)=X-\eta(X)\zeta$, and the restriction of $\phi$ to the (non-integrable) $2n$-dimensional distribution $\Ker \ \eta$ defines a para-complex structure (see Definition \ref{def:para_complex_structure});\vspace{0.3em}

\item[(iv)] $g$ is a pseudo-Riemannian metric of signature $(n,n+1)$ such that $g(\phi(X),\phi(Y))=-g(X,Y)+\eta(X)\eta(Y)$ and $\mathrm d\eta(X,Y)= g(X,\phi(Y))$;\vspace{0.3em}

\item[(v)] The tensor $[X, Y] + [\phi(X), \phi(Y)] - \phi\big([\phi(X), Y]\big) - \phi\big([X, \phi(Y)]\big)-\eta\big([X,Y]\big)\zeta-\mathrm{d\eta}(X,Y)$ vanishes for any $X,Y\in\Gamma(T\mathbb{S}^{n,n+1})$. 
\end{itemize} With this additional structure we can introduce a new class of immersions.
\begin{defi}\label{def:horizontal_phi_anti_invariant}
A non-degenerate immersion $\sigma:M\to\mathbb S^{n,n+1}$ is called \emph{horizontal $\phi$-anti-invariant} if: \begin{itemize}
    \item[(i)] for any $X\in\Gamma(TM)$ we have $\sigma_*(X)\in\mathcal H$ (\emph{horizontal}); \item[(ii)] for any $X,Y\in\Gamma(TM)$ we have $\phi\big(\sigma_*(X)\big)\perp_g\sigma_*(Y)$ (\emph{$\phi$-anti-invariant});\footnote{This definition comes from the similar setting of anti-invariant submanifolds in Sasaki spaces (\cite{yano1977anti,ishihara1979anti}).}
\end{itemize}
\end{defi}

\begin{remark}\label{rem:equivalent_anti_invariant}
Using point (v) of Definition  \ref{def:para_metric_contact}, the property of being $\phi$-anti-invariant can be rephrased using the contact geometry of $\mathbb S^{n,n+1}$ since $$
g\big(\sigma_*(X),\phi(\sigma_*(Y))\big)=\mathrm d\eta\big(\sigma_*(X),\sigma_*(Y)\big).
$$ As a consequence, we obtain that $\sigma$ is $\phi$-anti-invariant if and only if $
\mathrm d\eta\big(\sigma_*(X),\sigma_*(Y)\big)=0
$, for any $X,Y\in\Gamma(TM)$. If the immersion $\sigma$ is horizontal as well, then the horizontal distribution can be decomposed as
$
\mathcal H \;=\; \sigma_*(TM)\;\oplus\;\p\big(\sigma_*(TM)\big),
$
given that the restriction of $\phi$ to $\mathcal{H}$ defines a para-complex structure $\p$ (see Definition \ref{def:para_metric_contact}).
\end{remark}
\noindent In what follows, we shall adopt the model $
\R^{n+1,n+1}=\R^{n+1}\oplus(\R^{n+1})^*,
$ endowed with the para-Kähler structure $(\hat\p,\hat\g,\hat\ome)$, where
$$
\hat\p=(\Id,-\Id), \qquad 
\hat\g\big((v,\varphi),(w,\psi)\big)=\tfrac{1}{2}\big(\varphi(w)+\psi(v)\big),
$$ and
$
\hat\ome\big((v,\varphi),(w,\psi)\big)=\hat\g\big((v,\varphi),\hat\p(w,\psi)\big)=\tfrac{1}{2}\big(\varphi(w)-\psi(v)\big)
$. Recall that, in this model, the pseudosphere $\mathbb S^{n,n+1}$ is obtained as the quadric $\{(v,\varphi)\in\R^{n+1}\oplus(\R^{n+1})^*\ | \ \varphi(v)=1\}$ and, similarly, for $\mathbb H^{n+1,n}$ the condition is $\varphi(v)=-1$. Moreover, we can identify the dual space $(\mathbb{R}^{n+1})^*$ with $\mathbb{R}^{n+1}$ via a fixed non-degenerate symmetric bilinear form $Q$, that is,$$\begin{aligned}(\mathbb{R}^{n+1})^* &\longleftrightarrow \mathbb{R}^{n+1} \\\varphi &\longmapsto v_\varphi\end{aligned}$$where $v_\varphi$ is the unique vector such that $\varphi(w) = v_\varphi^t Q w$. Under this identification, it is clear that the bilinear form $\hat\g((v,\varphi),(w,\psi))=\frac{1}{2}\big(\varphi(w)+\psi(v)\big)$ corresponds exactly to the form $b$ introduced in (\ref{eq:bilinearform_b}) whenever $Q = \mathrm{Id}$. \\ \\ Given a non-degenerate equiaffine immersion $f:M\to\R^{n+1}$ and its conormal map $\nu:M\to(\R^{n+1})^*$ (see Section \ref{sec:affine_differential_geometry}), we can define a map \begin{equation}\begin{aligned}\label{eq:immersion_sigma}
    \sigma^+: \ &M\to\R^{n+1}\oplus(\R^{n+1})^* \\ &
p\longmapsto\sigma^+(p):=(\xi_p,\nu_p)
\end{aligned}\end{equation}
where $\xi_p$ denotes the transversal vector to $f$ at the point $p$, or, with a slight abuse of notation, $\sigma^+=(\xi,\nu)$.

\begin{theorem}\label{prop:immersion_sigma}
The map $\sigma^+:M\to\R^{n+1}\oplus(\R^{n+1})^*$ has the following properties: \begin{itemize}
    \item[(i)] the image of $\sigma^+$ is contained in $\mathbb S^{n,n+1}$;
    \item[(ii)] $\sigma^+$ is an immersion;
    \item[(iii)] the induced metric $g_T=(\sigma^+)^*g|_{TM}$ coincides with $\bar h=h\big(S(\cdot),\cdot\big)$, namely the induced affine metric on the dual immersion up to a sign;
    \item[(iv)] $\sigma^+$ is horizontal and $\phi$-anti-invariant.
\end{itemize}
\end{theorem}
\begin{proof}
Points (i) and (ii) are straightforward to verify since
$$
\hat\g(\sigma,\sigma)=\hat\g\big((\xi,\nu),(\xi,\nu)\big)=\frac{1}{2}\big(\nu(\xi)+\nu(\xi)\big)=1
$$
(see Definition \ref{def:conormal_map}), and hence $\sigma^+(M)\subset \mathbb S^{n,n+1}$. In the following, we will denote by $\sigma^+_*(X)$ the differential of $\sigma^+$ applied to a vector field $X$ along $M$. Using the structure equations for $f$ and $\nu$, we have
$
\sigma_*^+(X) = \big(D_X \xi, D^*_X \nu\big) = \big(-f_*(S(X)), \nu_*(X)\big).
$
Therefore, $\sigma_*^+(X)=0$ if and only if $f_*(SX) = \nu_*(X) = 0$. Since $\nu$ is an immersion (Proposition \ref{prop:properties_of_conormal_map}), we conclude that the differential of $\sigma^+$ is injective. \newline Regarding point (iii), we need to compute $\hat\g(\sigma_*^+(X),\sigma_*^+(Y))$ for $X,Y\in \Gamma(TM)$. Using the definitions of $\hat\g$ and $\sigma_*^+$, we have
\begin{align*}
    \hat\g(\sigma_*^+(X),\sigma_*^+(Y))&=\hat\g\big((-f_*(SX), \nu_*(X)),(-f_*(SY), \nu_*(Y))\big) \\ &=\frac{1}{2}\Big(\nu_*(X)\big(-f_*(SY)\big)+\nu_*(Y)\big(-f_*(SX)\big)\Big) \\ &=\frac{1}{2}\Big(h(SY,X)+h(Y,SX)\Big) \tag{Proposition \ref{prop:properties_of_conormal_map}} \\ &=\overline{h}(X,Y) \tag{$S$ is $h$-symmetric}
\end{align*}where we recall that $-\overline{h}$ is the affine metric induced on the dual immersion (see Equation (\ref{eq:structure_equations_nu})). \newline 
Thanks to Remark \ref{rem:contact_geometry_hypersurfaces} and to the explicit construction of the contact form $\eta$ on $\mathbb S^{n,n+1}$, we know that for a map $\sigma^+$ as above, the Reeb vector field $\zeta$ coincides with $\hat\p(\xi,\nu)=(\xi,-\nu)$. Moreover, the vector field in (\ref{eq:Liouville_vector_field_pseudosphere}), which is transverse to the quadric $\mathbb S^{n,n+1}$, is represented by the position vector $(\xi,\nu)$ in the model we are using. From this it is clear that $\eta((\xi,-\nu))=\hat\ome\big((\xi,\nu),(\xi,-\nu)\big)=1$ (see Proposition \ref{prop:contact_submanifolds}) and that $$
\hat\g\big(\sigma_*^+(X),(\xi,-\nu)\big)= \hat\g\big((-f_*(SX),\nu_*X),(\xi,-\nu)\big)=\tfrac{1}{2}\big(\nu_*(X)(\xi)+\nu(f_*(SX))\big).
$$ The last term in the equation vanishes as a consequence of Definition \ref{def:conormal_map} and item (i) of Proposition \ref{prop:properties_of_conormal_map}. Therefore, the immersion $\sigma^+$ is horizontal. Finally, in order to prove that it is $\phi$-anti-invariant we will equivalently prove that $\mathrm d\eta(\sigma_*^+(X),\sigma_*^+(Y))=0$, for any $X,Y\in\Gamma(TM)$ according to Remark \ref{rem:equivalent_anti_invariant}. By construction, $\mathrm d\eta=j^*\hat\ome$ (see Proposition \ref{prop:contact_submanifolds}), where $j:\mathbb S^{n,n+1}\hookrightarrow\R^{n+1}\oplus(\R^{n+1})^*$ is the inclusion. But since $\hat\ome(j_*\sigma_*^+(X),j_*\sigma_*^+(Y))=\hat\g(j_*\sigma_*^+(X),\hat\p j_*\sigma_*^+(Y))$, a straightforward computation as above shows that this last term vanishes, yielding the claim.
\end{proof}
\begin{remark}
By replacing the transversal vector $\xi$ with $-\xi$, we get a different immersion $\sigma^-:=(-\xi,\nu)$ whose image is now contained in $\mathbb H^{n+1,n}$. In other words, one obtains the equivalent immersion via the anti-isometry $F$ introduced in Remark \ref{rem:change_signature_induced_metric} between $\mathbb S^{n,n+1}$ and $\mathbb H^{n+1,n}$. In particular, $\sigma^-$ has the same properties as $\sigma^+$ proved in Theorem \ref{prop:immersion_sigma}, with the sole exception that the metric induced on $\sigma^-(M)$ now coincides precisely with the affine metric $-\bar h(\cdot,\cdot)=-h\big(S(\cdot),\cdot\big)$.

\end{remark}
\noindent We shall continue to focus on the immersion $\sigma^+$, since the details for $\sigma^-$ can be easily adapted. Notice that the induced metric $g_T$ coincides with the affine metric of the dual immersion up to a sign, hence the immersion $\sigma^+$ defined in (\ref{eq:immersion_sigma}) is not necessarily non-degenerate. Indeed, from the relation $\overline{h}(\cdot,\cdot)=h\big(S(\cdot),\cdot\big)$ it follows that $\sigma^+$ is non-degenerate if and only if $\det(S)\neq 0$, where we recall that $S$ is the affine shape operator of $f$. When this condition holds, both $h$ and $\overline h$ are, a priori, pseudo-Riemannian metrics. Moreover, the normal bundle of the immersion $\sigma^+$ is given by
$
NM = \mathbf{P}(TM)\oplus \R\,\mathbf{P}\sigma^+,
$ where we identify the Reeb vector field $\zeta$, evaluated at the point $\sigma^+(p)$, with $\mathbf{P}(\sigma^+(p))=(\xi_p,-\nu_p)$, as $p$ varies in $M$. A priori, the mean curvature is a section of the normal bundle of the immersion, but as we shall see in next result, in the special case of $\sigma^+$, its mean curvature $\mathbf{H}_{\sigma^+}$ does not have component along $\p\sigma^+$ and is directly related with some affine invariants of the dual immersion. 
\begin{lemma}\label{lem:mean_curvature_sigma}
Suppose that $\sigma^+=(\xi,\nu):M\to\mathbb S^{n,n+1}$ is non-degenerate, then its mean curvature is given by $$\mathbf H_{\sigma^+}=-\frac{1}{2n}\sum_{k=1}^n\epsilon_k\cdot \trace_{-\bar h}(\bar C)(e_k)\p\big(\sigma_*(e_k)\big),$$ where $\bar C=\bar\nabla\bar h$ is the Pick tensor of the dual immersion and $\{e_1,\dots,e_n\}$ is a $g_T=\bar h$-orthonormal local basis, with $\epsilon_i:=g_T(e_i,e_i)$.
\end{lemma}
\begin{proof}
Let $\{e_1,\dots,e_n\}$ be a local orthonormal basis with respect to the induced metric $g_T$, which is given by $\bar h$ according to Theorem \ref{prop:immersion_sigma}, and let $\epsilon_i:=\bar h(e_i,e_i)$ equal to plus or minus one. The mean curvature tensor is given, by definition $$\mathbf H_{\sigma^+}=\trace_{\bar h}\II=\frac{1}{n}\sum_{i=1}^n\epsilon_i\cdot\II(e_i,e_i).$$ The second fundamental form is obtained by projecting the ambient canonical connection $D \oplus D^*$ from the pull-back bundle of the tangent space of $\R^{n+1}\oplus(\R^{n+1})^*$ onto
$NM=\p(TM)\oplus\R\p\sigma^+$. In other words, we need to compute the following term
\begin{align*}
\II(e_i,e_i)&=-\sum_{k=1}^n \epsilon_k\cdot\hat\g\big((D\oplus D^*)_{e_i}\sigma_*^+e_i, \p(\sigma_*^+e_k)\big) \p\big(\sigma^+_*e_k\big) \\ & \ \ \ \ \ \ -\; \hat\g\big((D\oplus D^*)_{e_i}\sigma_*^+e_i, (\xi,-\nu)\big) (\xi,-\nu).
\end{align*} By making use of the structure equations of the immersion $f$ and its dual, we obtain
\begin{align*}
    (D\oplus D^*)_{e_i}\sigma_*^+e_i&=(D\oplus D^*)_{e_i}(-f_*(Se_i), \nu_*e_i) \\ &=(-D_{e_i}f_*(Se_i), D^*_{e_i}\nu_*e_i) \\ &=\big(-f_*(\nabla_{e_i}(Se_i))-h(e_i,Se_i)\xi, \nu_*(\bar\nabla_{e_i}e_i)-\bar h(e_i,e_i)\nu\big).
\end{align*}The term outside the sum in the expression of $\II(e_i,e_i)$ can be computed directly and is shown to vanish:
\begin{align*}
    \hat\g\big((D\oplus D^*)_{e_i}\sigma_*^+e_i, (\xi,-\nu)\big)&=\frac{1}{2}\Big(\nu\big(f_*(\nabla_{e_i}(Se_i))\big)+h(e_i,Se_i)\cdot\nu(\xi) \\ &  \ \ \ \ \ +\nu_*\big(\bar\nabla_{e_i}e_i\big)(\xi)-h(e_i,Se_i)\cdot\nu(\xi)\Big) \\ &=0,
\end{align*}where we have used point (i) of Proposition \ref{prop:properties_of_conormal_map} and the definition of the dual immersion. It is also important to observe that this computation implies that the second fundamental form of $\sigma^+$, and hence its mean curvature, take values in the part of the normal bundle of $M$ lying in the horizontal distribution $\mathcal H$ induced by the contact structure on $\mathbb S^{n,n+1}$ (see Proposition \ref{prop:contact_pseudo_hyperbolic}). For the remaining terms in the sum, we first observe that
$
\p\big(\sigma_*^+e_k\big)=(-f_*(Se_k),-\nu_*(e_k))
$ and then we compute: \begin{align*}
    \hat\g\big((D\oplus D^*)_{e_i}\sigma_*^+e_i, \p(\sigma_*^+e_k)\big)&=\frac{1}{2}\Big(-\nu_*\big(\bar\nabla_{e_i}e_i\big)\big(f_*(Se_k)\big)+h(e_i,Se_i)\nu\big(f_*(Se_k)\big) \\ & \ \ \ \ \ +\nu_*(e_k)\big(f_*(\nabla_{e_i}(Se_i))\big)+h(e_i,Se_i)\nu_*(e_k)(\xi)\Big) \\ & =\frac{1}{2}\Big(-\nu_*\big(\bar\nabla_{e_i}e_i\big)\big(f_*(Se_k)\big) +\nu_*(e_k)\big(f_*(\nabla_{e_i}(Se_i))\big)\Big) \\ &=\frac{1}{2}\Big(h\big(\bar\nabla_{e_i}e_i,Se_k\big)-h\big(\nabla_{e_i}(Se_i),e_k\big)\Big).
\end{align*} In the first and second equalities, we have used, as above, the properties of the dual immersion and point (i) of Proposition \ref{prop:properties_of_conormal_map}. At this point, we claim that the following equation is satisfied \begin{equation}\label{eq:Pick_tensor_dual}
    h\big(\bar\nabla_{e_i}e_i,Se_k\big)-h\big(\nabla_{e_i}(Se_i),e_k\big)=-\bar C(e_i,e_k,e_i),
\end{equation}which we now use to obtain the desired expression for the mean curvature. In fact, in light of (\ref{eq:Pick_tensor_dual}) we have \begin{align*}
    \mathbf H_{\sigma^+}&=\frac{1}{2n}\sum_{i,k}^n\epsilon_i\cdot\epsilon_k\cdot\bar C(e_k,e_i,e_i)\p(\sigma_*^+e_k) \tag{$\bar C$ is totally-symmetric} \\ &=-\frac{1}{2n}\sum_{k=1}^n\epsilon_k\cdot\trace_{-\bar h}(\bar C)(e_k)\p(\sigma_*^+e_k) \tag{Equation (\ref{eq:trace_Pick_tensor})}.
\end{align*}\footnote{It is important to emphasize that the affine metric induced on the dual immersion is $-\bar h$, hence the trace of $\bar C$ is computed with respect to this metric.} To complete the proof, we now show that formula (\ref{eq:Pick_tensor_dual}) holds. Recall that by definition, $\bar C(e_i,e_k,e_i) = (\bar\nabla_{e_i} \bar h)(e_k,e_i)$, and hence \begin{align*}
    -(\bar\nabla_{e_i} \bar h)(e_k,e_i)&=-e_i\cdot\bar h(e_k,e_i)+\bar h(\bar\nabla_{e_i}e_k,e_i)+\bar h(\bar\nabla_{e_i}e_i,e_k) \tag{Leibniz rule} \\ &=-e_i\cdot h(e_k,Se_i)+h(\bar\nabla_{e_i}e_k,Se_i)+ h(\bar\nabla_{e_i}e_i,Se_k) \tag{$\bar h(\cdot,\cdot)=h(\cdot,S\cdot)$} \\ &=-h\big(e_k,\nabla_{e_i}(Se_i)\big)+ h(\bar\nabla_{e_i}e_i,Se_k). \tag{Proposition \ref{prop:properties_of_conormal_map}}
\end{align*}\end{proof}
\noindent The components of the mean curvature of $\sigma^+$, up to a constant factor, coincide with the functions obtained by evaluating the one form $\operatorname{tr}_{-\bar h}(\bar C)$ on the $g_T=\bar h$-orthonormal basis $\{e_1,\dots,e_n\}$. In particular, $\mathbf H_{\sigma^+}=0$ if and only if each component vanishes, that is, $\operatorname{tr}_{-\bar h}(\bar C)(e_k)=0$ for every $k=1,\dots,n$. This is equivalent to requiring that the one form $\operatorname{tr}_{-\bar h}(\bar C)$ vanishes identically, and therefore, by Proposition \ref{prop:properaffinesphere_trace_less}, that $\nu$ is a proper affine sphere, being centroaffine. 
 We have thus established the following result.
\begin{cor}\label{cor:sigma_maximal_iff_proper_affine_sphere}
The non-degenerate immersion $\sigma^+=(\xi,\nu):M\to\mathbb S^{n,n+1}$, or equivalently $\sigma^-=(-\xi,\nu):M\to\mathbb H^{n+1,n}$, is maximal if and only if the dual immersion $\nu:M\to(\R^{n+1})^*$ is a proper affine sphere.
\end{cor}
\begin{remark}\label{rem:dual_of_affine_sphere}
In general, taking the dual of $\nu$ and using the canonical identification between $\mathbb{R}^{n+1}$ and its double dual, one obtains a centroaffine immersion $\nu^*$ into $\mathbb{R}^{n+1}$, which is obtained from the initial one $f:M\to\R^{n+1}$ by rescaling its affine data using the shape operator $S$ (see \cite[\S 2.3]{nie2022regular}). In particular, if $f$ is assumed to be centroaffine from the beginning, the double dual $\nu^*$ gives the same immersion back. In the case where $\nu$ is a proper affine sphere, the same its true for its double dual $\nu^*$ (see \cite{gigena1978integral,gigena1981conjecture}), thus yielding a proper affine sphere in $\mathbb{R}^{n+1}$ related to the initial immersion via the shape operator.
\end{remark}
\subsection{Immersions in para-complex hyperbolic space}\label{sec:immersion_para-complex_hyperbolic}
\noindent Here we aim to study the immersions obtained by projecting $\sigma^+$ or $\sigma^-$ into the para-complex hyperbolic space, using the structure of principal $\mathbb R$–bundle described in Section \ref{sec:paracomplex_geometry}.
\\ \\
We briefly recall that the action of $\mathbb R \cong \SO_0(1,1) < \SO_0(n+1,n+1)$ introduced in (\ref{eq:matrix_invertible_elements}) and generated by the matrices $\Lambda_t=\begin{psmallmatrix}
e^t\Id & 0 \\ 0 & e^{-t}\Id
\end{psmallmatrix}$ gives rise to a principal $\mathbb R$–bundle
\(
\pi_{+} : \mathbb S^{n,n+1} \to\mathbb S^{n,n+1}/\mathbb R
\)
whose base is a smooth $2n$–dimensional manifold identified with the para-complex hyperbolic space (see \cite{trettel2019families, rungi2025complex} for an alternative definition). More explicitly, it can be defined as
\[
\mathbb H_{\tau}^{n}
:= 
\big\{(x,y)\in \mathbb R^{n+1,n+1} \ \big|\ b((x,y),(x,y)) = +1 \big\} \big/_{\sim},
\]
where $(x,y)\sim (\tilde x,\tilde y)$ if and only if $\tilde x = e^{t}x$ and $\tilde y = e^{-t}y$. One can show that $\mathbb H_{\tau}^{n}$ carries a para-Kähler structure $(\p, \ome, \g)$ (Appendix \ref{appendix_A}), where $\p$ is a para-complex structure, $\ome$ is a symplectic form, and $\g$ is a pseudo-Riemannian metric of signature $(n,n)$ satisfying $g(\p\cdot,\p\cdot) = -\g$ and $\ome = g(\cdot, \p\cdot)$. Such a para-K\"ahler metric can be induced from the one defined on $\R^{n+1}\oplus(\R^{n+1})^*$ introduced in Section \ref{sec:immersions_pseudo_spheres} and denoted by $(\hat\p,\hat\ome,\hat\g)$. \begin{prop}[{\cite{trettel2019families},\cite{RT_bicomplex}}]\label{prop:diffeo_hyperbolic_incidence}
The space $$\mathbb{P}\mathcal L:=\{([v],[\varphi])\in \R\mathbb P^n\times(\R\mathbb P^n)^{*} \ | \ \varphi(v)\neq 0\},$$ is diffeomorphic to the projective model $\hat\h_\tau^n$ of the para-complex hyperbolic space.
\end{prop}\noindent By the projective model $\hat{\mathbb H}_\tau^n$ we simply mean the quotient of the projective model of the pseudosphere $\hat{\mathbb S}^{\,n,n+1}$ by the $\mathbb R \cong \SO_0(1,1)$ action given in (\ref{eq:matrix_invertible_elements}). There is a well-defined notion of para-holomorphic sectional curvature for the metric $\g$, and it can be shown (\cite[Proposition 2.5]{rungi2025complex}) that this curvature is constant and positive. If instead one considers the same space as the base of the principal $\mathbb R$–bundle $\pi_- : \mathbb H^{n+1,n} \to \mathbb H_\tau^n$, then the para-holomorphic sectional curvature is constant and negative. This is due to the fact that the anti-isometry between the pseudosphere and the pseudohyperbolic space (see Remark \ref{rem:changing_signature_p=q}) descends to an involution on $\mathbb H_\tau^n$, reversing the sign of the pseudo-Riemannian metric. \\ \\ The main result of this section is the following theorem and its corollary.
\begin{theorem}\label{prop:immersion_in_H_tau}
The non-degenerate immersion $\sigma^+=(\xi,\nu):M\to\mathbb S^{n,n+1}$ obtained from the equiaffine immersion $f:M\to\R^{n+1}$ and its dual $\nu$ induces a Lagrangian immersion $\bar\sigma:=\pi_+\circ\sigma^+:M\to\h_\tau^n$, namely $(\bar\sigma)^*\ome|_{TM}\equiv 0$, and the induced metric coincides with $\bar h$. Moreover, if $\sigma^+$ is maximal the same holds for $\bar\sigma$.
\end{theorem}
\begin{proof}
The para-Sasaki metric $(\eta, \zeta, \phi, g)$ on $\mathbb{S}^{n,n+1}$ has been obtained from the para-Kähler metric $(\g, \p, \ome)$ on $\h_\tau^n$ using the structure of the principal $\mathbb{R}$-bundle $\pi_+: \mathbb{S}^{n,n+1} \to \h_\tau^n$, where the fiber $\mathbb{R}$ acts exactly via the isometric action described previously. From the fact that $\sigma^+$ is horizontal, it follows that the composition $\pi_+ \circ \sigma^+$ is an immersion; moreover, from the fact that $\sigma^+$ is $\phi$-anti-invariant and from Remark \ref{rem:equivalent_anti_invariant} it follows that the composition $\pi_+ \circ \sigma^+$ is Lagrangian with respect to $\ome$, since $\pi^*_+\ome = \mathrm{d}\eta$ (see Proposition \ref{prop:para_Sasaki_principal_bundle}). The metric induced on $\pi_+\big(\sigma^+(M)\big) \subset \h_\tau^n$ coincides exactly with the one induced within $\mathbb{S}^{n,n+1}$ as just explained, and hence it is equal to $\bar{h}$. Finally, during the proof of Lemma \ref{lem:mean_curvature_sigma} we observed that the second fundamental form of $\sigma^+$ in $\mathbb{S}^{n,n+1}$ has no component along the one-dimensional vertical part, which coincides exactly with $\ker(\mathrm d\pi_+)$. This allows us to conclude that the calculation of the mean curvature $\mathbf{H}_{\sigma^+}$ coincides with that of $\mathbf{H}_{\bar\sigma}$, and thus the claim holds.
\end{proof}
\noindent Repeating the same calculations, it is shown that $\varsigma^+(p,t) :=\Lambda_t\cdot\sigma^+(p)= (e^t \xi_p, e^{-t} \nu_p)$ satisfies all the properties of Theorem \ref{prop:immersion_sigma}. Moreover, the induced affine metric $\bar{h}_t$ on the dual immersion $\nu_t := e^{-t} \nu$ is given by $\bar{h}_t = e^t \bar{h}$ (\cite[Proposition 2.5]{nomizu1994affine}), hence $\bar{C}_t = \bar{\nabla} \bar{h}_t = e^t \bar{C}$. This allows us to observe, from Lemma \ref{lem:mean_curvature_sigma}, that the mean curvature $\mathbf{H}_{\varsigma^+}$ of $\varsigma^+$ remains invariant with respect to the variation of $t \in \mathbb{R}$ and coincides exactly with that of $\varsigma^+(\cdot,0) \equiv \sigma^+(\cdot)$.
\noindent In fact what is proven is that, for any fixed $t\in\R$, the immersion $\varsigma^+(\cdot,t): M \to \mathbb{S}^{n,n+1}$ can be projected onto $\h_\tau^n$, giving rise to the same Lagrangian immersion and having the same mean curvature for every $t \in \mathbb{R}$. At the level of the affine immersion, this corresponds to acting by a dilation factor on the transversal vector field $\xi$ and by its inverse on the dual immersion $\nu$. 
\begin{cor}\label{cor:immersion_paracomplex_fixed_t}
For any fixed $t\in\R$, the map $\bar\sigma_t:=\pi_+\circ\varsigma^+(\cdot,t):M\to\h_\tau^n$ satisfies the following properties: \begin{itemize}
    \item[(i)] $\bar\sigma_t$ is Lagrangian;
    \item[(ii)] the induced metric coincides with the affine metric $e^t\bar h$.\end{itemize}
   Moreover, its mean curvature $\mathbf H_{\bar\sigma^t}$ is equal to $\mathbf{H}_{\sigma^+}$, hence $\sigma^+$ is maximal if and only if $\bar\sigma_t$ is maximal, for every fixed $t\in\R$.
\end{cor}
\noindent It will be shown in the next section that the converse also holds, namely that the property of being Lagrangian is the only local obstruction to finding a horizontal lift from $\h_\tau^n$ to $\mathbb{S}^{n,n+1}$, which will give rise to a one-parameter family of affine immersions $f_t$ with their duals $\nu_t$.

\subsection{The inverse problem}\label{sec:inverse_problem}
The goal of this section is to describe the inverse construction with respect to the one developed in Section \ref{sec:immersions_pseudo_spheres}. More precisely, given a non-degenerate Lagrangian immersion $\bar\sigma : M \to \h_\tau^n$, where $M$ is a smooth, connected, and simply connected $n$-manifold, one can prove the existence of a unique horizontal lift $\sigma^+ : M \to \mathbb{S}^{n,n+1}$ which is also $\phi$–anti-invariant, and whose induced metric coincides with that induced by $\bar\sigma$. The latter gives rise to two centroaffine immersions into $\mathbb{R}^{n+1}$ and $(\mathbb{R}^{n+1})^*$, respectively, satisfying the duality property of Definition \ref{def:conormal_map}.
\\ \\ 
Let $M$ be a smooth orientable and simply-connected $n$-manifold, then given a non-degenerate Lagrangian immersion $\bar{\sigma} : M \to \h_\tau^n$, we can always find a lift $\sigma^+=(\xi,\nu) : M \to \mathbb{S}^{n,n+1}$. This is standard and a consequence of the principal $\R$-bundle structure $\pi_+:\mathbb S^{n,n+1}\to\h_\tau^n$ and the fact that $M$ is simply-connected. Any other lift is obtained as $
\sigma_\mu^+ = (e^\mu \xi, e^{-\mu} \nu),
$
where $\mu$ is a smooth real-valued function on $M$, $\xi$ is a smooth vector field on $\mathbb{R}^{n+1}$, and $\big\{\nu(p)\big\}_{p \in M}$ is a family of linear functionals on $\mathbb{R}^{n+1}$. Moreover, by the results presented in \cite{dillen1990conjugate} and \cite{hildebrand2011cross, hildebrand2011half}, we can find a family of centroaffine immersions $f^\mu:M\to\R^{n+1}$ with affine normals $f^\mu=e^{\mu}\xi$, and the same for $\nu^\mu=e^{-\mu}\nu:M\to(\R^{n+1})^*$. It is important to emphasize that the immersion into the dual space is not, a priori, the dual of $f^\mu$ in the sense of Definition \ref{def:conormal_map}. As we shall show, there exists a special choice of $\mu$ for which this is indeed the case. \begin{theorem}\label{thm:inverse_problem_lift}In the framework above, the following statements hold:
 \begin{enumerate}
     \item[(i)] each lift $\sigma_\mu^+ : M \to \mathbb{S}^{n,n+1}$ is $\phi$–anti-invariant, and the induced metric coincides with that induced by $\bar{\sigma}$;
\item[(ii)] there exists a unique function $\hat{\mu}$, up to an additive constant, such that the lift $\sigma_{\hat{\mu}}^+$ is horizontal;
\item[(iii)] the centroaffine immersions $f^{\hat{\mu}}$ and $\nu^{\hat{\mu}}$ induced by $\sigma_{\hat{\mu}}^+$ are dual to each other and unique up to homothety.
 \end{enumerate}
\end{theorem}
\begin{proof}
Fixing a smooth function $\mu : M \to \mathbb{R}$, the differential of the lift $\sigma^\mu$ is 
$$(\sigma_\mu^+)_*(X) = \big( \mathrm{d}\mu(X)e^{\mu}f - e^{\mu}f_*(X), -\mathrm{d}\mu(X)e^{-\mu}\nu + e^{-\mu}\nu_*(X) \big),
$$
where $X \in \Gamma(TM)$. Using Remark \ref{rem:equivalent_anti_invariant}, and the fact that the projection $\bar{\sigma}$ with values in $\h_\tau^n$ is Lagrangian, one obtains the $\phi$–anti-invariance. In fact, \begin{align*}
\mathrm d\eta\big((\sigma_\mu^+)_*(X),(\sigma_\mu^+)_*(Y)\big)&=\big((\pi_+)^*\ome\big)((\sigma_\mu^+)_*(X),(\sigma_\mu^+)_*(Y)\big) \\ &=\ome\big(\bar\sigma_*(X),\bar\sigma_*(Y)\big) \\ &=0.
\end{align*}Moreover, the same applies to the symplectic form $\hat\ome$ defined on $\R^{n+1,n+1}$, which gives in the case $\mu\equiv 0$: $$\nu_*(X)\big(f_*(Y)\big)=\nu_*(Y)\big(f_*(X)\big).$$ We now aim to determine a smooth function $\hat\mu : M \to \mathbb{R}$ satisfying the conditions
$$
\hat{\g}\big((\sigma_{\hat\mu}^+)_*(X), \sigma_{\hat\mu}^+\big) = \hat{\g}\big((\sigma_{\hat\mu}^+)_*(X), \hat\p(\sigma_{\hat\mu}^+)\big) = 0,
$$
for all $X \in \Gamma(TM)$. Recall that $\hat{\g}$ and $\hat{\p}$ denote, respectively, the pseudo-Riemannian bilineare form of signature $(n+1,n+1)$ and the para-complex structure on $\mathbb{R}^{n+1,n+1}$ (see the discussion after Remark \ref{rem:equivalent_anti_invariant} for further details). In other words, we seek for a function $\hat\mu$ so that the lift $\sigma_{\hat\mu}^+$ is horizontal with respect to the distribution $\mathcal H\subset T\mathbb S^{n,n+1}$ induced by the contact $1$-form $\eta$. From the expression of the differential $(\sigma_{\hat\mu}^+)_*(X)$, one obtains
$$
\hat{\g}\big((\sigma_{\hat\mu}^+)_*(X), \sigma_{\hat\mu}^+\big) = \frac{1}{2}\big(\nu_*(X)(f) + \nu(f_*(X))\big),
$$
which vanishes identically, since it corresponds to the derivative along the vector field $X$ of the relation $\nu(\xi) = 1$. Regarding the second term, with a similar computation, we get $$\hat{\g}\big((\sigma_{\hat\mu}^+)_*(X), \hat\p(\sigma_{\hat\mu}^+)\big) = 0 \iff -\mathrm d\hat\mu(X)+\nu_*(X)(f)=0, $$ for any $X\in\Gamma(TM)$. Let us define the $1$-form $\alpha(X):=\nu_*(X)(f)$ on $M$, then the lift $\sigma_{\hat\mu}^+$ is horizontal if and only if $\alpha=\mathrm d\hat\mu$. Let us first compute the differential of $\alpha$ by using the standard formula $\mathrm d\alpha(X,Y)=X\cdot\big(\alpha(Y)\big)-Y\cdot\big(\alpha(X)\big)-\alpha\big([X,Y]\big)$. According to its definition we have \begin{align*}
   &X\cdot\big(\alpha(Y)\big)=D_X^*\big(\nu_*(X)\big)(f)+\nu_*(Y)\big(D_Xf\big) \\ &\qquad\qquad\quad=\nu_*\big(\bar{\nabla}_XY\big)(f)-\bar h(X,Y)\nu(f)-\nu_*(Y)\big(f_*(X)\big) \ ; \\ &Y\cdot\big(\alpha(X)\big)=\nu_*\big(\bar{\nabla}_YX\big)(f)-\bar h(Y,X)\nu(f)-\nu_*(X)\big(f_*(Y)\big) \ ; \\ &\alpha\big([X,Y]\big)=\nu_*\big([X,Y]\big)(f) \ ;
\end{align*}where we used the structure equations (\ref{eq:structure_equations_f}),(\ref{eq:structure_equations_nu}) defining $f$ and $\nu$ in the case $f$ is centroaffine. Since $\bar{\nabla}$ is torsion-free, by summing all the previously obtained terms and deleting the equal ones, we arrive at the following expression:
$$
\mathrm{d}\alpha(X,Y)=-\nu_*(Y)\big(f_*(X)\big)+\nu_*(X)\big(f_*(Y)\big).
$$ The last term is zero for any $X,Y\in\Gamma(TM)$ as a consequence of $\bar\sigma$ being Lagrangian in $\h_\tau^n$ and the discussion at the beginning of the proof. Therefore, $\alpha$ is a closed $1$-form on the simply-connected manifold $M$, hence there exists a smooth function $\lambda:M\to\R$ such that $\alpha=\mathrm d\lambda$. To conclude the proof of point $(ii)$, we observe that by choosing $\hat{\mu} = \lambda$, the lift $\sigma_{\hat{\mu}}^+$ is horizontal, and such a function is unique up to an additive constant. Point $(iii)$ follows directly from the construction already described in Section \ref{sec:immersions_pseudo_spheres} and from the fact that $\nu^{\hat{\mu}}$ is precisely the dual immersion of $f^{\hat{\mu}}$, as the following computation shows:
\begin{align*}\nu^{\hat\mu}\big(f_*^{\hat\mu}(X)\big)=e^{-\hat\mu}\nu\big(\mathrm d\hat\mu(X)e^{\hat\mu}f+e^{\hat\mu}f_*(X)\big)=\mathrm d\hat\mu(X)+\nu\big(f_*(X)\big)=0.
\end{align*}
\end{proof}

\begin{cor}\label{cor:inverse_sigma_maximal_iff_proper_affine_sphere}
Given a non-degenerate Lagrangian immersion $\bar\sigma:M\to\h_\tau^n$, the following are equivalent: \begin{itemize}
    \item[(i)] $\bar\sigma$ is maximal;
    \item[(ii)] The horizontal lift $\sigma^+_{\hat\mu}:M\to\mathbb S^{n,n+1}$, or equivalently $\sigma^-_{\hat\mu}:M\to\mathbb H^{n+1,n}$, is maximal;
    \item[(iii)] the affine immersion $f^{\hat\mu}$ and its dual $\nu^{\hat\mu}$ are proper affine spheres.
\end{itemize}
\end{cor}
\begin{proof}
If we start with a Lagrangian and maximal immersion $\bar\sigma$ in $\mathbb{H}_\tau^n$, then its horizontal lift $\sigma^+_{\hat\mu}$ still has zero mean curvature, since horizontal immersions in $\mathbb{S}^{n,n+1}$ have second fundamental form taking values in $\mathbf{P}(TM)$, and it coincides with that of $\bar\sigma$ (see the proof of Lemma \ref{lem:mean_curvature_sigma}). Conversely, if we only know that $\sigma^+_{\hat\mu}$ is maximal, then Corollary \ref{cor:sigma_maximal_iff_proper_affine_sphere} implies that $\nu^{\hat\mu}$ is a proper affine sphere; its dual, which now coincides with $f^{\hat\mu}$ since it is centroaffine (Remark \ref{rem:dual_of_affine_sphere}), must also be a proper affine sphere. Finally, if both $f^{\hat\mu}$ and $\nu^{\hat\mu}$ are proper affine spheres, then from the computations in Lemma \ref{lem:mean_curvature_sigma} and Corollary \ref{cor:sigma_maximal_iff_proper_affine_sphere} we deduce that $\bar\sigma$ is maximal.
\end{proof}
\subsection{Examples}

\subsubsection{The case $n=1$}\label{sec:examples}\hfill\vspace{0.3em}\\
The first interesting case is $n=1$, namely, on the one hand we have affine curves $f:M\to\mathbb R^2$ and, on the other, immersions $\sigma^+$ and $\sigma^-$ into $\mathbb S^{1,2}$ and $\mathbb H^{2,1}$, respectively. The pseudohyperbolic space $\mathbb H^{2,1}$ is an isometric model of the so-called $3$–dimensional Anti–de Sitter space, denoted $\mathrm{AdS}^3$. It can be introduced in various isometric realizations and has a deep connection with (universal) Teichmüller theory (see \cite{bonsante2020anti} for a survey on the subject). In this very special case, we may assume that hyperbolic or elliptic affine spheres are, up to affine transformations, a hyperbola asymptotic to the bisectors of the plane and an ellipse, respectively. Moreover, if $f$ is an hyperbolic or elliptic affine sphere, the induced map $\sigma^-:M\to \mathrm{AdS}^3$ is respectively spacelike or timelike. Indeed, in this case $\sigma^-=(-f,\nu)$ and $D_X f=-f_*(X)$ when $f$ is an elliptic affine sphere; therefore, by retracing the computations of Theorem \ref{prop:immersion_sigma}, we obtain that the induced metric is negative definite. Furthermore, $\sigma^-$ is maximal, and in dimension $1$ this is equivalent to saying that $\sigma^-$ is a spacelike or timelike geodesic, depending on the chosen proper affine sphere. Notice that timelike geodesics in \(\mathrm{AdS}^3\) are closed and never reach the boundary at infinity \(\partial_\infty \mathrm{AdS}^3\), whereas spacelike ones intersect the boundary in two distinct points which, in light of Lemma \ref{lem:buondary_pseudosphere}, proven in the next section, correspond exactly to the two endpoints of the convex segment \(\Omega\) on which the hyperbolic affine sphere is asymptotic.
\\ \\ Another interesting phenomenon in this case is that, after applying the isometric action of \(\mathbb{R} \cong \mathrm{SO}_0(1,1) < \mathrm{SO}_0(n+1,n+1)\) introduced in (\ref{eq:matrix_invertible_elements}) to the immersion 
\(\sigma^-\), we obtain an immersion \(\varsigma^- : \mathbb{R}^2 \to \mathrm{AdS}^3\) which is spacelike, complete, 
and maximal (see Section \ref{sec:immersions_pseudo_spheres}). Moreover, by construction, its Gauss curvature is identically zero at every point of the 
surface. In particular, we recover the unique flat, spacelike, maximal surfaces in 
\(\mathrm{AdS}^3\), studied by Barbot (\cite{barbot2015deformations}), whose boundary at infinity is a polygon with four vertices explicitly 
described in terms of the boundary of the convex domain \(\Omega\), as explained in Proposition \ref{prop:boundary_immersion_into_paracomplex_hyperbolic}.

\subsubsection{The case $n=2$}\hfill\vspace{0.3em}\\
Given a hyperbolic affine sphere $f: M \to \mathbb{R}^3$ with positive definite Blaschke metric $h$, it has been previously shown (see \cite{hildebrand2011cross} and \cite{RT_bicomplex}) that this corresponds to a spacelike\footnote{In our case $\bar\sigma$ is timelike since we are considering the anti-isometric model of $\h_\tau^n$} maximal Lagrangian surface $\bar{\sigma}: M \to \mathbb{H}_\tau^2$, whose induced metric coincides exactly with the Blaschke metric (see also Theorem \ref{prop:immersion_in_H_tau}). In terms of the unique horizontal timelike and $\phi$-anti-invariant immersion $\sigma^+: M \to \mathbb{S}^{2,3}$, we know that it is anti-isometric to a spacelike surface $\sigma^-: M \to \mathbb{H}^{3,2}$ with the same properties (see Remark \ref{rem:equivalent_anti_invariant}). By looking at the totally geodesic copy of $\mathbb{H}^{3,2}$ into $\mathbb{H}^{4,2}$ and thinking of the previously described surface $\sigma^-$ as taking values in $\mathbb{H}^{4,2}$, we recover a particular case of the alternating holomorphic curves studied by Collier–Toulisse (\cite[\S 5.5]{collier2023holomorphic}) whenever the Pick tensor $C$ (see (\ref{eq:tensoreC})) of the hyperbolic affine sphere $f$ does not vanish. In their setting, $\mathbb{H}^{4,2}$ is endowed with the natural almost complex structure induced by the algebra of split octonions, and is therefore viewed as a homogeneous space for the split real form $\mathrm G_2'<\SO_0(4,3)$ of the complex simple group $\mathrm G_2^{\mathbb{C}}$. In our construction, the copies of $\mathbb{R}^3$ and $(\mathbb{R}^3)^*$ in which the hyperbolic affine sphere and its dual respectively lie correspond to two transverse isotropic $3$-planes in $\mathbb{R}^{3,3}$. Using the construction with split octonions and the two totally geodesic inclusions $\mathbb{R}^{3,3} \hookrightarrow \mathbb{R}^{4,3}$ and $\mathbb{H}^{3,2} \hookrightarrow \mathbb{H}^{4,2}$, the three-dimensional distributions $\mathcal{D}^+$ and $\mathcal{D}^-$ introduced in \cite[\S 3.3]{collier2023holomorphic} correspond precisely to the isotropic $3$-planes just described. In particular, the construction of Collier–Toulisse allows one to explicitly describe a hyperbolic affine sphere in $\mathcal{D}^+$ and its dual in $\mathcal{D}^-$ starting from a non–totally geodesic alternating holomorphic curve in $\mathbb{H}^{3,2} \subset \mathbb{H}^{4,2}$, which coincides precisely with our spacelike, maximal, horizontal and $\phi$-anti-invariant immersion.

\subsubsection{Pseudo-flats}\hfill\vspace{0.3em}\\
An interesting example of hyperbolic affine spheres in every dimension is given by the hypersurface in \(\mathbb{R}^{n+1}\) parametrized as
\[
\mathcal{F} := \big\{(x_1,\dots,x_{n+1}) \in \mathbb{R}^{n+1} \ \big| \ \prod_{i=1}^{n+1} x_i = 1,\ x_i > 0 \ \text{for } i=1,\dots,n\big\}.
\]
Initially studied by \c{T}i\c{t}eica and later proved by Calabi (\cite{calabi1972complete}), the hypersurface \(\mathcal{F}\) is a hyperbolic affine 
sphere whose affine metric is Riemannian with zero sectional curvature, and whose Pick tensor is everywhere 
non-vanishing. The boundary of the cone to which it is asymptotic coincides with the first orthant in $\R^{n+1}$. Using our construction from Section \ref{sec:immersions_pseudo_spheres}, we obtain an immersion 
\(\sigma^- : \R^n \to \mathbb{H}^{n+1,n}\) which is spacelike, maximal, horizontal, and 
\(\phi\)-anti-invariant. After applying the \(\mathbb{R}\)-action by isometries to \(\sigma^-\), the resulting 
immersion \(\varsigma^- : \mathbb{R}^{n+1} \to \mathbb{H}^{n+1,n}\) gives rise to what have been called 
\emph{pseudo-flat} submanifolds, recently studied in \cite{moriani2024rigidity}. These are flat \((n+1)\)-dimensional 
submanifolds of \(\mathbb{H}^{n+1,n}\) that arise as orbits of the action of the Cartan subgroup of 
\(\mathrm{SO}_0(n+1,n+1)\) and generalize the Barbot surface in $\mathbb H^{2,1}\cong\AdS^3$ (see \cite[\S 3.3]{moriani2024rigidity} for more details).

\section{Applications}
\subsection{A boundary problem}\label{sec:boundary_problem}
The goal of this section is to show how the main construction presented in Section \ref{sec:main_construction} finds an application to the problem of existence of maximal submanifolds in 
\(\hat{\mathbb{H}}^{n+1,n}\) and $\hat{\mathbb H}^n_\tau$ with prescribed boundary data. More precisely, by introducing a notion of a \emph{hyperplane boundary set}
\(\Lambda_\Omega \subset \partial_\infty \hat{\mathbb{H}}^{n+1,n}\) arising from a strictly convex domain 
\(\Omega \subset \mathbb{RP}^n\), we first prove the existence of a $1$-parameter family $\{\sigma_t\}_{t\in\R}$ of \(n\)-dimensional spacelike, maximal, complete, 
horizontal, and \(\phi\)-anti-invariant submanifolds in $\hat{\mathbb H}^{n+1,n}$ whose boundary at infinity gives rise to a $1$-parameter family $\{\Lambda_\Omega^t\}_{t\in\R}$ of subsets in $\Lambda_\Omega$. Subsequently, we consider the projection $\overline{\Lambda}_\Omega$ of $\Lambda_\Omega$ onto the partial flag manifold of lines and hyperplanes in $\mathbb{R}^{n+1}$, which coincides with $\partial_\infty \hat{\mathbb H}^n_\tau$.
This allows us to prove the existence and uniqueness of an $n$-dimensional spacelike, maximal, and Lagrangian submanifold in $\hat{\mathbb H}^n_\tau$ whose boundary at infinity is equal to $\overline{\Lambda}_\Omega$. The maximal submanifolds in \(\hat{\mathbb{H}}^{n+1,n}\) studied in the first case differ from those previously studied in the literature (\cite{bonsante2010maximal, andersson2012cosmological, CTT, labourie2024plateau, seppi2023complete}). 
\\ \\
\noindent In this section we will work primarily with the boundary at infinity of the projective model of the pseudohyperbolic space, which is described by
\[
\partial_\infty\hat{\mathbb H}^{\,n+1,n}
=\{[v,\varphi]\in\mathbb P\big(\R^{n+1}\oplus(\R^{n+1})^*\big) \ \mid \ \varphi(v)=0\},
\]
and, in this specific case, it coincides with the boundary at infinity of the pseudosphere \(\mathbb S^{n,n+1}\). This description follows from the bi-linear form \(\hat\g\) of signature $(n+1,n+1)$ on \(\R^{n+1}\oplus(\R^{n+1})^*\) introduced in Section \ref{sec:immersions_pseudo_spheres}. In what follows, we denote the space $\R^{n+1}\oplus(\R^{n+1})^*$ by $V$. Recall that, given a strictly convex subset $\Omega \subset \mathbb{RP}^n$, one can consider its lift to the space $\R^{n+1}\setminus\{0\}$ given by the full cone 
\(
\mathcal C(\Omega) = \mathcal C_+(\Omega) \sqcup \mathcal C_-(\Omega)
\)
(see Section \ref{sec:convex_projective_manifolds}).
\begin{defi}
Let $\Omega\subset\R\mathbb P^n$ be any strictly convex subset, then the set
\[
\Lambda_\Omega := \{ [v, \varphi_v] \in \mathbb P(V) \mid v \in \partial \mathcal C(\Omega), \ [\varphi_v] \ \text{unique supporting hyperplane at} \ [v] \},
\] is called a \emph{hyperplane boundary set}. The elements $[v],[\varphi_v]$ represent the projective classes in $\R\mathbb P^n,(\R\mathbb P^n)^*$, respectively.
\end{defi}
\noindent There is another geometric way to construct the subset $\Lambda_\Omega$, which also makes its topology clearer and sheds light on its definition. We denote by $\mathrm{Ein}_0^{n,n}$ the subspace of $\partial_\infty \hat{\mathbb H}^{n+1,n}$ consisting of pairs $[v,\varphi]\in\mathbb{P}(V)$ such that $\varphi(v)=0$, $v\neq 0$, and $\varphi\neq 0$. This subspace admits a projection onto the partial flag variety
$$
\mathcal F_{1,n}(\R^{n+1}) :=\big\{([v],[\varphi])\in\mathbb{RP}^n\times(\mathbb{RP}^n)^* \mid \varphi(v)=0\big\},
$$
which consists of pairs made of a line and a hyperplane in $\R^{n+1}$, with the line contained in the hyperplane. Indeed, the projection is simply defined as
\begin{align*}
p: \ &\mathrm{Ein}_0^{n,n} \longrightarrow \mathcal F_{1,n}(\R^{n+1}) \\ & [v,\varphi]\longmapsto ([v],[\varphi]).
\end{align*}
Since $\Omega$ is strictly convex, we obtain a copy of $\partial\Omega$ embedded homeomorphically in $\mathcal F_{1,n}(\R^{n+1})$ as the graph of the bijection $G:\partial\Omega\to\partial\Omega^*$ defined by $G([v]) := [\varphi_v]$, where $[\varphi_v]$ represents the unique supporting hyperplane at $[v]$. The preimage of $\mathrm{Graph}(G)\cong\partial\Omega$ under the projection $p$ yields exactly the subset $\Lambda_\Omega$ inside $\mathrm{Ein}_0^{n,n}$. In particular, by restricting the projection $p$ to the subset $\Lambda_\Omega$ we obtain a surjective map
$
p:\Lambda_\Omega \to \partial\Omega,
$
whose fiber over a point $[v]$ is given by all linear functionals that are real multiples of $\varphi_v$, and therefore generate the same supporting hyperplane at $[v]$. We therefore obtain a homeomorphism between $\Lambda_\Omega$ and $\partial\Omega \times \R$. \\ \\
\noindent Specializing the construction of Section \ref{sec:immersions_pseudo_spheres} to the case where the affine immersion $f:M\to\mathbb{R}^{n+1}$ is convex and, more specifically, a hyperbolic affine sphere, we obtain a maximal, horizontal and $\phi$–anti-invariant map $
\sigma^- = (-f, \nu) : M \to \mathbb{H}^{n+1,n},$
whose induced metric is positive definite and coincides with the Blaschke metric of the hyperbolic affine sphere (Theorem \ref{prop:immersion_sigma}). As a consequence of Theorem \ref{thm:asymptotic_hyperbolic_sphere}, every such affine immersion $f$ is asymptotic to a sharp convex cone $\mathcal{C}_+(\Omega)$ over a properly convex set $\Omega \subset \mathbb{RP}^n$, or equivalently, the induced affine metric is complete. In particular, since the dual immersion is also a hyperbolic affine sphere in $(\mathbb{R}^{n+1})^*$, the same description applies to $\nu$, with dual convex set $\Omega^*$ and cone $\mathcal{C}_+^*(\Omega)$. By an abuse of notation, we shall continue to denote by $\sigma^-$ the map induced in the projective model $\hat{\mathbb{H}}^{n+1,n}$. 
Recall that there is an $\R$-action on $\hat{\mathbb H}^{n+1,n}$ given by $\Psi_t\big([v,\varphi]\big):=[e^tv,e^{-t}\varphi]$ for $t\in\R$. In particular, since $[e^tv,e^{-t}\varphi]\in\mathbb P(V)$ it defines the same point as $[v,e^{-2t}\varphi]\in\mathbb P(V)$. For any fixed $t\in \R$ we can consider the immersion $\sigma_t^-:=\Psi_t\circ\sigma^-:M\to\hat{\mathbb H}^{n+1,n}$ defined as $\sigma_t^-(p)=[-f(p),e^{-2t}\nu(p)]$. In the following, we want to understand the structure of the boundary at infinity of these immersions and relate it to the hyperplane boundary subset introduced above. Recall that the boundary at infinity of any immersion \(\sigma:M\to\hat{\mathbb H}^{n+1,n}\) is defined as \begin{equation}
    \partial_\infty\sigma:=\overline{\sigma(M)}^{\mathbb P}\cap\partial_\infty\hat{\mathbb H}^{n+1,n}
\end{equation} where $\overline{\sigma(M)}^{\mathbb P}$ denotes the projective closure of $\sigma(M)$ inside the projective model of the pseudohyperbolic space $\hat{\mathbb H}^{n+1,n}\subset\R\mathbb P^{2n+1}$. 
\begin{lemma}\label{lem:buondary_pseudosphere}
Let $f:M\to\R^{n+1}$ be a hyperbolic affine sphere asymptotic over a strictly convex subset $\Omega\subset\mathbb{RP}^n$ and let $\sigma^-_t=[-f,e^{-2t}\nu]:M\to\hat{\mathbb H}^{n+1,n}$ be the $1$-parameter family of induced immersions. Then the boundary at infinity of $\sigma^-_t$, for $t\in\R$ fixed, gives rise to a subset $\Lambda_\Omega^t$ of the hyperplane boundary set $\Lambda_\Omega$.
\end{lemma}
\begin{proof}
Let us first recall the relation between $f$ and $\nu$, namely that $\nu_p(f(p))=1$ for any $p\in M$ and $\nu\big(f_*(X)\big)=0$ for any $X\in\Gamma(TM)$ (see Definition \ref{def:conormal_map}). In other words, the vector $f(p)$ is transverse to the hyperplane $f_*(T_pM)$, which is identified with $\Ker \ \nu_p$ for any $p\in M$. Let us denote by $\pi_{\mathbb P}:\R^{n+1}\setminus\{0\}\to\R\mathbb P^n$ the map defining the projective class of a non-zero vector. Since $f$ is asymptotic to the sharp convex cone $\mathcal{C}_+(\Omega)$ we know that for any $[v]\in\pi_{\mathbb P}\big(\partial\mathcal C(\Omega)\big)\cong\partial\Omega$ there exists a sequence of points $\{p_k\}\subset M$ leaving every compact sets and scalars $\lambda_k\in\R\setminus\{0\}$ such that $\lim_k\lambda_k\cdot f(p_k)=v$. Viceversa, for every $\{p_k\}\subset M$ as above, the projective limit of $\{f(p_k)\}$ takes value in $\pi_{\mathbb P}\big(\partial\mathcal C(\Omega)\big)$. Moreover, the same discussion applies to $\nu$ and $\pi_\mathbb P^*:(\R^{n+1})^*\setminus\{0\}\to(\R\mathbb P^n)^*$ by Theorem \ref{thm:asymptotic_hyperbolic_sphere} and Proposition \ref{prop:isometry_dual_immersion}. Let us denote by $[f_\infty]$ and $[\nu_\infty]$ the limit points in $\pi_{\mathbb P}\big(\partial\mathcal C(\Omega)\big)$ and $\pi_{\mathbb P}^*\big(\partial\mathcal C(\Omega^*)\big)$, respectively. In particular, any representative in $[f_\infty]$ is  a vector in $\partial\mathcal C(\Omega)$ and we will show that any representative in $[\nu_\infty]$ gives rise to the same hyperplane in $\R^{n+1}$ containing the line generated by $f_\infty$. In fact, under the identification between $(V,\hat\g)$ and $(\R^{n+1,n+1}, b)$ using the standard scalar product in $\R^{n+1}$ (see Section \ref{sec:immersion_para-complex_hyperbolic}), we know that $\vl\vl f(p_k)\vl \vl\to\infty$ and $\vl\vl\nu_{p_k}\vl\vl_*\to\infty$ being both proper map by Theorem \ref{thm:asymptotic_hyperbolic_sphere}. Together with the previous considerations, this yields $\nu_\infty(f_\infty)=0$, i.e. in the limiting configuration $\Ker \ \nu_\infty$ is a hyperplane containing the vector $f_\infty$. Equivalently, for any fixed $t\in\mathbb{R}$, the boundary at infinity of $\sigma_t^-$ is formed by all projective classes $[-f_\infty,e^{-2t}\nu_\infty]\in\mathbb P(V)$, which then determine a subset $\Lambda_\Omega^t\subset\Lambda_\Omega$.
\end{proof}
\begin{remark}
An interesting feature of this construction is that the projective model of the pseudosphere $\hat{\mathbb S}^{n,n+1}$ and the pseudohyperbolic space $\hat{\mathbb H}^{n+1,n}$ share the same boundary at infinity, which we denoted by $\mathrm{Ein}^{n,n}$ (see Section \ref{sec:hyperbolic_generale_p_q}). Indeed, if we consider the vector space $V$ endowed with the bilinear form $\hat\g$ introduced in Section \ref{sec:immersions_pseudo_spheres}, we obtain the decomposition
$
V = V_{\hat\g>0}\sqcup V_{\hat\g=0}\sqcup V_{\hat\g<0},
$
according to the sign of $\hat\g$. Projectivising gives
$
\mathbb P(V)=\hat{\mathbb H}^{n+1,n}\sqcup \mathrm{Ein}^{n,n}\sqcup \hat{\mathbb S}^{n,n+1}.
$
From the proof of Lemma \ref{lem:buondary_pseudosphere}, it follows that the boundaries at infinity of $\sigma^-_t$ and $\sigma^+_t$ define the same subset of the Einstein universe $\mathrm{Ein}^{n,n}$. 
\end{remark}
\noindent Each set $\Lambda_\Omega^t$ is homeomorphic to a copy of $\partial\Omega \times \{t\}$ as a subset of $\Lambda_\Omega$. In particular, this shows that as $t \in \R$ varies, the immersions $\sigma_t^-$ all have distinct boundaries at infinity, which together generate a foliation of the hyperplane boundary set $\Lambda_\Omega$. The key point is that the functionals $e^{-2t}\nu_\infty$ define the same supporting hyperplane for $f_\infty$ as $t \in \R$ varies, but the projective classes $[f_\infty, e^{-2t}\nu_\infty] \in \mathbb P(V)$ are different. \\ \\
 \noindent Similarly, one can address a corresponding problem for the para-complex hyperbolic space, which also admits a natural boundary at infinity defined as
\[
\partial_\infty \hat{\mathbb H}_\tau^n := \{ ([v], [\varphi]) \in \mathbb{RP}^n \times (\mathbb{RP}^n)^* \mid \varphi(v) = 0 \},
\] so that it can be identified with the partial flag viariety $\mathcal{F}_{1,n}(\R^{n+1})$ introduced above. Notice that the whole $1$-parameter family of immersions $\{\sigma_t^-\}_{t\in\R}$ induces the same spacelike Lagrangian immersion $\bar\sigma$ into $\hat{\mathbb H}_\tau^n$ (see Section \ref{sec:immersion_para-complex_hyperbolic}). At this point, however, it is clear that the results in $\partial_\infty \hat{\mathbb H}_\tau^n$ follow directly from the work done for the pseudohyperbolic space.
\begin{cor}\label{cor:boundary_map_in_paracomplex_hyperbolic}
The boundary at infinity of the induced map $\bar\sigma:M\to\hat\h_\tau^n$ is described as $$\partial_\infty\bar\sigma=\big\{([v],[\varphi_v])\in\R\mathbb P^n\times(\R\mathbb P^n)^* \ | \ [v]\in\partial\Omega, \ [\varphi_v] \ \text{supporting hyperplane in} \ [v] \big\}.$$ In particular, since $\Omega$ is strictly convex then $\partial_\infty\bar\sigma$ is equal to the graph of the bijection $
G: \partial\Omega \to \partial\Omega^*, \ G([v]) := [\varphi_v]
$.
\end{cor}
\noindent Recall we have a projection $p:\mathrm{Ein}_0^{n,n}\subset\partial_\infty\hat{\mathbb H}^{n+1,n}\to\partial_\infty\hat{\mathbb H}^n_\tau$ and, as explained at the beginning of the section, every hyperplane boundary set \(\Lambda_\Omega\) maps to a subset $\bar{\Lambda}_\Omega$ homeomorphic to $\partial\Omega$. This allows us to solve a boundary value problem with prescribed boundary data in the partial flag variety $\mathcal{F}_{1,n}(\R^{n+1})$.
\begin{theorem}\label{cor:plateau_problem_paracomplex_hyperbolic}
Given the subset \(p\big(\Lambda_\Omega\big)=\bar{\Lambda}_\Omega\subset\partial_\infty\hat{\mathbb H}^n_\tau\cong\mathcal{F}_{1,n}(\R^{n+1})\) defined from a strictly convex subset $\Omega\subset\mathbb{RP}^n$, there exists a unique spacelike, maximal, Lagrangian, and complete \(n\)-submanifold in $\hat{\mathbb H}^n_\tau$ whose boundary at infinity is equal to \(\bar{\Lambda}_\Omega\).
\end{theorem}
\begin{proof}
Let $\Omega$ be the strictly convex subset of $\R\mathbb P^n$ defining the hyperplane boundary set $\Lambda_\Omega$. By Theorem \ref{thm:asymptotic_hyperbolic_sphere} we know there exists a unique hyperbolic affine sphere $f:M\to\R^{n+1}$ which is asymptotic to the boundary of the sharp convex cone $\mathcal C_+(\Omega)$ over $\Omega$. If we consider the induced immersion $\sigma^- = (-f, \nu) : M \to \hat{\mathbb H}^{\,n+1,n}$, then, thanks to Theorem \ref{prop:immersion_sigma} and Corollary \ref{cor:sigma_maximal_iff_proper_affine_sphere}, we know that $\sigma^-$ is spacelike (the induced metric is precisely the Blaschke metric $h$ of $f$), maximal, horizontal, and $\phi$–anti-invariant. Moreover, since $f$ is proper and $h$ is complete (\cite{cheng1977regularity, cheng1986complete}), by post-composing with the projection $\hat{\pi}^-:\hat{\mathbb H}^{n+1,n} \to \hat{\mathbb H}^n_\tau$, which gives the structure of a principal $\R$-bundle, we obtain an immersion
$
\bar{\sigma} = \hat{\pi}^- \circ \sigma^- : M \to \hat{\mathbb H}^n_\tau
$ that is spacelike, Lagrangian, and complete (see Section \ref{sec:immersion_para-complex_hyperbolic}). Using Corollary \ref{cor:boundary_map_in_paracomplex_hyperbolic}, one directly obtains that
$
\partial_\infty \bar{\sigma} = \overline{\Lambda}_\Omega.
$
Regarding the uniqueness, suppose there exists another spacelike, maximal, Lagrangian immersion $\bar{\sigma}':M\to\hat{\mathbb H}^n_\tau$ such that $\partial_\infty\bar\sigma'=\partial_\infty\bar\sigma=\mathrm{graph}(G)$, where $G:\partial\Omega\to\partial\Omega^*$ is the bijection defined as in Corollary \ref{cor:boundary_map_in_paracomplex_hyperbolic}. The immersion $\bar\sigma'$ corresponds to an hyperbolic affine sphere $f':M\to\R^{n+1}$ (see Corollary \ref{cor:sigma_maximal_iff_proper_affine_sphere} and Theorem \ref{thm:inverse_problem_lift}) which, by construction, it has to be asymptotic to the cone over $\Omega$. By the uniqueness result of Theorem \ref{thm:asymptotic_hyperbolic_sphere} this implies $f(M)=f'(M)$ and hence $\bar\sigma(M)=\bar\sigma'(M)$.
\end{proof}
\noindent Another interesting direction of this construction is to compute the boundary at infinity of the immersion obtained after applying the isometric flow generated by the $\R$-action introduced in (\ref{eq:matrix_invertible_elements}). In other words, if we consider
$
\varsigma^-(p,t) := (-e^t f(p), e^{-t} \nu_p),
$ we obtain an immersion $\varsigma^-: M \times \R \to \hat{\mathbb{H}}^{n+1,n}$ which is still spacelike, but no longer horizontal (see Definition \ref{def:horizontal_phi_anti_invariant}). Before stating the result, we recall that we can view $\overline{\Omega}$ and $\overline{\Omega}^*$ as subspaces of $\mathbb{P}(V)$ via the identifications
$
\overline{\Omega} \simeq \mathbb{P}\big(\mathcal{C}(\overline{\Omega}) \times \{\ast\}\big)$ and
$\overline{\Omega}^* \simeq \mathbb{P}\big(\{\ast\} \times \mathcal{C}(\overline{\Omega}^*)\big).
$
\begin{prop}\label{prop:boundary_immersion_into_paracomplex_hyperbolic}
In the setting above, the boundary $\partial_\infty\varsigma^-$ can be identified with $\Lambda_\Omega \cong \partial\Omega \times \R$, with two copies of $\overline{\Omega}$ and $\overline{\Omega^*}$ attached to $\Lambda_\Omega$ at the two ends at infinity. In particular, $\partial_\infty\varsigma^-$ is homeomorphic to $\mathbb S^n$.
\end{prop}
\begin{proof}
The part in the boundary made of $\Lambda_\Omega$ comes from the calculations of Lemma \ref{lem:buondary_pseudosphere}. In fact, for any fixed $t\in\R$, the boundary of the immersion $\sigma_t^-(\cdot)=\varsigma^-(\cdot,t)$ it can be parametrized as the set $\Lambda_\Omega^t$, whose union as $t$ varies form the whole hyperplane boundary set $\Lambda_\Omega$. As for the remaining part, we need to compute the projective limit of the points $(-e^t f, e^{-t} \nu)$ as $t \to +\infty$ or $t \to -\infty$. In the first case, multiplying by $e^{-t}$ gives the projective limit $(-f,0)$; in the second, multiplying by $e^{t}$ yields $(0,\nu)$.
In other words, the projective closure of $\varsigma^-\big(M \times \R\big)$ also contains the points of the form
$
\{(-f(p),0) \mid p \in M\} \ \text{and} \ \{(0,\nu(p)) \mid p \in M\}.
$
By the construction in Theorem \ref{thm:asymptotic_hyperbolic_sphere} (see also the subsequent discussion), these points, and their projective limits along a sequence $\{p_k\}\subset M$, correspond respectively to $\overline\Omega$ and $\overline\Omega^*$. Finally, since $\Omega$ and $\Omega^*$ are strictly convex, we are topologically considering two closed $n$-dimensional discs glued at the two ends at infinity of $\Lambda_\Omega\cong\partial\Omega\times\R$ which gives an homeomorphic copy of $\mathbb S^n$.
\end{proof}

\subsection{Harmonic maps in the symmetric spaces}\label{sec:5}
\noindent In this section we will mainly study two maps, $\mathcal G_f$ and $\mathcal G_{\sigma^+}$, constructed from the affine immersion $f:M\to\mathbb R^{n+1}$ and $\sigma^+=(f,\nu):M\to\mathbb S^{n,n+1}$, with values in the symmetric spaces of $\SL(n+1,\mathbb R)$ and $\SO_0(n+1,n+1)$, respectively. The main result will be to show that when $f$ is a hyperbolic affine sphere with positive definite Blaschke metric, then the induced map $\mathcal G_f$ into the symmetric space is harmonic. Using then the inclusion $\SL(n+1,\mathbb R)\hookrightarrow\SO_0(n+1,n+1)$ described in Section \ref{sec:Lie_groups_inclusion}, we will obtain as a corollary that the map $\mathcal G_{\sigma^+}$ into the symmetric space of $\SO_0(n+1,n+1)$ is also harmonic.
\subsubsection{Harmonic maps}\label{sec:harmonic_maps}\hfill\vspace{0.3em}\\
We now briefly introduce the necessary definitions for the theory of harmonic maps between (pseudo)-Riemannian manifolds, together with a lemma that will be useful later on. In the literature, there are numerous references concerning harmonic maps, but we shall primarily rely on the book (\cite{eells1983selected}). \\ \\ Let $(N_1,g_1)$ be a Riemannian manifold and $(N_2,g_2)$ a pseudo-Riemannian manifold of dimensions $n_1$ and $n_2$, respectively. We denote by $\nabla^{g_i}$ the corresponding Levi-Civita connections for $i=1,2$. Given a smooth map $\psi:(N_1,g_1)\to(N_2,g_2)$, one can regard $\mathrm d\psi$ as a smooth section of $T^*N_1\otimes\psi^*TN_2$. Let $\nabla^\otimes$ be the connection induced on $T^*N_1\otimes\psi^*TN_2$ by the Levi-Civita connection $(\nabla^{g_1})^*$ and $\nabla^{\psi^*g_2}$ of $T^*N_1$ and $\psi^*TN_2$ respectively. \begin{defi}
    The \emph{tension field} of a smooth map $\psi:(N_1,g_1)\to(N_2,g_2)$ is the smooth section of $\psi^*TN_2$ given by \begin{equation}
        \tau(\psi,g_1,g_2):=\trace_{g_1}\nabla^\otimes\mathrm d\psi \ .
    \end{equation}The map $\psi$ is called \begin{enumerate}
    \item[(i)] \emph{totally geodesic} if $(\nabla^\otimes\mathrm d\psi)(X,Y)=0$ for any $X,Y\in\Gamma(TN_1)$;
        \item[(ii)] \emph{harmonic} if $\tau\equiv 0$.
\end{enumerate}\end{defi}\noindent In a local $g_1$-orthonormal basis $\{e_1, \dots, e_{n_1}\}$, the tension field of a map $\psi : (N_1, g_1) \to (N_2, g_2)$ can be computed more explicitly as follows:
\begin{equation}\label{eq:tension_field_local_basis}
   \tau(\psi,g_1,g_2)=\sum_{i=1}^{n_1}(\nabla^\otimes\mathrm d\psi)(e_i,e_i)=\sum_{i=1}^{n_1}\Big(\nabla^{\psi^*g_2}_{e_i}\big(\mathrm d\psi(e_i)\big)-\mathrm d\psi\big(\nabla_{e_i}^{g_1}e_i\big)\Big)
\end{equation}
\begin{lemma}[{\cite[\S 5]{eells1964harmonic}}]\label{lem:composition_harmonic_maps}
Let $(N_i,g_i)$ be smooth Riemannian manifolds for $i=1,2,3$ and $\psi_j:(N_j,g_j)\to(N_{j+1},g_j)$ be smooth maps for $j=1,2$. Suppose that $\psi_1$ is harmonic and $\psi_2$ is totally geodesic, then the composition $\psi_2\circ\psi_1:(N_1,g_1)\to(N_3,g_3)$ is harmonic.
\end{lemma}
\subsubsection{Homogeneous and symmetric spaces for $\SL(n+1,\R)$}\label{sec:homogeneous_space_SL(n+1,R)}\hfill\vspace{0.3em}\\
We will provide an explicit and Lie-theoretic description of the symmetric space for $\mathrm{SL}(n+1, \mathbb{R})$ as well as another homogeneous space, and we will study some properties of a natural projection that can be defined from one to another. The classical material on Lie groups and Lie algebras that we will use can be found in \cite{besse2007einstein}.
 \\ \\ We denote by $K := \mathrm{SO}(n+1)$ the unique (up to conjugation) maximal compact subgroup of $G := \mathrm{SL}(n+1, \mathbb{R})$, by $\mathrm{S}H := \mathrm{S}\big(\mathrm{GL}(n, \mathbb{R}) \times \mathbb{R}^*\big)$ and $H = \mathrm{SL}(n, \mathbb{R})$, given by the reducible inclusion in $\mathrm{SL}(n+1, \mathbb{R})$, and by $C := \mathrm{SO}(n)$, the maximal compact of $\mathrm{SL}(n, \mathbb{R})$, again included reducibly in $\mathrm{SL}(n+1, \mathbb{R})$. As observed in Section \ref{sec:paracomplex_geometry}, the two quotients $G / \mathrm{S}H$ and $G / H$ can be identified, respectively, with the para-complex hyperbolic space $\h_\tau^n$ equipped with its para-Kähler metric $(\g, \ome, \p)$ and the pseudosphere $\mathbb{S}^{n, n+1}$ equipped with its para-Sasaki metric $(\eta, \zeta, \phi, g)$ (see Lemma \ref{lem:para_sasaki_pseudosphere}). In particular, the space $G/\mathrm{S}H$ can be identified with pairs $(l, P)$, where $l$ and $P$ are, respectively, a line and a hyperplane in $\mathbb{R}^{n+1}$ such that $l \pitchfork P$. In contrast, $G/H$ is the set of pairs $(v, P)$ where $v$ is now a vector in $\mathbb{R}^{n+1}$ such that $v\notin P$. With this description, it is easy to see that the quotient $\mathcal X := G / C$ is given by triples $(u, P, q)$ such that the pair $(u, P)$ belongs to $G / H$ and $q$ is an inner product defined on $P$. Finally, regarding the symmetric space $\mathbb{X}_{n+1} := G / K$, this is given by the set of positive definite symmetric bilinear forms on $\mathbb{R}^{n+1}$ with determinant equal to $1$. \\ \\ From the perspective of Lie algebras, the inclusion of the maximal compact $K < G$ and of the subgroup $H < G$ induce two well-known decompositions:
\begin{align}\label{eq:decomposition_lie_algebra_maximalcompact}
    &\mathfrak g=\mathfrak k\oplus\mathfrak p\quad \mathrm{ad_{\mathfrak{k}}-invariant}, \quad \mathfrak k:=\mathrm{Lie}(K) \ \text{and} \ \mathfrak p\cong T_{\Id}\mathbb X_{n+1} \\ &\label{eq:decomposition_lie_algebra_homogeneous}\mathfrak g=\mathfrak h\oplus\mathfrak m\quad \mathrm{ad_{\mathfrak{h}}-invariant}, \quad \mathfrak h:=\mathrm{Lie}(H) \ \text{and} \ \mathfrak m\cong T_{\Id}(G/H).
\end{align}Since $\mathfrak{g}=\mathfrak{sl}(n+1,\R)$ is the set of $(n+1)$-dimensional matrices with zero trace, we obtain the following description for the subspaces appearing in the decomposition: \begin{equation}
    \begin{aligned}
        &\mathfrak k=\{M\in\mathfrak{sl}(n+1,\R) \ | M^T=-M\} \\ &\mathfrak p=\{M\in\mathfrak{sl}(n+1,\R) \ | M^T=M\} \\ & \mathfrak h=\Big\{M=\begin{pmatrix}
            N & 0 \\ 0 & 0 
        \end{pmatrix} \ \Big| \ \trace{N}=0, \ N\in\mathfrak{gl}(n,\R)\Big\} \\ &\mathfrak m=\Big\{M=\begin{pmatrix}
            0 & v \\ w^T & 0
        \end{pmatrix} \ \Big| \ v,w\in\R^n\Big\}
    \end{aligned}
\end{equation}
A similar approach can be used with the subgroup $C = \mathrm{SO}(n) < G$, whose induced decomposition at the level of Lie algebras can be derived from those in (\ref{eq:decomposition_lie_algebra_maximalcompact}) and (\ref{eq:decomposition_lie_algebra_homogeneous}). Indeed, if we denote by $\mathfrak{c} := \mathrm{Lie}(C)=\{N\in\mathfrak{sl}(n,\R) \ | \ N^T=-N\}$, by construction we have that $\mathfrak{c} = \mathfrak{k} \cap \mathfrak{h}$ inside $\mathfrak{g}$, and therefore \begin{equation}\label{eq:decomposition_liealgebra_subspace_C}
    \mathfrak g=\mathfrak c\oplus(\mathfrak k\cap m)\oplus(\mathfrak p\cap\mathfrak h)\oplus(\mathfrak p\cap\mathfrak m) \quad \mathrm{ad_{\mathfrak c}-invariant},
\end{equation}where $T_{\Id}\mathcal X\cong(\mathfrak k\cap m)\oplus(\mathfrak p\cap\mathfrak h)\oplus(\mathfrak p\cap\mathfrak m)$. Moreover, each element of the direct sum has a description as a subset of trace-free matrices: \begin{equation}
    \begin{aligned}\label{eq:decomposition_liealgebra_intersections}
        &\mathfrak k\cap m=\Big\{M=\begin{pmatrix}
            0 & v \\ -v^T & 0
        \end{pmatrix} \ \Big| \ v\in\R^n\Big\} \\ & \mathfrak p\cap\mathfrak h=\Big\{M=\begin{pmatrix}
            N & 0 \\ 0 & 0 
        \end{pmatrix} \ \Big| \ N=N^T, \ N\in\mathfrak{sl}(n,\R)\Big\} \\ & \mathfrak p\cap\mathfrak m=\Big\{M=\begin{pmatrix}
            0 & v \\ v^T & 0
        \end{pmatrix} \ \Big| \ v\in\R^n\Big\}.
    \end{aligned}
\end{equation}
There is a natural projection $\pi_{n+1} : \mathcal{X} \to \mathbb{X}_{n+1}$ defined using the description of the spaces as quotients of $\mathrm{SL}(n+1, \mathbb{R})$. Given a point $(v,P,q)\in\mathcal X$ let us consider the unique scalar product $Q$ on $\R^{n+1}$ defined by the conditions: $Q|_P\equiv q$, $Q(v,w)=0$ for any $w\in P$, and $Q(v,v)=\lambda>0$ so that $\det(Q)=1$. By construction, the only possibility for the scalar $\lambda$ is to be equal to $\big(\det(q)\big)^{-1}$. Then, the aforementioned projection is defined as $\pi_{n+1}(v,P,q):=Q$. Before studying some fundamental properties of this map that will be useful later, let us recall that the spaces $\mathcal{X}$ and $\mathbb{X}_{n+1}$ are respectively endowed with a $G$-bi-invariant pseudo-Riemannian metric $\langle\cdot,\cdot\rangle_{\mathcal X}$ of signature $\big(\frac{n}{2}(n+3), n\big)$ and a Riemannian metric $\langle\cdot,\cdot\rangle_{\mathbb X_{n+1}}$, both induced by the Killing form on the Lie algebra.

\begin{lemma}\label{lem:submersion_to_symmetric_space}
The map $\pi_{n+1}:(\mathcal X,\langle\cdot,\cdot\rangle_{\mathcal X})\to(\mathbb X_{n+1},\langle\cdot,\cdot\rangle_{\mathbb X_{n+1}})$ is a pseudo-Riemannian submersion with negative-definite fiber isomorphic to $g(K/C)\cong\mathbb S^n$, whose tangent space at the identity point is identified with $\mathfrak k\cap\mathfrak m$.
\end{lemma}
\begin{proof}
The surjectivity of $\pi_{n+1}$, together with the surjectivity of its differential, follows directly from the definition, since at the level of $G$–homogeneous spaces it maps the class $gC$ to the class $gK$, where we recall that $C \subset K$. As for the differential, it is enough to perform the computation at the point $[\Id]$ and then use the $G$–invariance to conclude at every other point. Using the decompositions presented in (\ref{eq:decomposition_lie_algebra_maximalcompact}), (\ref{eq:decomposition_lie_algebra_homogeneous}) and (\ref{eq:decomposition_liealgebra_subspace_C}), it is clear that the complement of $\mathfrak c$ in $\mathfrak g$ contains $\mathfrak p$ (that is, the complement of $\mathfrak k$ in $\mathfrak g$). This allows us to conclude that the differential of $\pi_{n+1}$ restricts to the identity on $\mathfrak p$ and vanishes on its complement inside $\mathfrak g / \mathfrak c$, which is isomorphic to $\mathfrak k \cap \mathfrak m$. This argument also implies that the horizontal distribution is isomorphic to $\mathfrak p$ at $[\mathrm{Id}]$, while the vertical one is isomorphic to $\mathfrak k \cap \mathfrak m$. Again, by using $G$–invariance and $\mathrm{ad}$-invariance, the result holds at every other point. The fiber over a point $gK$ is isomorphic to $g(K/C)$, which is a copy of $\mathbb S^n$, since $K = \mathrm{SO}(n+1)$ and $C = \mathrm{SO}(n)$. Finally, the fact that it is also a local isometry when restricted to the horizontal distribution, namely
$
\pi_{n+1}^*\langle \cdot , \cdot \rangle_{\mathcal X} = \langle \cdot , \cdot \rangle_{\mathbb X_{n+1}},
$
follows from the discussion above.

\end{proof}
\begin{remark}
The fact that the fiber over each point is isomorphic to $\mathbb S^n$ also admits a geometric description in terms of the incidence geometry introduced at the beginning of this section. Recall that
$
\pi_{n+1}(v,P,q) = Q,
$
where $Q$ is the unique inner product on $\mathbb R^{n+1}$ with determinant one such that $P = v^{\perp_Q}$, $Q(v,v) = (\det q)^{-1}$, and $Q|_P = q$. Thus, given any inner product $Q$, its preimage via $\pi_{n+1}$ consists of all triples $(v,P,q)$ satisfying the properties above. In other words, given $Q$ and a vector $v \in \mathbb R^{n+1}$, the hyperplane $P$ and the inner product $q$ are completely determined and since $Q(v,v)=(\det q)^{-1}$ is fixed for any $v\in\R^{n+1}\setminus\{0\}$ we get
$
\pi_{n+1}^{-1}(Q) \cong \mathbb S^n.
$

\end{remark}

\subsubsection{Lifts of the immersions}\label{sec:lift_immersion_symmetric_space}\hfill\vspace{0.3em}\\
As seen in Section \ref{sec:immersions_pseudo_spheres}, given a hyperbolic affine sphere $f: M \to \mathbb{R}^{n+1}$ with positive definite Blaschke metric, we can construct the immersion $\sigma^+: M \to \mathbb{S}^{n,n+1}$, which is known to be timelike, maximal, horizontal, and $\phi$-anti-invariant. Projecting onto the para-complex hyperbolic space, we obtain another timelike, maximal, and Lagrangian immersion $\bar{\sigma}: M \to \mathbb{H}_\tau^n$. Moreover, from Section \ref{sec:inverse_problem} we know that this construction can be inverted, as the property of being Lagrangian is the only local obstruction to the existence of a horizontal lift from $\mathbb{H}_\tau^n$ to $\mathbb{S}^{n,n+1}$. In this section, starting from $f: M \to \mathbb{R}^{n+1}$ as above, we define a map $\mathcal{G}_f: M \to \mathbb{X}_{n+1}$ obtained as the composition $\pi_{n+1} \circ \tilde{\mathcal{G}}_f$, where $\pi_{n+1}$ is the pseudo-Riemannian submersion described in Lemma \ref{lem:submersion_to_symmetric_space} and $\tilde{\mathcal{G}}_f$ is a suitable lift of $\sigma^+$ to the space $\mathcal{X}$. \\ \\
Using the description of the homogeneous spaces of $\mathrm{SL}(n+1, \mathbb{R})$ considered in Section \ref{sec:homogeneous_space_SL(n+1,R)}, given $\sigma^+ = (f, \nu): M \to \mathbb{S}^{n,n+1}$, we define its lift $\tilde{\mathcal{G}}_f : M \to \mathcal{X}$ as $\tilde{\mathcal{G}}_f(p) := \big(f(p), f_*(T_p M), h_p(\cdot, \cdot)\big)$, where we are identifying the hyperplane $f_*(T_p M)$ of $\mathbb{R}^{n+1}$ with the kernel of the functional $\nu_p : M \to (\mathbb{R}^{n+1})^*$ (see Definition \ref{def:conormal_map}) and where $h_p(\cdot, \cdot)$ is the positive definite affine metric on the hyperplane $f_*(T_p M)$. The induced map in the symmetric space is simply given by $$\mathcal{G}_f(p) := (\pi_{n+1} \circ \tilde{\mathcal{G}}_f)(p) = h_p(\cdot, \cdot) \oplus \lambda(p) \cdot f(p),$$ where $\lambda(p) := \big(\det h_p\big)^{-1}$, and we are using the splitting $\R^{n+1}=f_*\big(T_pM\big)\oplus\R\cdot f(p)$. In the case where $\dim M = 2$, the map $\mathcal{G}_f$ was already considered in \cite{Labourie_cubic} (see also \cite{labourie2017cyclic}), where it was shown to be harmonic and conformal, hence minimal, and called the \emph{Blaschke lift} of the hyperbolic affine sphere $f: M \to \mathbb{R}^3$. In our more general context, the construction is summarized by the following diagram:
\begin{equation*}\begin{tikzcd}
                                                                                                              &                                      & \mathcal X \arrow[ld,] \arrow[rd, "\pi_{n+1}"] &                 \\
                                                                                                              & {\mathbb S^{n,n+1}} \arrow[d, "\pi_+"] &                                                                   & \mathbb X_{n+1} \\
M \arrow[r, "\bar\sigma"] \arrow[ru, "\sigma^+", shift left] \arrow[rruu, "\tilde{\mathcal G}_f", bend left=49] & \mathbb H_\tau^n                     &                                                                   &                
\end{tikzcd}\end{equation*}
\begin{theorem}\label{thm:harmonic_transvers_map_lift}
Let $f:M\to\R^{n+1}$ be a hyperbolic affine sphere with positive definite Blaschke metric $h$, then \begin{enumerate}
    \item[(i)] The lift $\tilde{\mathcal G}_f:M\to\mathcal X$ is tangent to the horizontal distribution generated by $\pi_{n+1}:\mathcal X\to\mathbb X_{n+1}$; \item[(ii)] The lift $\tilde{\mathcal{G}}_f$ is harmonic with respect to the pseudo-Riemannian metric $\langle\cdot,\cdot\rangle_{\mathcal X}$ on the target space induced by the Killing form, and the Blaschke metric on $M$. 
\end{enumerate}
\end{theorem}
\begin{proof}
Recall that $\tilde{\mathcal{G}}_f(p) := \big(f(p), f_*(T_p M), h_p(\cdot, \cdot)\big)$, for a point $p\in M$. The first step is to compute the differential of the map along a vector field on $M$. For the first component we simply obtain $D_X f$, which we know is equal to $f_* X$ thanks to the structure equations (\ref{eq:structure_equations_f}) and the fact that $f$ is a hyperbolic affine sphere, hence $S = -\mathrm{Id}$. For the second component, we need to compute the variation of $f_*(TM)$ considered, pointwise, as an element of the Grassmannian $\mathrm{Gr}_n(\mathbb{R}^{n+1})$. Given an element $P \in \mathrm{Gr}_n(\mathbb{R}^{n+1})$, we have an identification $T_P \mathrm{Gr}_n(\mathbb{R}^{n+1}) \cong \mathrm{Hom}\big(P, \mathbb{R}^{n+1}/P\big)$. Thus, in the case $P \equiv f_*(T_p M)$, we have $\varphi(X_p) = h_p(X_p, \cdot)f(p)$ for every $\varphi \in T_P \mathrm{Gr}_n(\mathbb{R}^{n+1})$ and every $X_p \in T_p M$, as a consequence of the structure equations (\ref{eq:structure_equations_f}) in the centroaffine case. In other words, the second component of the differential of $\tilde{\mathcal G}_f$ along $X$ is identified with $f_*\big(X^{\dagger}\big)$, where $X^{\dagger}$ is the $1$-form $h(X,\cdot)$ corresponding to $X$ through $h$. The third component is simply given by $\nabla_X h$, where $\nabla$ is the affine connection of the hyperbolic affine sphere, which thus coincides with the tensor $C(X, \cdot, \cdot)$ introduced in (\ref{eq:tensoreC}). At this point, recalling that $T_{\mathrm{Id}} \mathcal{X} \cong (\mathfrak{k} \cap \mathfrak{m}) \oplus (\mathfrak{p} \cap \mathfrak{h}) \oplus (\mathfrak{p} \cap \mathfrak{m})$, whose expression for each component of the direct sum is described in (\ref{eq:decomposition_liealgebra_intersections}), we can represent $\mathrm{d} \tilde{\mathcal{G}}_f(X)$ as the following block matrix: \begin{equation}
    \begin{pmatrix}
        \nabla_Xh & X \\  X^{\dagger} & 0
    \end{pmatrix}
\end{equation} for $X\in\Gamma(TM)$. Since $\nabla_X h$ is a smooth section of $S^2(T^*M)$ and satisfies $\mathrm{tr}_h(\nabla_X h)=0$ for every $X \in \Gamma(TM)$ (see Proposition \ref{prop:properaffinesphere_trace_less}), the element $\mathrm{d}\tilde{\mathcal{G}}_f(X)$ lies only in the blocks $\mathfrak{p} \cap \mathfrak{h} =\{M=\begin{psmallmatrix}   N & 0 \\ 0 & 0  \end{psmallmatrix} \ | \ N=N^T, \ N\in\mathfrak{sl}(n,\R)\} \ \text{and} \ \mathfrak{p} \cap \mathfrak{m} =\{M=\begin{psmallmatrix}    0 & v \\ v^T & 0 \end{psmallmatrix} \ | \ v\in\R^n\}.$
In other words, the map $\tilde{\mathcal{G}}_f$ is transverse to the distribution generated by the Lie algebra $\mathfrak{k} \cap \mathfrak{m}$, that is, to the vertical distribution of the submersion $\pi_{n+1} : \mathcal{X} \to \mathbb{X}_{n+1}$ (see Lemma \ref{lem:submersion_to_symmetric_space}), hence point $(i)$ is proven. \newline 
For point $(ii)$ we start by computing the following element:
$$ (\nabla^\otimes \mathrm d\tilde{\mathcal G}_f)(X,Y) = \nabla_X^{\tilde{\mathcal G}_f^*\langle\cdot,\cdot\rangle_{\mathcal X}}\big(\mathrm d\tilde{\mathcal G}_f(Y)\big) - \mathrm d\tilde{\mathcal G}_f\big(\nabla^h_X Y\big) $$
for every $X,Y \in \Gamma(TM)$, where $\nabla^h$ denotes the Levi–Civita connection of $h$, and all the other connections appearing in the formula correspond to the appropriate induced connections on the bundles, as explained in Section \ref{sec:harmonic_maps}. From the computations already carried out we obtain that \begin{align*}
    &\nabla_X^{\tilde{\mathcal G}_f^*\langle\cdot,\cdot\rangle_{\mathcal X}}\big(\mathrm d\tilde{\mathcal G}_f(Y)\big)=\big(\nabla^h_XY,\nabla^h_X\big(Y^{\dagger}\big),\nabla_X^h\big(C(Y,\cdot,\cdot)\big)\big) \ , \\ & \mathrm d\tilde{\mathcal G}_f\big(\nabla^h_X Y\big)=\big(\nabla^h_XY,\big(\nabla_X^hY\big)^{\dagger},C(\nabla^h_XY,\cdot,\cdot)\big).
\end{align*}
Taking the difference of the equatons component-wise, only two terms remain to be studied. The first is given by: \begin{align*}
    \big(\nabla^h_X(Y^{\dagger})\big)(Z)-\big(\nabla_X^hY\big)^{\dagger}(Z)&=X\cdot\big(Y^\dagger(Z)\big)-Y^\dagger\big(\nabla^h_XZ\big)-h(\nabla_X^hY,Z) \\ & =X\cdot h(Y,Z)-h(Y,\nabla_X^hZ)-h(\nabla_X^hY,Z), \\ & =0 \tag{$h$ is compatible with $\nabla^h$}
\end{align*}for all $Z\in\Gamma(TM)$. While the second, slightly more involved, is given by: \begin{align*}
    \Big(\nabla_X^h\big(C(Y,\cdot,\cdot)\big)-C(\nabla^h_XY,\cdot,\cdot)\Big)(Z,W)&=X\cdot C(Y,Z,W)-C(Y,\nabla_X^hZ,W) \\ & \ \ \ \ \ -C(Y,Z,\nabla_X^hW)-C(\nabla_X^hY,Z,W) \\ &=(\nabla_X^hC)(Y,Z,W)
\end{align*} for all $Z,W\in\Gamma(TM)$. In the end, we obtain $$ (\nabla^\otimes \mathrm d\tilde{\mathcal G}_f)(X,Y) =\big(0,0,(\nabla_X^hC)(Y,\cdot,\cdot)\big). $$
Given a $h$-orthonormal local basis $\{e_1,\dots,e_n\}$, according to (\ref{eq:tension_field_local_basis}), the tension field $\tau$ of the map $\tilde{\mathcal G}_f$ can be computed as \begin{align*}
    \tau=\sum_{i=1}^n(\nabla^\otimes \mathrm d\tilde{\mathcal G}_f)(e_i,e_i)=\Big(0,0,\sum_{i=1}^n(\nabla_{e_i}^hC)(e_i,\cdot,\cdot)\Big).\end{align*}
Using the symmetry of $(\nabla^h_{\cdot}C)(\cdot,\cdot,\cdot)$ (see Lemma \ref{lem:total_symmetry_nabla_C}), the term $\sum_{i=1}^n(\nabla_{e_i}^hC)(e_i,\cdot,\cdot)$, when evalueated on a pair of vector fields $X,Y\in\Gamma(TM)$ is exactly equal to the trace of the covariant derivative with respect to $\nabla^h$ of the Pick tensor $A = h^{-1} C$, i.e. $\trace{(\nabla^h_XA)(Y)}$, which is known to be zero in the case of a proper affine sphere (see Lemma \ref{lem:trace_covariant_pick_tensor}). As a consequence, we obtain that the tension field $\tau$ associated with the map $\tilde{\mathcal{G}}_f$ is identically zero, and therefore $\tilde{\mathcal{G}}_f$ is harmonic.

\end{proof}
\begin{cor}\label{cor:gauss_lift_harmonic}
Given a hyperbolic affine sphere $f:M\to\R^{n+1}$ with positive definite affine metric $h$ its Blaschke lift $\mathcal G_f=\pi_{n+1}\circ\tilde{\mathcal G}_f:M\to\mathbb X_{n+1}$ is harmonic.
\end{cor}
\begin{proof}
Since $\mathcal G_f$ is the composition of two smooth maps, there is a standard formula relating the tension field $\tau(\mathcal G_f)$ with that of $\tilde{\mathcal G}_f$, together with another term depending on $\pi_{n+1}$ (see \cite{eells1964harmonic, eells1983selected} for instance). More precisely, given a local $h$–orthonormal frame ${e_1,\dots,e_n}$ on $M$, we have
\begin{equation}\label{eq:tension_field_of_composition}
\tau(\mathcal G_f) = \mathrm d\pi_{n+1}\big(\tau(\tilde{\mathcal G}_f)\big)+\sum_{i=1}^n \big(\nabla^\otimes \mathrm d\pi_{n+1}\big)\big(\mathrm d\tilde{\mathcal G}_f(e_i), \mathrm d\tilde{\mathcal G}_f(e_i)\big).
\end{equation}
By point $(ii)$ of Theorem \ref{thm:harmonic_transvers_map_lift}, we know that $\tau\big(\tilde{\mathcal G}_f\big)\equiv 0$, and we will now show that also the second term appearing in (\ref{eq:tension_field_of_composition}) vanishes. For any fixed index $i\in\{1,\dots,n\}$ we have \begin{align*}
    \big(\nabla^\otimes \mathrm d\pi_{n+1}\big)\big(\mathrm d\tilde{\mathcal G}_f(e_i), \mathrm d\tilde{\mathcal G}_f(e_i)\big)&=\nabla_{\mathrm d\tilde{\mathcal G}_f(e_i)}^{\pi_{n+1}^*\langle\cdot,\cdot\rangle_{\mathbb X_{n+1}}}\Big(\mathrm d\pi_{n+1}\big(\mathrm d\tilde{\mathcal G}_f(e_i)\big)\Big) \\ & \ \ \ \ \ \ -\mathrm d\pi_{n+1}\Big(\nabla^{\langle\cdot,\cdot\rangle_{\mathcal X}}_{\mathrm d\tilde{\mathcal G}_f(e_i)}\mathrm d\tilde{\mathcal G}_f(e_i)\Big) . 
\end{align*}Instead, by point $(i)$ of Theorem \ref{thm:harmonic_transvers_map_lift}, we know that the vector field $\mathrm d\tilde{\mathcal G}f(e_i)$ is horizontal in $T\mathcal X$, and this allows us to obtain, using also that $\pi_{n+1}$ is a local isometry by Lemma \ref{lem:submersion_to_symmetric_space}, the following chain of equalities:
\begin{align*}
\mathrm d\pi_{n+1}\Big(\nabla^{\langle\cdot,\cdot\rangle_{\mathcal X}}_{\mathrm d\tilde{\mathcal G}_f(e_i)}\mathrm d\tilde{\mathcal G}_f(e_i)\Big)&=\mathrm d\pi_{n+1}\Bigg(\Big(\nabla^{\langle\cdot,\cdot\rangle_{\mathcal X}}_{\mathrm d\tilde{\mathcal G}_f(e_i)}\mathrm d\tilde{\mathcal G}_f(e_i)\Big)^{\mathcal H}\Bigg) \\ &=\nabla^{\langle\cdot,\cdot,\rangle_{\mathbb X_{n+1}}}_{\mathrm d\pi_{n+1}\big(\tilde{\mathrm d\mathcal G}_f(e_i)\big)}\mathrm d\pi_{n+1}\big(\mathrm d\tilde{\mathcal G}_f(e_i)\big) \tag{Thereom \ref{thm:harmonic_transvers_map_lift}} \\ &=\nabla_{\mathrm d\tilde{\mathcal G}_f(e_i)}^{\pi_{n+1}^*\langle\cdot,\cdot\rangle_{\mathbb X_{n+1}}}\Big(\mathrm d\pi_{n+1}\big(\mathrm d\tilde{\mathcal G}_f(e_i)\big)\Big), \tag{Lemma \ref{lem:submersion_to_symmetric_space}}
\end{align*}where the superscript $\mathcal H$ in the first line denotes the horizontal component of the vector field. Since this computation holds for every index $i \in\{1,\dots,n\}$, we conclude that the second term in (\ref{eq:tension_field_of_composition}) vanishes, and hence $\mathcal G_f$ is harmonic.

\end{proof}

\subsubsection{Induced map in the $\SO_0(n+1,n+1)$ symmetric space}\hfill\vspace{0.3em}\\
Recall that there is an inclusion of Lie groups $\iota: \mathrm{GL}(n+1,\mathbb R) \hookrightarrow \mathrm{SO}(n+1,n+1), $ explicitly given by $
\iota(M) = \begin{psmallmatrix} M & 0 \\ 0 & (M^T)^{-1} \end{psmallmatrix},
$ and described in Section \ref{sec:Lie_groups_inclusion}. In particular, by restricting the homomorphism $\iota$ to the subgroup $\SL(n+1,\R)$, the image is contained in $\SO_0(n+1,n+1)$. We now aim to describe the induced map at the level of symmetric spaces and to show how pre-composition with the Blaschke lift $\mathcal G_f$, introduced in Section \ref{sec:lift_immersion_symmetric_space}, still produces a harmonic map taking values in the symmetric space of $\mathrm{SO}_0(n+1,n+1)$. \\ \\
The maximal compact subgroup of $\SO_0(n+1,n+1)$ is isomorphic to $\SO(n+1)\times\SO(n+1)$ and the quotient space, denoted with $\mathbb Y_{n+1}$, can be described as the set of $(n+1)$-dimensional positive definite subspaces of $\R^{n+1}\oplus\R^{n+1}$ endowed with the non-degenerate bilinear form $$ b\big((x,y),(\tilde x,\tilde y)\big):=(x^T,y^T)B\begin{pmatrix}\tilde x \\ \tilde y\end{pmatrix}, \quad \text{where} \ B:=\frac{1}{2}\begin{pmatrix}
    0 & \Id \\ \Id & 0
\end{pmatrix} $$ already introduced in Section \ref{sec:isometric_model}. 
Given an element $N \in \mathbb X_{n+1}$, viewed as a symmetric positive-definite matrix with unit determinant, we can consider the graph $\Gamma_N$ of the induced linear map
$
\varphi_N : \mathbb R^{n+1} \to \mathbb R^{n+1}
$
as a subset of the direct sum vector space $\mathbb R^{n+1} \oplus \mathbb R^{n+1}$. The graph $\Gamma_N$ is positive definite as a subspace of $(\R^{n+1,n+1}, b)$. Indeed, given an arbitrary element $(x, N \cdot x) \in \Gamma_N$, we have that \begin{align*}
    b\big((x,N\cdot x),(x,N\cdot x)\big)&=\frac{1}{2}(x^T,x^T\cdot N^T)\begin{pmatrix}
        0 & \Id \\ \Id & 0
    \end{pmatrix}\begin{pmatrix}
        x \\ N\cdot x
    \end{pmatrix} \\ &=\frac{1}{2}\big(x^T\cdot N\cdot x+ x^T\cdot N^T\cdot x\big) \\ &=x^T\cdot N\cdot x>0, \ \text{for all} \ x\neq 0 
\end{align*}since $N$ is symmetric and positive definite. In other words, we have a well-defined smooth map \begin{equation}
    \begin{aligned}\label{eq:map_between_symmetric_spaces}
        \Phi: \ &\mathbb X_{n+1}\longrightarrow\mathbb Y_{n+1} \\ & N\longmapsto\Gamma_N
    \end{aligned}
\end{equation}from the symmetric space of $\SL(n+1,\R)$ to the symmetric space of $\SO_0(n+1,n+1)$. The existence of the left action of $\mathrm{SL}(n+1,\mathbb R)$ on $\mathbb X_{n+1}$, given by
$M \ast N := (M^T)^{-1} N M^{-1},\ \text{for} \ M \in \mathrm{SL}(n+1,\mathbb R) \ \text{and} \ N \in \mathbb X_{n+1},$
allows us to relate the map $\Phi$ with the homomorphism $\iota$. Indeed, the graph
$
\Gamma_{M \ast N} = \{ (x, (M^T)^{-1} N M^{-1} \cdot x) \ | \ x\in\R^{n+1} \}
$
still corresponds to a point of $\mathbb Y_{n+1}$, being positive definite and of dimension $n+1$. On the other hand, if one considers the action of $\iota(M) \in \mathrm{SO}_0(n+1,n+1)$ on the graph $\Gamma_N$, one obtains the subset
$
\{ (M \cdot x, (M^T)^{-1} N \cdot x) \ | \ x \in \mathbb R^{n+1} \},
$
which is clearly in bijection with $\Gamma_{M \ast N}$ via the identification $x \mapsto y := M^{-1} \cdot x$. We have therefore established the following result \begin{lemma}
The smooth map $\Phi:\mathbb X_{n+1}\to\mathbb Y_{n+1}$ defined in (\ref{eq:map_between_symmetric_spaces}) is $\iota$-equivariant, namely $$\Phi\big(M\ast N\big)=\iota(M)\cdot \Phi(N), \quad \text{for all} \  M\in\SL(n+1,\R), \ N\in\mathbb X_{n+1}$$
\end{lemma} 

\noindent Returning to the construction of Section \ref{sec:lift_immersion_symmetric_space}, starting from a hyperbolic affine sphere $f : M \to \mathbb R^{n+1}$ with positive-definite affine metric $h$, we defined its Blaschke lift $\mathcal G_f : M \to \mathbb X_{n+1}$ as $\mathcal G_f(p)= h_p(\cdot,\cdot) \oplus \lambda(p),\ \lambda = (\det h)^{-1},$
which we know to be a harmonic map (Corollary \ref{cor:gauss_lift_harmonic}). If we now consider its post-composition with $\Phi$, we obtain the following result:
\begin{theorem}\label{thm:harmonic_map_SO(n+1,n+1)}
Given a hyperbolic affine sphere $f:M\to\R^{n+1}$ with positive definite affine metric, the smooth map $\mathcal G_{\sigma^+}:=\Phi\circ\mathcal G_f:M\to\mathbb Y_{n+1}$ is harmonic, where on $\mathbb Y_{n+1}$ we consider its Riemannian metric induced by the Killing form.
\end{theorem}
\begin{proof} As we have already observed, $\Phi$ is a smooth map between Riemannian symmetric spaces, induced by a group homomorphism that preserves the inclusion of maximal compact subgroups (see Section \ref{sec:Lie_groups_inclusion}). In this case, a standard argument involving Lie triple systems (see \cite[Theorem 7.2, Chapter IV]{helgason1979differential}) allows us to conclude that the image of $\Phi$ is in fact a totally geodesic submanifold of $\mathbb Y_{n+1}$. We may therefore complete the proof by invoking Lemma \ref{lem:composition_harmonic_maps} and Corollary \ref{cor:gauss_lift_harmonic}.

\end{proof}

\appendix
\section{Contact geometry}\label{appendix_A}
\noindent In this first appendix, we begin by recalling the fundamental notions of contact geometry, following \cite[\S 3.5]{mcduff2017introduction}, and then state and prove a result that allows one to construct contact submanifolds within symplectic manifolds possessing specific symmetries. The proposition will be applied to the pseudosphere $\mathbb{S}^{n,n+1}$, regarded as a quadric in $\mathbb{R}^{n+1,n+1}$ naturally endowed with the standard symplectic form.  \begin{defi}\label{def:contact_structure}
A \emph{contact structure} on a smooth manifold $M$ of dimension $2n+1$ is the datum of a $1$-form $\eta$ on $M$ such that $\eta\wedge(\mathrm d\eta)^{n}\neq 0$. The $1$-form $\eta$ is called \emph{contact $1$-form}
\end{defi} \noindent By virtue of Frobenius' integrability theorem, having a contact $1$-form $\eta$ on the manifold $M$ is equivalent to saying that the distribution $\ker \eta \subset TM$ of rank $2n$ is maximally non-integrable. In particular, $\mathrm{d}\eta$ restricts to a non-degenerate 2-form on $\ker \eta$, implying that $M$ is orientable. For any given contact form $\eta$ there exists a unique vector field $\zeta\in\Gamma(TM)$ such that \begin{equation}
    \iota_\zeta(\mathrm d\eta)=0, \quad \eta(\zeta)=1
\end{equation}and it is called the \emph{Reeb vector field} determined by $\eta$. In other words, $\zeta$ is the normalized vector field pointing towards the unique null direction of $\mathrm{d}\eta$.

\begin{prop}\label{prop:contact_submanifolds}
Let $(N,\omega)$ be a symplectic manifold of dimension $2n$ and let $H:N\to \R$ be a smooth function. Let $c\in\R^*$ be a regular value for $H$ and $\Sigma_c:=H^{-1}(c)\subset N$ the associated embedded hypersurface. Suppose there exists a vector field $X\in\Gamma(TN)$ such that: \begin{itemize}
    \item[(i)] $\mathcal{L}_X\omega=\omega$ ($X$ is a Liouville vector field); \item[(ii)] $X$ is transverse to $\Sigma_c$, i.e. $\mathrm dH(X)|_{\Sigma_c}\neq 0$;  
\end{itemize}then $\eta:=j^*(\iota_X\omega)$ is a contact $1$-form on $\Sigma_c$, where $j:\Sigma_c\to N$ denotes the inclusion.
\end{prop}
\begin{proof}
Since $\Sigma_c$ is a manifold of dimension $2n-1$, the key point is to show that the condition $\eta \wedge (\mathrm{d}\eta)^{n-1} \neq 0$ holds, where $\eta = j^*(\iota_X \omega)$. Thanks to the assumptions on the vector field $X$ and Cartan's magic formula, on the manifold $N$ we have \begin{equation}\label{eq:Cartan_magic_contact}\omega = \mathcal{L}_X \omega = \mathrm{d}(\iota_X \omega) + \iota_X \mathrm{d}\omega = \mathrm{d}(\iota_X \omega). \end{equation} Pulling back via the inclusion $j$, we obtain $\mathrm{d}\eta = j^*\omega$. Moreover, since by assumption the vector field $X$ is transverse to the hypersurface $\Sigma_c$, at every point $p \in N$ we have a decomposition $T_p N = j_*\big(T_p \Sigma_c\big) \oplus \mathbb{R} X_p$. Given a local basis $\{e_1, \dots, e_{2n-1}\}$ of $T_p \Sigma_c$, we obtain a local basis $\{e_1, \dots, e_{2n-1}, X_p\}$ of $T_p N$, and thus $(\omega^n)(j_*e_1, \dots, j_*e_{2n-1}, X_p) \neq 0, \  \text{for any } p \in N$, since $\omega^n$ is a volume form. In particular \begin{align*}
    \eta\wedge(\mathrm d\eta)^{n-1}(e_1,\dots,e_{2n-1})&=j^*\big(\iota_X\omega\wedge(\mathrm d\iota_X\omega)^{n-1}\big)(e_1,\dots,e_{2n-1}) \\ & =\big(\iota_X\omega\wedge\omega^{n-1}\big)(j_*e_1,\dots,j_*e_{2n-1}) \tag{Equation (\ref{eq:Cartan_magic_contact})}
\end{align*} A straightforward application of induction shows that
$\iota_X(\omega^n) = n\, (\iota_X \omega) \wedge \omega^{n-1},$
and therefore, returning to the main computation, we conclude that
$$\eta\wedge(\mathrm d\eta)^{n-1}(e_1,\dots,e_{2n-1})=\frac{1}{n}(\omega^n)(j_*e_1,\dots,j_*e_{2n-1},X_p)\neq 0, \quad \text{for any} \ p\in N.$$
\end{proof}
\begin{remark}\label{rem:contact_geometry_hypersurfaces}
In the particular case of Proposition \ref{prop:contact_submanifolds}, the distribution $\ker \eta$ can be identified with the space
$\{Y \in \Gamma(T\Sigma_c) \mid \omega(X, j_*Y) = 0\} = j_*(X^{\perp_\omega}).$
In particular, we have a decomposition $T\Sigma_c = \ker \eta \oplus L$, where
$L = \{Y \in \Gamma(T\Sigma_c) \mid \omega(X, j_*Y) \neq 0\}$
is a 1-dimensional and hence integrable distribution.
If $\zeta \in \Gamma(T\Sigma_c)$ denotes the vector field generating it, then, up to scaling, $\eta(\zeta) = \omega(X, j_*\zeta) = 1$, and $\mathrm{d}\eta(\zeta, Y) = \omega(\zeta,j_*Y) = 0$ for every $Y \in \Gamma(T\Sigma_c)$. 
\end{remark}
\noindent At this point, we want to apply the previous result to the case $M=\mathbb{R}^{n+1,n+1}$ and $\Sigma_{1}=\mathbb{S}^{n,n+1}$, for a specific choice of the function $H$. Recall that in Section \ref{sec:isometric_model}, we have defined $b((x,y),(x,y))=x_1y_1+\dots+x_{n+1}y_{n+1}$ on $\mathbb{R}^{n+1,n+1}$, which induces a non-degenerate bilinear form of signature $(n+1,n+1)$. In particular, we have a well-defined smooth function $H(x,y) := b\big((x,y),(x,y)\big)$ from $\mathbb{R}^{n+1,n+1}$ to $\mathbb{R}$. Consider the symplectic form $\omega = -\sum_{i=1}^{n+1} \mathrm{d}x_i \wedge \mathrm{d}y_i$ and the vector field \begin{equation}\label{eq:Liouville_vector_field_pseudosphere}
X := \frac{1}{2}\sum_{i=1}^{n+1} \Big(x_i \frac{\partial}{\partial x_i} + y_i \frac{\partial}{\partial y_i}\Big) 
\end{equation} both defined on $\R^{n+1,n+1}$. It is readily seen that $\mathcal{L}_X \omega = \omega$ and $\mathrm{d}H(X) = -1$, i.e., $X$ is transverse to the hypersurface $
\Sigma_{1} := \{(x,y)\in \mathbb{R}^{n+1,n+1} \mid H(x,y)=1\}\cong\mathbb{S}^{n,n+1}.$ Therefore, applying Proposition \ref{prop:contact_submanifolds}, we have proved the following result \begin{prop}\label{prop:contact_pseudo_hyperbolic}
The $1$-form $\iota_X\omega=\frac{1}{2}\sum_{i=1}^{n+1}(y_i\mathrm dx_i-x_i\mathrm dy_i)$ induces, by restriction, a contact form $\eta$ on $\mathbb{S}^{n,n+1}$ and a decomposition $T\mathbb{S}^{n,n+1}=\Ker \ \eta \ \oplus L$, where $L$ is generated by the Reeb vector field $$\zeta=\sum_{i=1}^{n+1}\Big(x_i \frac{\partial}{\partial x_i} - y_i \frac{\partial}{\partial y_i}\Big).$$
\end{prop}
\begin{remark}
In a similar setting, but expressed in a different language, the contact form on hypersurfaces in $\mathbb{R}^{n+1,n+1}$ given by level sets of non-degenerate quadratic forms has also been studied in \cite[\S 9.1]{delarue2025locally}, where it was used to investigate dynamical properties of the space. See also \cite[\S A.1.2]{dorfmeister2024half} \end{remark}

\section{Para-Sasaki metrics on principal bundles}\label{appendix_B}
\noindent We now introduce the notion of a \emph{para-Sasaki} metric on a smooth manifold of odd dimension. Several references in the literature address this topic, but we will primarily follow the exposition in \cite{zamkovoy2009canonical,loiudice2025sasakian}. The goal is to show that such special metrics can be induced on principal $G$-bundles over para-Kähler manifolds, where $G$ is a one-dimensional Lie group. An analogous result has been established for Sasaki metrics on principal bundles over Kähler manifolds (\cite{hatakeyama1963some}), and seemed not to be known in the para-complex setting. We will apply this theory to the principal $\mathbb{R}$-bundle given by the pseudosphere $\mathbb{S}^{n,n+1}$ over the para-complex hyperbolic space $\mathbb{H}^n_\tau$.
\begin{defi}\label{def:para_complex_structure}A \emph{para-complex} structure on a smooth manifold \( N \) of real dimension \( 2n \) is an endomorphism \( \p \) of the tangent bundle such that:  
\begin{enumerate}  
\item[(i)] \( \p^2 = \mathrm{Id} \)  
\item[(ii)] The two eigendistributions \( D_{\pm} := \Ker\big(\p \mp \mathrm{Id}\big) \) both have rank \( n \) and are involutive.  
\end{enumerate}  \end{defi}
\noindent The pair $(N,\p)$ will be called a \emph{para-complex manifold} of para-complex dimension $n$. It is worth mentioning that the integrability conditions on the distributions $D_\pm$ can be rephrased in terms of the vanishing of the Nijenjuis tensor $$N_\p(X,Y)=[X, Y] + [\p(X), \p(Y)] - \p\big([\p(X), Y]\big) - \p\big([X, \p(Y)]\big)$$ for all $X,Y\in TN$ (see \cite[Proposition 1]{cortes2004special}). Moreover, if there exists a non-degenerate symmetric bilinear two-tensor \( \g \) on \( (N,\p) \) such that  
\[
\g(\p X, \p Y) = -\g(X,Y), \quad \text{for any} \ X,Y\in\Gamma(TM)
\]  
then \( \g \) defines a pseudo-Riemannian metric on \( N \) of signature \( (n, n) \), and the triple \( (N, \p, \g) \) is called a \emph{para-Hermitian manifold}. It is easily verified that the pairing  
\[
\ome := \g(\cdot, \p \cdot)
\]  
defines a 2-form on \( N \), known as the fundamental 2-form of \( (N,\p, \g) \). If, in addition, \( \ome \) is a symplectic 2-form, i.e., \( \mathrm{d}\ome = 0 \), then the manifold \( (N,\p, \g, \ome) \) is called \emph{para-Kähler}.
\\ \\ Recall that a contact manifold is a triple $(P,\eta,\zeta)$ where $P$ is an odd dimensional smooth manifold, $\eta$ is a $1$-form on $P$ such that $\eta\wedge(\mathrm d\eta)^{\frac{\dim P-1}{2}}\neq 0$ and $\zeta$ is the Reeb vector field (Definition \ref{def:contact_structure}). \begin{defi}\label{def:para_metric_contact}
    Let $(P,\eta,\zeta)$ be a contact manifold of dimension $2n+1$, then it is called a \emph{para-contact metric} manifold if there exists an endomorphism $\phi$ of $TP$ and a pseudo-Riemannian metric $\hat g$ such that: \begin{enumerate}
    \vspace{0.2em}\item[(i)] $\phi(\zeta)=0$ and $\eta\circ\phi=0$; 
\vspace{0.2em}\item[(ii)] $\phi$ induces a para-complex structure when restricted to the horizontal distribution of the contact structure; 
    \vspace{0.2em}\item[(iii)] $\phi^2(X)=X-\eta(X)\zeta$; 
    \vspace{0.2em}\item[(iv)] $g(\phi(X),\phi(Y))=-g(X,Y)+\eta(X)\eta(Y)$;
   \vspace{0.2em} \item[(v)] $\mathrm d\eta(X,Y)=g(X,\phi(Y))$ 
\end{enumerate} for each $X,Y\in\Gamma(TP)$.
\end{defi}
\noindent The properties of a para-contact metric manifold induce a splitting of its tangent bundle as
$$
TP = D_+ \oplus D_- \oplus \mathbb{R} \zeta,
$$
where $D_\pm$ are the eigendistributions of $\phi$ corresponding to the eigenvalues $\pm1$. In particular, the signature of the metric $g$ is $(n+1, n)$, since
$$
g(\zeta, \zeta) = 1 +g(\phi(\zeta), \phi(\zeta)),
$$
but $\phi(\zeta) = 0$. As in the case of complex and para-complex structures, the endomorphism $\phi$ also admits a notion of integrability, encoded in a tensor that generalizes the classical Nijenhuis tensor.
\begin{defi}\label{def:para_Sasaki_metric}
A para-contact metric manifold $(P, \eta, \zeta, \phi,g)$ is called a \emph{para-Sasaki metric manifold} if
$$
N_\phi(X, Y) - \mathrm{d}\eta(X, Y)\, \zeta = 0, \quad \text{for each } X, Y \in \Gamma(TP),
$$
where
$
N_\phi(X, Y) :=[X, Y] + [\phi(X), \phi(Y)] - \phi\big([\phi(X), Y]\big) - \phi\big([X, \phi(Y)]\big)-\eta\big([X,Y]\big)\zeta.
$
\end{defi}
\noindent 
We briefly recall the definition of connection on a principal $G$-bundle and then we state and prove the main result concerning the existence of para-Sasaki metrics. \begin{defi}
A \emph{connection} on a principal $G$-bundle $P$ is a $\mathfrak g$-valued $1$-form $\omega$ on $P$ such that:\begin{enumerate}
    \item[(i)]$\mathrm{Ad}_g(R_g^*\omega)=\omega$;
    \item[(ii)] for any element $v\in\mathfrak g$ we have $\omega(X_v)=v$, where $X_v$ is the vector field obtained by differentiating the $G$-action on $P$.
\end{enumerate}
\end{defi}
\begin{remark}
In the case $G=\R$, the element $\omega$ is simply a real-valued $1$-form, and the first property implies that it is invariant under the $\R$-action
\end{remark}
\begin{prop}\label{prop:para_Sasaki_principal_bundle}
Let $N$ be a smooth manifold endowed with a para-K\"ahler metric $(\g,\p,\ome)$ of signature $(n,n)$ and with integral cohomology class $[\ome]$. Let $\pi:P\to N$ be the associated principal $G$-bundle, with Euler class $[\ome]$ and where $G$ is a Lie group of dimension $1$. Then, there exists a para-Sasaki metric $(\eta,\zeta,\phi,g)$ on $P$ such that $\mathrm d\eta=\pi^*\ome$.
\end{prop}
\begin{proof}
Let $\eta$ be the 1-form defining the connection on the principal bundle $\pi : P \to N$ whose curvature form $\mathrm d\eta$ is given by $\pi^*\ome$. This is equivalent to have a $G$-equivariant horizontal and vertical distribution denoted by $\mathcal{H}$ and $\mathcal{V} = \ker(\pi_*)$, respectively, such that $TP = \mathcal{H} \oplus \mathcal{V}$. Moreover, since $\mathcal{V}$ is one-dimensional, it is integrable and generated by a vector field $\zeta$. Given any point $x \in P$ and any vector $Y \in T_xP$, there exists a unique decomposition $Y = Y^\mathcal{H} + Y^\mathcal{V}$ into horizontal and vertical components, with $Y^\mathcal{H} \in \mathcal{H}_x$ and $Y^\mathcal{V} \in \mathcal{V}_x$. Now let us define an endomorphism $\phi_x$ of $T_xP$ by
$$
\phi_x(Y) :=\overline{\p\big(\pi_*(Y)\big)}, \quad \text{for } Y \in T_xP,
$$ where the bar above the vector denotes its horizontal lift. Notice that since $\zeta$ is the vector field generating the vertical distribution, we have $\phi_x(\zeta_x) = 0$ for all $x\in P$.  Moreover,
\begin{align*}
\phi_x^2(Y) &= \phi_x\Big(\overline{\p\big(\pi_*(Y)\big)}\Big) \\
&= \overline{\p\Big(\pi_*\big(\overline{\p\big(\pi_*(Y)\big)}\big)\Big)} \\
&= \overline{\p\Big(\p\big(\pi_*(Y)\big)\Big)} \tag{$\pi_*(v^\mathcal H)=v,$ for any $v\in T_xP$} \\
&= \overline{\pi_*(Y)} \tag{$\p^2=\Id$} \\ &=Y^{\mathcal H}.
\end{align*}
Thanks to the decomposition of $TP$ induced by $\eta$ and $\zeta$, we can express the horizontal component of a vector $Y$ in terms of these data. Indeed,
$$
Y = Y^\mathcal{H} + Y^\mathcal{V} = \overline{\pi_*(Y)} + Y^\mathcal{V} = \overline{\pi_*(Y)} + \lambda(x) \zeta_x,
$$
where $\lambda$ is a smooth real-valued function on $P$. Since $\eta|_{\mathcal{H}} \equiv 0$ and $\eta_x(\zeta_x) = 1$, we compute
$$
\eta_x(Y) = \eta_x\big(\overline{\pi_*(Y)}\big) + \lambda(x)\eta_x(\zeta_x) = \lambda(x),
$$
From this, we obtain the identity $Y^\mathcal{H} = Y - \eta_x(Y)\zeta_x$, which, combined with the computation above, implies that
$
\phi_x^2(Y) = Y - \eta_x(Y)\zeta_x.
$ It is now clear that, for each $x \in P$, the endomorphism $\phi_x$ restricts to a para-complex structure on $\mathcal{H}_x$, which can be identified with $\p_{\pi(x)}$ via the isomorphism ${\pi_*}_{|_{\mathcal{H}_x}} : \mathcal{H}_x \to T_{\pi(x)}N$. \noindent To complete the proof, it remains to verify properties (iv) and (v) of Definition \ref{def:para_metric_contact} and the integrability condition of Definition \ref{def:para_Sasaki_metric}. As a first step, we introduce a symmetric bilinear form on $T_xP$ given by the formula:
$$
{g}_x(Y, Z) := \g_{\pi(x)}\big(\pi_*(Y), \pi_*(Z)\big) + \eta_x(Y)\eta_x(Z), \quad \text{for } Y, Z \in T_xP.
$$ The bilinear form $g_x$ is non-degenerate and has signature $(n+1, n)$, since it restricts to the pullback of $\g$ via $\pi$ on the horizontal distribution, and
$
g(\zeta_x, \zeta_x) = \eta_x(\zeta_x)\eta_x(\zeta_x) = 1.
$
We thus obtain a well-defined pseudo-Riemannian metric of signature $(n+1, n)$ on $P$. We can therefore compute for all $x\in P$
\begin{align*}
g_x(\phi_x(Y), \phi_x(Z))&=\g_{\pi(x)}\big(\pi_*\big(\phi_x(Y)\big),\pi_*\big(\phi_x(Z)\big)\big)+\eta_x\big(\phi_x(Y)\big)\eta_x\big(\phi_x(Z)\big) \\ &=\g_{\pi(x)}\Big(\pi_*\big(\overline{\p(\pi_*(Y))}\big),\pi_*\big(\overline{\p(\pi_*(Y))}\big)\Big) \tag{$\eta\circ\phi=0$} \\ &=\g_{\pi(x)}\big(\p(\pi_*(Y)),\p(\pi_*(Z))\big) \\ &=-\g_{\pi(x)}(\pi_*(Y),\pi_*(Z)) \tag{$\g(\p\cdot,\p\cdot)=-\g(\cdot,\cdot)$} \\ &=-g_x(Y,Z)+\eta_x(Y)\eta_x(Z), 
\end{align*}for all $Y,Z\in T_xP$. With a similar approach, we compute: \begin{align*}
g_x(Y,\phi_x(Z))&=\g_{\pi(x)}\Big(\pi_*(Y),\pi_*\big(\overline{\p\big(\pi_*(Z)\big)}\big)\Big)+\eta_x(Y)\eta_x(\phi_x(Z)) \\ &=\g_{\pi(x)}\big(\pi_*(Y),\p(\pi_*(Z))\big) \\ &=\ome_{\pi(x)}(\pi_*(Y),\pi_*(Z)) \tag{$g(\cdot,\p\cdot)=\omega$} \\ &=(\pi^*\ome)_x(Y,Z) \\ &=(\mathrm d\eta)_x(Y,Z), \tag{$\omega$ is the curvature form of $\eta$}
\end{align*} for all $Y,Z\in T_xP$. We conclude with the proof of the integrability criterion in Definition \ref{def:para_Sasaki_metric}. In particular, we need to show that
$
N_\phi(X, Y) -\mathrm{d}\eta(X, Y)\, \zeta = 0 \quad \text{for all } X, Y \in \Gamma(TP),
$
where $N_\phi$ denotes the Nijenhuis tensor of $\phi$, as defined earlier. Notice that, this is equivalent to show the following two conditions: $$\pi_*(N_\phi(Y,Z))=0 \quad \text{and} \quad \eta\big(N_\phi(X, Y) -\mathrm{d}\eta(X, Y)\, \zeta\big)=0$$ for all $X,Y\in TP$. Let $x \in P$ be a point, and let $Y, Z \in T_xP$. Denote by $Y^G$ and $Z^G$ their $G$-invariant horizontal extensions in a neighborhood of $x$, chosen so that their values at $x$ coincide with $Y$ and $Z$, respectively.
For the first part of the computation, using the identities
$
\pi_*[Y^G, Z^G] = [\pi_*(Y^G), \pi_*(Z^G)] \ \text{and} \ \pi_*(\phi(Y^G)) = \p(\pi_*(Y^G)),
$
we deduce that the term $\pi_*(N_\phi(Y^G, Z^G))$ coincides exactly with the Nijenhuis tensor of the para-complex structure $\p$ evaluated on the vectors $\pi_*(Y^G)$ and $\pi_*(Z^G)$, which vanishes since $\p$ is integrable. For the second part, we can proceed with a direct computation for $X, Y \in \Gamma(TP)$. Indeed, we have: \begin{align*}
\eta\big(N_\phi(Y,Z)\big)&=\eta\big([Y,Z]\big)+\eta\big([\phi(Y),\phi(Z)]\big)-\eta\big([Y,Z]\big)\eta(\zeta) \tag{$\eta\circ\phi=0$} \\ &=\eta\big([\phi(Y),\phi(Z)]\big) \tag{$\eta(\zeta)=1$} \\ &=\phi(Y)\cdot\eta\big(\phi(Z)\big)-\phi(Z)\cdot\eta\big(\phi(Y)\big)-\mathrm d\eta(\phi(Y),\phi(Z)) \\ &=-\mathrm d\eta(\phi(Y),\phi(Z)) \\ &=- g\big(\phi(Y),\phi^2(Z)\big) \\ &=-\mathrm d\eta(Z,Y) \tag{$\phi(Y)\perp_{ g}\zeta$} \\ &=\mathrm d\eta(Y,Z)
\end{align*} and the claim follows.
\end{proof}

\begin{cor}\label{cor:para_sasaki_pseudo_sphere}
The pseudosphere $\mathbb{S}^{n,n+1}$ can be endowed with a para-Sasaki metric $(\eta,\zeta,\phi,g)$, where \begin{itemize}
    \item[(i)] $\eta$ and $\zeta$ are respectively the contact form and Reeb vector field described in Proposition \ref{prop:contact_pseudo_hyperbolic}; \item[(ii)] $\phi(Y)=\mathrm{pr}_{\mathcal H}\big(\hat\p(Y)\big)=\hat\p(Y)-\eta(Y)\zeta-\eta\big(\hat\p(Y)\big)\hat\p(\zeta)$, where $Y$ is a vector field tangent to $\mathbb S^{n,n+1}$ but considered with values in $\R^{n+1,n+1}$ and $\hat\p=(\Id,-\Id)$;
    \item[(iii)] $g$ is the pseudo-Riemannian metric of signature $(n,n+1)$ induced on $\mathbb S^{n,n+1}$ by restriction of the bilinear form on $\R^{n+1,n+1}.$  
\end{itemize}
\end{cor}
\begin{proof}
The fact that $\phi$ satisfies the properties required by Definition \ref{def:para_Sasaki_metric} follows easily from the observation that, for a vector field $Y \in \Gamma(T\mathbb{S}^{n,n+1})$, it is defined as the projection onto the horizontal distribution $\mathcal{H} \subset T\mathbb{S}^{n,n+1}$ of the vector field $\p(Y)$. Moreover, the pseudo-Riemannian metric on $\mathbb{S}^{n,n+1}$, obtained as the restriction of the non-degenerate bilinear form on $\mathbb{R}^{n+1,n+1}$, has constant positive sectional curvature and is unique up to isometries, being also induced by the Killing form on the Lie algebra of $\mathrm{SO}_0(n+1,n+1)$ when viewing the pseudosphere as a homogeneous space. It must therefore coincide with the metric arising from the construction of Proposition \ref{prop:para_Sasaki_principal_bundle} (see \cite[Proposition 3.2 and Example 4.4]{loiudice2025sasakian} where the difference in the sign depends on the convention).

\end{proof}

\bibliographystyle{alpha}
\bibliography{biblio}

\begin{thebibliography}{CMMS04}

\bibitem[ABBZ12]{andersson2012cosmological}
Lars Andersson, Thierry Barbot, Fran{\c{c}}ois B{\'e}guin, and Abdelghani Zeghib.
\newblock Cosmological time versus cmc time in spacetimes of constant curvature.
\newblock {\em Asian Journal of Mathematics}, 16(1):37--88, 2012.

\bibitem[Anc11]{anciaux2011minimal}
Henri Anciaux.
\newblock {\em Minimal submanifolds in pseudo-Riemannian geometry}.
\newblock World Scientific, 2011.

\bibitem[Bar15]{barbot2015deformations}
Thierry Barbot.
\newblock Deformations of fuchsian ads representations are quasi-fuchsian.
\newblock {\em Journal of Differential Geometry}, 101(1):1--46, 2015.

\bibitem[Ben01]{benoist2001convexes}
Yves Benoist.
\newblock Convexes divisibles {I}.
\newblock {\em Comptes Rendus de l'Acad{\'e}mie des Sciences-Series I-Mathematics}, 332(5):387--390, 2001.

\bibitem[Ben03]{benoist2003convexes}
Yves Benoist.
\newblock Convexes divisibles {II}.
\newblock {\em Duke Mathematical Journal}, 120(1), 2003.

\bibitem[Ben05]{benoist2005convexes}
Yves Benoist.
\newblock Convexes divisibles {III}.
\newblock {\em Annales scientifiques de l’ecole normale sup{\'e}rieure}, 38(5):793--832, 2005.

\bibitem[Ben06]{benoist2006convexes}
Yves Benoist.
\newblock Convexes divisibles {IV}.
\newblock {\em Invent. math}, 164(2):249--278, 2006.

\bibitem[Bes07]{besse2007einstein}
Arthur~L Besse.
\newblock {\em Einstein manifolds}.
\newblock Springer, 2007.

\bibitem[BH13]{benoist2013cubic}
Yves Benoist and Dominique Hulin.
\newblock Cubic differentials and finite volume convex projective surfaces.
\newblock {\em Geometry \& Topology}, 17(1):595--620, 2013.

\bibitem[BK23]{beyrer2023mathbb}
Jonas Beyrer and Fanny Kassel.
\newblock $\mathbb{H}^{p,q}$-convex cocompactness and higher higher {T}eichm\"uller spaces.
\newblock {\em arXiv:2305.15031 (to appear in: GAFA)}, 2023.

\bibitem[BS10]{bonsante2010maximal}
Francesco Bonsante and Jean-Marc Schlenker.
\newblock Maximal surfaces and the universal teichm{\"u}ller space.
\newblock {\em Inventiones mathematicae}, 182(2):279--333, 2010.

\bibitem[BS20]{bonsante2020anti}
Francesco Bonsante and Andrea Seppi.
\newblock Anti-de sitter geometry and {T}eichm{\"u}ller theory.
\newblock {\em In the tradition of Thurston: Geometry and topology}, pages 545--643, 2020.

\bibitem[Cal72]{calabi1972complete}
Eugenio Calabi.
\newblock Complete affine hyperspheres {I}.
\newblock In {\em Instituto Nazionale di Alta Matematica Symposia Mathematica}, volume~10, pages 19--38, 1972.

\bibitem[Can21]{canary2021anosov}
Richard Canary.
\newblock Anosov representations: Informal lecture notes, 2021.

\bibitem[CLT15]{cooper2015convex}
Daryl Cooper, DD~Long, and Stephan Tillmann.
\newblock On convex projective manifolds and cusps.
\newblock {\em Advances in Mathematics}, 277:181--251, 2015.

\bibitem[CMMS04]{cortes2004special}
Vicente Cort{\'e}s, Christoph Mayer, Thomas Mohaupt, and Frank Saueressig.
\newblock Special geometry of euclidean supersymmetry 1. {V}ectormultiplets.
\newblock {\em Journal of High Energy Physics}, 2004(03):028, 2004.

\bibitem[Coc49]{cockle1849iii}
James Cockle.
\newblock {I}{I}{I}. {O}n a new imaginary in algebra.
\newblock {\em The London, Edinburgh, and Dublin Philosophical Magazine and Journal of Science}, 34(226):37--47, 1849.

\bibitem[CT24]{collier2023holomorphic}
Brian Collier and J{\'e}r{\'e}my Toulisse.
\newblock Holomorphic curves in the 6-pseudosphere and cyclic surfaces.
\newblock {\em Transactions of the American Mathematical Society}, 377(09):6465--6514, 2024.

\bibitem[CTT19]{CTT}
Brian Collier, Nicolas Tholozan, and J\'{e}r\'{e}my Toulisse.
\newblock The geometry of maximal representations of surface groups into {${\rm SO}_0(2,n)$}.
\newblock {\em Duke Math. J.}, 168(15):2873--2949, 2019.

\bibitem[CY77]{cheng1977regularity}
Shiu-Yuen Cheng and Shing-Tung Yau.
\newblock On the regularity of the {M}onge-{A}mp{\`e}re equation det $(\frac{\partial^2 u}{\partial x_i\partial x_j})= f (x, u)$.
\newblock {\em Communications on Pure and Applied Mathematics}, 30(1):41--68, 1977.

\bibitem[CY86]{cheng1986complete}
Shiu-Yuen Cheng and Shing-Tung Yau.
\newblock Complete affine hypersurfaces. {P}art i. {T}he completeness of affine metrics.
\newblock {\em Communications on Pure and Applied Mathematics}, 39(6):839--866, 1986.

\bibitem[DGK18]{danciger2018convex}
Jeffrey Danciger, Fran{\c{c}}ois Gu{\'e}ritaud, and Fanny Kassel.
\newblock Convex cocompactness in pseudo-{R}iemannian hyperbolic spaces.
\newblock {\em Geometriae Dedicata}, 192(1):87--126, 2018.

\bibitem[DHK24]{dorfmeister2024half}
Josef~F Dorfmeister, Roland Hildebrand, and Shimpei Kobayashi.
\newblock Half-dimensional immersions into the para-complex projective space and {R}uh--{V}ilms type theorems.
\newblock {\em Results in Mathematics}, 79(7):245, 2024.

\bibitem[DMS25]{delarue2025locally}
Benjamin Delarue, Daniel Monclair, and Andrew Sanders.
\newblock Locally homogeneous axiom a flows {II}: geometric structures for {A}nosov subgroups.
\newblock {\em arXiv preprint arXiv:2502.20195}, 2025.

\bibitem[DNV90]{dillen1990conjugate}
Franki Dillen, Katsumi Nomizu, and Luc Vranken.
\newblock Conjugate connections and {R}adon's theorem in affine differential geometry.
\newblock {\em Monatshefte f{\"u}r Mathematik}, 109(3):221--235, 1990.

\bibitem[EES22]{el2022gauss}
Christian El~Emam and Andrea Seppi.
\newblock On the {G}auss map of equivariant immersions in hyperbolic space.
\newblock {\em Journal of topology}, 15(1):238--301, 2022.

\bibitem[EL83]{eells1983selected}
James Eells and Luc Lemaire.
\newblock {\em Selected topics in harmonic maps}, volume~50.
\newblock American Mathematical Soc., 1983.

\bibitem[ES64]{eells1964harmonic}
James Eells and Joseph~H Sampson.
\newblock Harmonic mappings of {R}iemannian manifolds.
\newblock {\em American journal of mathematics}, 86(1):109--160, 1964.

\bibitem[Gig78]{gigena1978integral}
Salvador Gigena.
\newblock Integral invariants of convex cones.
\newblock {\em Journal of Differential Geometry}, 13(2):191--222, 1978.

\bibitem[Gig81]{gigena1981conjecture}
Salvador Gigena.
\newblock On a conjecture by {E}. {C}alabi.
\newblock {\em Geometriae Dedicata}, 11(4):387--396, 1981.

\bibitem[GW12]{guichard2012anosov}
Olivier Guichard and Anna Wienhard.
\newblock Anosov representations: domains of discontinuity and applications.
\newblock {\em Inventiones mathematicae}, 190(2):357--438, 2012.

\bibitem[Hat63]{hatakeyama1963some}
Yoji Hatakeyama.
\newblock Some notes on differentiable manifolds with almost contact structures.
\newblock {\em Tohoku Mathematical Journal, Second Series}, 15(2):176--181, 1963.

\bibitem[Hel79]{helgason1979differential}
Sigurdur Helgason.
\newblock {\em Differential geometry, Lie groups, and symmetric spaces}, volume~80.
\newblock Academic press, 1979.

\bibitem[Hil11a]{hildebrand2011cross}
Roland Hildebrand.
\newblock The cross-ratio manifold: a model of centro-affine geometry.
\newblock {\em International Electronic Journal of Geometry}, 4(2):32--62, 2011.

\bibitem[Hil11b]{hildebrand2011half}
Roland Hildebrand.
\newblock Half-dimensional immersions in para-{K}{\"a}hler manifolds.
\newblock {\em International Electronic Journal of Geometry}, 4(2):85--113, 2011.

\bibitem[Ish79]{ishihara1979anti}
Ikuo Ishihara.
\newblock Anti-invariant submanifolds of a {S}asakian space form.
\newblock {\em Kodai Mathematical Journal}, 2(2):171--186, 1979.

\bibitem[Lab07]{Labourie_cubic}
Fran\c{c}ois Labourie.
\newblock Flat projective structures on surfaces and cubic holomorphic differentials.
\newblock {\em Pure Appl. Math. Q.}, 3(4, Special Issue: In honor of Grigory Margulis. Part 1):1057--1099, 2007.

\bibitem[Lab17]{labourie2017cyclic}
Fran{\c{c}}ois Labourie.
\newblock Cyclic surfaces and {H}itchin components in rank 2.
\newblock {\em Annals of Mathematics}, 185(1):1--58, 2017.

\bibitem[Li90]{li1990calabi}
An-Min Li.
\newblock Calabi conjecture on hyperbolic affine hyperspheres.
\newblock {\em Mathematische Zeitschrift}, 203(1):483--491, 1990.

\bibitem[Li92]{li1992calabi}
An-Min Li.
\newblock Calabi conjecture on hyperbolic affine hyperspheres (2).
\newblock {\em Mathematische Annalen}, 293(1):485--493, 1992.

\bibitem[Loi25]{loiudice2025sasakian}
Eugenia Loiudice.
\newblock Para-sasakian $\phi$-symmetric spaces.
\newblock {\em Annals of Global Analysis and Geometry}, 67(1):4, 2025.

\bibitem[LSZH15]{li2015global}
An-Min Li, Udo Simon, Guosong Zhao, and Zejun Hu.
\newblock {\em Global affine differential geometry of hypersurfaces}.
\newblock de Gruyter, 2015.

\bibitem[LT23]{labourie2023quasicircles}
Fran{\c{c}}ois Labourie and J{\'e}r{\'e}my Toulisse.
\newblock Quasicircles and quasiperiodic surfaces in pseudo-hyperbolic spaces.
\newblock {\em Inventiones mathematicae}, pages 1--88, 2023.

\bibitem[LTW24]{labourie2024plateau}
Fran{\c{c}}ois Labourie, J{\'e}r{\'e}my Toulisse, and Michael Wolf.
\newblock Plateau problems for maximal surfaces in pseudo-hyperbolic space.
\newblock {\em Annales Scientifiques de lEcole Normale Sup{\'e}rieure}, 57(2), 2024.

\bibitem[MS17]{mcduff2017introduction}
Dusa McDuff and Dietmar Salamon.
\newblock {\em Introduction to symplectic topology}, volume~27.
\newblock Oxford University Press, 2017.

\bibitem[MT24]{moriani2024rigidity}
Alex Moriani and Enrico Trebeschi.
\newblock Rigidity for complete maximal spacelike submanifolds in pseudo-hyperbolic space.
\newblock {\em arXiv preprint arXiv:2411.10352}, 2024.

\bibitem[NS94]{nomizu1994affine}
Katsumi Nomizu and Takeshi Sasaki.
\newblock {\em Affine differential geometry: geometry of affine immersions}.
\newblock Cambridge university press, 1994.

\bibitem[NS22]{nie2022regular}
Xin Nie and Andrea Seppi.
\newblock Regular domains and surfaces of constant gaussian curvature in 3-dimensional affine space.
\newblock {\em Analysis \& PDE}, 15(3):643--697, 2022.

\bibitem[RT24]{RT_bicomplex}
Nicholas Rungi and Andrea Tamburelli.
\newblock Complex {L}agrangian minimal surfaces, bi-complex {H}iggs bundles and $\mathrm{SL}(3,\mathbb{C})$-quasi-{F}uchsian representations.
\newblock {\em arXiv:2406.14945}, 2024.

\bibitem[RT25]{rungi2025complex}
Nicholas Rungi and Andrea Tamburelli.
\newblock Para-complex geometry and cyclic {H}iggs bundles.
\newblock {\em arXiv:2503.01615}, 2025.

\bibitem[Sas80]{sasaki1980hyperbolic}
Takeshi Sasaki.
\newblock Hyperbolic affine hyperspheres.
\newblock {\em Nagoya Mathematical Journal}, 77:107--123, 1980.

\bibitem[SS62]{shirokov1962affine}
Petr~Alekseevich Shirokov and Aleksandr~Petrovich Shirokov.
\newblock Affine differentialgeometrie.
\newblock {\em Teubner}, 1962.

\bibitem[SST23]{seppi2023complete}
Andrea Seppi, Graham Smith, and J{\'e}r{\'e}my Toulisse.
\newblock On complete maximal submanifolds in pseudo-hyperbolic space.
\newblock {\em arXiv:2305.15103}, 2023.

\bibitem[Tre19]{trettel2019families}
Steve~J. Trettel.
\newblock {\em Families of geometries, real algebras, and transitions}.
\newblock University of California, Santa Barbara, 2019.

\bibitem[Tre24]{trebeschi2024constant}
Enrico Trebeschi.
\newblock Constant mean curvature hypersurfaces in anti-de sitter space.
\newblock {\em International Mathematics Research Notices}, 2024(9):8026--8066, 2024.

\bibitem[Vra02]{LucVrancken2002}
Luc Vrancken.
\newblock Centroaffine differential geometry and its relations to horizontal submanifolds.
\newblock {\em Banach Center Publications}, 57(1):21--28, 2002.

\bibitem[YK77]{yano1977anti}
Kentaro Yano and Masahiro Kon.
\newblock Anti-invariant submanifolds of {S}asakian space forms {I}.
\newblock {\em Tohoku Mathematical Journal, Second Series}, 29(1):9--23, 1977.

\bibitem[Zam09]{zamkovoy2009canonical}
Simeon Zamkovoy.
\newblock Canonical connections on paracontact manifolds.
\newblock {\em Annals of Global Analysis and Geometry}, 36(1):37--60, 2009.

\end{thebibliography}

\end{document}